\documentclass[10pt,a4paper]{article}
\usepackage[utf8]{inputenc}
\usepackage{custom}

\def\A{{\mathcal {A}}}
\def\P{{\mathcal {P}}}

\def\S{{\mathcal {S}}}
\DeclareMathSymbol{\Z}{\mathbin}{AMSb}{"5A}

\begin{document}
\title{Ergodic properties of Kantorovich operators}

\author{
Malcolm Bowles\thanks{Department of Mathematics, Rutgers University, New Brunswick, NJ, USA}   \quad   and \quad 
Nassif  Ghoussoub\thanks{Department of Mathematics,  The University of British Columbia, Vancouver BC Canada V6T 1Z2.
Partially supported by a grant from the Natural Sciences and Engineering Research Council of Canada.} 
}

\maketitle

\begin{abstract}  \textit{Kantorovich operators} are non-linear extensions of \textit{Markov operators} and are omnipresent in several branches of mathematical analysis. The asymptotic behaviour of their iterates plays an important role even in classical ergodic, potential and probability theories, which are normally concerned with linear Markovian operators, semi-groups, and resolvents.
The Kantorovich operators that appear implicitly in these cases, though non-linear, are all positively $1$-homogenous. General Kantorovich operators amount to assigning ``a cost" to most operations on measures and functions normally conducted ``for free" in these classical settings. Motivated by extensions of the Monge-Kantorovich duality in mass transport, the stochastic counterpart of  Aubry-Mather theory for Lagrangian systems,  weak KAM theory \`a la Fathi-Mather, and ergodic optimization of dynamical systems, we study the asymptotic properties of general Kantorovich operators.

  \end{abstract}
  
  \tableofcontents

\section{Introduction}  
The introduction in \cite{BG2} of Kantorovich operators, which are important non-linear extensions of \textit{Markov operators} and which are ubiquitous in several aspects of mathematical analysis, was motivated by the duality established by Kantorovich in the context of the Monge mass transport problem, hence their name. The fact that their iterates, when they originate from optimal mass transports, play an important role in the dynamics of Hamilton-Jacobi equations developed by Mather, Aubry, Fathi and others, was first noted by Bernard and Buffoni \cite{BB1, BB2}. Our motivation was to extend this approach to the stochastic counterparts of Fathi-Mather weak KAM theory and of ergodic optimization in stochastic symbolic dynamics which do not correspond to to classical mass transports but to the more encompassing notion of {\it linear transfers} introduced in \cite{BG1, BG2}. In retrospect, we realized that Kantorovich operators -and their iterates- already appeared in several classical branches of analysis, including ergodic, probability, and potential theories, which are normally concerned with  the analysis of linear Markov operators, semi-groups, and resolvents. 

First, recall that a {\bf Markov operator} is a positive linear continuous map $T: C(Y) \to C(X)$ such that $T1=1$, where $C(K)$ is the space of continuous functions on a compact space $K$. To define Kantorovich operators, we consider $USC(\Omega)$ (resp., $LSC(\Omega)$ to be the space of proper, bounded above, upper semi-continuous), (resp., proper, bounded below, lower semi-continuous) functions on $\Omega$ that are valued in $\R\cup \{+\infty\}\cup \{-\infty\}$. By {\bf proper}, we mean here a function $f$ such that there exists at least one element $\omega\in \Omega$ such that $-\infty <f(\omega)<+\infty$. 

\begin{definition} {\bf A backward 
Kantorovich operator} is a map $T^-: C(Y) \to USC(X)$, which satisfies the following properties:  It is

  \begin{enumerate}
 \item 
  {\it monotone increasing}, i.e., if $g_1, g_2 \in C(Y)$ and $g_1\leq g_2$,   
  then $T^-g_1\leq T^-g_2$. 
 \item 
 {\it affine on the constants}, i.e., for any $c\in \R$ and $g\in C(Y)$, 
$T^-(g+c)=T^-g +c$.
 \item  {\it convex}, i.e., 
 for $\lambda \in [0, 1]$, $g_1, g_2 \in C(Y)$,
$T^-(\lambda g_1+(1-\lambda)g_2)\leq \lambda T^-g_1+(1-\lambda)T^-g_2.$
\item {\it lower semi-continuous} in the following sense,  
if $g_n \to g$ in $C(Y)$,  
 then $T^-g\leq \liminf\limits_{n \to \infty}T^-g_n $.
 \end{enumerate}
{\bf A forward Kantorovich operator} is a map $T^+: C(X) \to LSC(Y)$ that satisfies 1), 2), as well as:\\

3') {\rm concave},  i.e., 
 for any $\lambda \in [0, 1]$, $f_1, f_2 \in C(X)$,
$T^+(\lambda f_1+(1-\lambda)f_2)\geq \lambda T^+f_1+(1-\lambda)T^+f_2.
$\\

4') {\it upper-semi-continuous}, i.e., if $f_n \to f$ in $C(X)$,  
 then $T^+f\geq \liminf\limits_{n \to \infty}T^+f_n $.
\end{definition}
 It is clear that Markov operators are both backward and forward Kantorovich operators. A bit surprising is that non-linear Kantorovich operators and their iterates appeared often, even in studies of linear Markov operators.
For example, if $S: C(X)\to C(X)$ is a Markov operator, then 
  \eq{T^-g=g\vee Sg  \quad \hbox{(resp.,\,\, $T^+f=f\wedge Sf$}),}
is a backward (resp., forward) Kantorovich operator, and their iterates
   converge to another important idempotent backward (resp., forward) Kantorovich operator: The {\it r\'eduite}  
 \eq{T^-_\infty g=\hat g, \quad \hbox{ resp.,\,\, $T_\infty^+f=\check{f}$},}
  where $\hat h$ (resp., $\check{h}$) is the least $T$-superharmonic function above (resp., the greatest $T$-subharmonic function below) $g$.
  
  The {\em filling scheme} associated with $S$,  
\eq{T^-g=Sg^+-g^- \quad \hbox{(resp.,\,\, $T^+f=f^+-Sf^-$}),}
is a backward (resp., forward) Kantorovich operator, whose iterates play a crucial role in the proof of the celebrated Chacon-Ornstein theorem \cite{K}, and in the construction of solutions for Poisson equations associated with $S$ \cite{Meyer}.

    Kantorovich operators and their iterates appear in {\it potential theory}. For example,  iterates of the operator  
 \eq{
T^- g(x) =\sup_{r \geq 0} \bigg\{ \int_{B} g (x + r y) \,dm(y);\,  x + r \overline{B}\} \subset O \bigg\},
}
where $B$ is the open unit ball in $\R^n$ centered at $0$, and $m$ is normalized Lebesgue measure on $\R^n$, lead to the {\it superharmonic envelope via optimal stopping} \cite{GKL2},	
 \eq{T_\infty^-g(x)
		:= \sup\Big\{{\mathbb E}^{x}\Big[g(B_\tau)\Big]; \, \tau \ge 0 \,\, {\rm stopping\, time}, {\mathbb E}^{x}[\tau]<+\infty
		\Big\}.
}
	The expectation $\mathbb{E}^{x}$  refers
	to Brownian motions $(B_t)_t$ starting at $x$.

Similarly, in the theory of {\it several complex variables},  if $X$ is the closure of an open bounded domain $O$ in $\C^n$, then  the operator: 
 \eq{
T^- g(x):= \sup_{v\in \R^n} \bigg\{ \int ^{2\pi}_0 g(x +
 e^{i\theta}v) {d\theta \over 2\pi};\,  x + \bar \Delta v\subset O \bigg\}, 
 }
where $\Delta = \{ z \in {\mathbb C}, \vert z\vert < 1\}$ is the open unit disc in $\C$,  is a backward Kantorovich operator, and its iterates lead to the idempotent backward Kantorovich operator giving the  {\it plurisuperharmonic envelopes} of functions on $O$ \cite{GM}, 
 \eq{
T_\infty^-g(x)=\sup \left\{ \int^{2\pi}_0 g (P(e^{i\theta})) {d\theta
\over 2\pi};\ \hbox{\rm $P$ polynomial, $P(\bar \Delta)\subset O$.   
 $P(0) = x$}\right\}
 }

So far, all examples mentioned above are {\bf positively $1$-homogenous} Kantorovich operators, i.e., those they satisfy 
 \begin{equation}
 \hbox{$T(\lambda f)=\lambda Tf$ for all $\lambda \geq 0$.}
 \end{equation} 
 
 They all can be characterized as the maximal expected gain 
 \eq{\lbl{max.gain}
 T^-g(x):=\sup\{\int_X g\, d\sigma; (\delta_x, \sigma)\in \S\},}
 associated to a reward function $f$, where  $\S \subset \P(X)\times \P(Y)$ is a ``gambling house", i.e.,  the section $\S_x=\{\sigma \in \P(X); (\delta_x, \sigma)\in \S\}$, which is the collection of distributions of gains available to a gambler having wealth $x$, is non-empty \cite{DM}. Moreover,  the gambling house is governed by rules described in terms of {\it a balayage order}. Indeed, the following holds. 
 
 \begin{theorem} \cite{G} Let $T$ be a map from $C(Y)$ to $USC(X)$, where $X$ and $Y$ are two compact metric spaces. The following are then equivalent: 
\begin{enumerate}
\item $T$ is a $1$-homogenous backward Kantorovich operator. 

\item There exists a closed convex, max-stable balayage cone $\A$ of lower semi-continuous functions on the disjoint union  $X \bigsqcup Y$ such that 
 \begin{equation}
  Tg(x):=\sup\{\int_X g\, d\sigma;\,  \delta_x  \prec_\A \sigma \}, 
   \end{equation}
where for $(\mu, \nu) \in \P(X)\times \P(Y)$, 
\[
\hbox{$\mu \prec_\A \nu$\quad iff \quad $\int_X\phi d\mu \leq \int_Y\phi d\nu$ for all $\phi \in \A$.} \quad {\rm (Restricted \,\, balayage)}.
\]
\end{enumerate}
The gambling house is then, 
 \begin{equation}
\S=\{(\mu, \nu) \in \P(X)\times \P(Y);\, \mu \prec_\A \nu \}=\{(\mu, \nu) \in \P(X)\times \P(Y);\, \nu \leq T_\#\mu  \},
   \end{equation}
where $T_\#\mu (g)=\int_XTg \, d\mu$ for every $g\in C(Y)$. 
\end{theorem}

 The simplest non-homogenous (yet affine) Kantorovich operator appears in ergodic optimization of symbolic dynamics  \cite{Ga, GL}. It has the form 
  \eq{
T^-g(x) := g\circ \sigma(x) - A(x),
}
where $A$ is a given lower semi-continuous potential $A$ and $\sigma$ is a point transformation. Its iterates lead to minimizing the action $\mu \mapsto \int_{X}Ad\mu$ among all $\sigma$-invariant measures $\mu$.

As mentioned before, our terminology originates in the theory of optimal mass transport involving  a lower semi-continuous cost function $c(x, y)$ on a product space $X\times Y$. The operator 
\eq{
T^-g(x)=\sup_{y\in Y} \{g(y)-c(x, y)\} \quad \hbox{ (resp.,\,\, $T^+f(y)=\inf\limits_{x\in X} \{f(x)+c(x, y)\}$)},
}
is then a backward (resp., forward Kantorovich operator)  \cite{V}, which is not necessarily positively $1$-homogenous. For example one that appears in Brenier's work \cite{B} is 
\eq{
T^-g=-g^* \quad \hbox{(resp., $T^+f=(-f)^*$)},
}
where $\phi^*$ is the Legendre transform of a function $\phi$. 

Another related example is the Kantorovich operator 
\eq{
T^-_\nu g(x)=\epsilon \log \int_{Y}e^{\frac{g(y)-c(x, y)}{\epsilon}}d\nu (y),
} 
where $\nu \in \P(X)$, which arises from an entropic regularization of the above optimal mass transport with cost $c$, as well as the composition $T^-_\nu\circ T^-_\mu$, where $\mu$ is another probability in $\P(X)$, whose iterates are the building clocks of the Sinkhorn algorithm \cite{Sa}. 

Another important example arises in the study Hamiltonian dynamics \cite{BB1}. Indeed, if  $L: TM \to \R\cup\{+\infty\}$ is a given \textit{Tonelli Lagrangian} on a given smooth compact Riemannian manifold $M$ without boundary, with tangent bundle $TM$, and $H$ is the associated Hamiltonian, then the operator 
\begin{equation}\label{value.2}
T^-g(x) :=\sup\Big\{g(\gamma (1))-\int_0^1L(\gamma (s), {\dot \gamma}(s))\, ds; \gamma \in C^1([0, 1), M);   \gamma(0)=x\Big\},
\end{equation}
which associates to a state $g$ at time $1$, the initial state of the  viscosity solution for the backward  Hamilton-Jacobi equation,  
 \eq{\label{HJ.0} 
\left\{ \begin{array}{lll}
\partial_tV+H(x, \nabla_xV)&=&0 \,\, \text{ on }\,\,  (0, 1)\times M\\
\hfill V(1, x)&=&g(x),
\end{array}  \right.
}
is a backward Kantorovich operator, whose iterates lead to Fathi's weak KAM solutions \cite{F} and Mather's minimal measures \cite{Mat}. Similarly, 
\eq{
T^+ f(y) := \inf\{ f(\gamma(0)) + \int_{0}^{1}L(\gamma(s), \dot{\gamma}(s))d s\,;\, \gamma \in C^1([0,1]; M), \gamma(1) = y\},
}
is a forward Kantorovich operator, which associates to an initial state $f$, the state  at time $1$ of the viscosity solution for the associated forward Hamilton-Jacobi equation.

Non-homogenous Kantorovich operators can be characterized as a gambling house that charges fees, i.e., a gambler with wealth $x$, incurs a cost $c(x, \sigma)$ each time they choose a distribution of gains $\sigma$. This is an important characterization that allows for their extension beyond continuous functions as {\it functional Choquet capacities}. 

\begin{theorem} \cite{G}  Let $T$ be a map from $C(Y)$ to $USC(X)$. The following are then equivalent:
\begin{enumerate}

\item $T$ is  a backward Kantorovich operator from $C(Y)$ to $USC(X)$. 

\item There exists a proper lower semi-continuous cost functional $c: X\times {\mathcal P}(Y)\to  \R \cup\{+\infty\}$ with $\sigma \mapsto c(x, \sigma)$ convex for each $x\in X$, and a balayage cone $\A \in LSC(X\sqcup Y)$ such that
 \eq{
Tg(x) := \sup\{ \int_{Y}gd\sigma - c(x,\sigma)\,;\, \sigma \in {\cal P}(Y)\,{\rm and}\,\, \delta_x\prec_\A \sigma \}. 
}
\end{enumerate}
\end{theorem}

We note that in all examples mentioned above of backward Kantorovich operators, there is always a corresponding forward Kantorovich operator, mostly related through the formula 
\eq{
T^+f=-T^-(-f), 
}   
because, as we shall see later, they correspond to ``symmetric transfers of mass distributions". In general, Kantorovich operators are one-sided. 
For example, in stochastic mass transfer problems \cite{Go, MT}, one considers 
the cost function 
\begin{align}\label{Primal with fixed end time}
c (x,\sigma) :=\inf_{\beta} \Big\{\mathbb{E}\Big[\int_0^1 L\big(t,X_t,\beta_t\big)dt\Big];\ dX_t=\beta_t\, dt+dW_t,\ X_0=x,\ X_1\sim \sigma\Big\},
\end{align}
where $W_t$ is Weiner measure and the minimization is taken over all suitable drifts $\beta$.  
Under some assumptions on the Lagrangian $L$, the associated Kantorovich operator is then given by 
$T^-g=J_g(0, x)$
  where $J_g$ is the initial state of the backward second order Hamilton-Jacobi equation 
\begin{equation} \label{PDE with fixed end time}
\begin{cases}
\partial_t J(t,x)  + \frac{1}{2} \Delta J(t,x) + H\big(x,\nabla J(t,x)\big)&= 0 \  \mbox{in\ }\, (0,1) \times \mathbb{R}^d,\\
\hfill J(1,x)&= g(x) \  \mbox{on\ }\,\mathbb{R}^d,
\end{cases}
\end{equation}
where $H$ is the Hamiltonian associated to $L$. 
We shall see  how iterates of this operator lead to a stochastic counterpart of Mather-Fathi theory.

Our main objective in this paper is to study the ergodic properties of general Kantorovich operators. We consider the corresponding {\em Mather constant}, 
 \eq{\lbl{Mane}
c(T) :=  \inf_{\mu \in \P(X)}\sup_{h\in C(X)}\{\int_X(h-T^-h)\, d\mu,
}
and give conditions that guarantees the existence of an upper semi-continuous function $h$, such that 
\eq{
T^- h+c(T)=h,
}
We call such a function $h$, a {\em backward weak KAM solution}. Our terminology is in reference to the work of Mather \cite{Mat}, Aubry \cite{Au}, Man\'e  \cite{Man}, and Fathi \cite{F} on Hamiltonian dynamics and the associated Hamilton-Jacobi equations. Our premise is that there are ``several weak KAM theories"  out there to which this general theory is applicable. For example, the stochastic counterpart of  Hamiltonian dynamics, but also discrete versions as in the theory of ergodic minimization in both a deterministic and stochastic setting. 

In section 2, we give a quick review of the basic properties of Kantorovich operators that will be used in the paper. We exhibit the one-to-one correspondence between linear transfers, optimal balayage transport and Kantorovich operators. We also recall that they are Choquet functional capacities, which will be crucial to their limiting properties. 

Section 3 introduces the {\em Mather constant} and the {\em Man\'e constant} associated to a Kantorovich operator, as well as the first results on the existence of weak KAM solutions. 

In Section 4, we construct {\em weak KAM operators} associated to a Kantorovich operator. These are idempotent  Kantorovich operators $T_\infty$ that commute with $T$ and which map $C(X)$ onto the cone of weak KAM solutions in $USC(X)$. 

We give in Section 5 various applications to classical Mather-Fathi theory and its stochastic counterpart by considering semi-groups of Kantorovich operators, while Section 6 -developed jointly with Dorian Martino \cite{Mar}- deals with applications to deterministic and stochastic ergodic minimization.   

\section{Basic properties of Kantorovich operators} 

 We collect in this section the properties of Kantorovich operators that will be used in the rest of the paper. Detailed proofs can be found in \cite{G}. We shall focus here on probability measures on compact metric spaces, even though the right settings for most applications and examples are complete metric spaces, Riemannian manifolds, or at least $\R^n$. This will allow us to avoid the usual functional analytic complications, and concentrate on the conceptual aspects of the theory. The simple compact case will at least point to results that can be expected to hold and be proved --albeit with additional analysis and suitable hypothesis -- in more general situations. 

The rich and flexible structure of Kantorovich operators stems from their duality -via Legendre transform- with certain weak$^*$-lower semi-continuous and convex functionals $\T:\mcal{M}(X) \times \mcal{M}(Y) \to \R\cup \{+\infty\}$, where ${\cal M}(X)$ is the class of signed Radon measures on $X$. The domain of such a functional  will be denoted 
 \eq{D(\T):=\{(\mu, \nu) \in {\mathcal M}(X)\times  {\mathcal M}(Y);  \T (\mu, \nu)<+\infty\}.
}
 We shall consider for each $\mu \in {\mathcal P}(X)$ (resp.,  $\nu \in {\mathcal P}(Y)$), the partial maps ${\mathcal T}_\mu$ on ${\mathcal P}(Y)$ (resp., ${\mathcal T}_\nu$ on ${\mathcal P}(X)$) given by $\nu \to {\mathcal T} (\mu, \nu)$ (resp., $\mu \to {\mathcal T} (\mu, \nu)$). 

\defn{\lbl{lineartransfers} Recall that a functional $\T:\mcal{M}(X) \times \mcal{M}(Y) \to \R\cup \{+\infty\}$ 
 is said to be a \textbf{backward linear transfer} (resp., \textbf{forward linear transfer}) if 
\begin{enumerate}
\item $\T:\mcal{M}(X) \times \mcal{M}(Y) \to \R\cup \{+\infty\}$ is a proper, convex, bounded below, and weak$^*$ lower semi-continuous functional.
\item $D(\T) \subset \P(X)\times \P(Y)$.
\item There exists a map $T^-$ (resp., $T^+$) from $C(Y)$ (resp., $C(X)$) into the space of 
Borel-measurable functions on $X$ (resp., on $Y$) such that for all $\mu \in \P(X)$  
and $g \in C(Y)$ (resp., $\nu\in \P(Y)$  and $f \in C(X)$), 
\eq{\lbl{Legendre_trans}
\T_\mu^*(g) = \int_{X}T^- g d\mu,\quad {\rm resp.},\quad  \T_\nu^*(f) = -\int_{Y}T^+ (-f) d\nu,
}
where $\T_\mu^*$ (resp., $\T_\nu^*$) is the Fenchel-Legendre transform of $\T_\mu$ (resp., $\T_\nu$) for the duality of $C(X)$ (resp., $C(Y)$) and ${\mathcal M}(X)$ (resp., and ${\mathcal M}(Y))$. 
\end{enumerate}
}

Note that since $D(\T) \subset \P(X)\times \P(Y)$, the Legendre transforms are simply
\begin{equation}
\T_\mu^*(g) := \sup_{\sigma \in \mathcal{P}(Y)}\{\int_{Y}g d\sigma - \T(\mu, \sigma)\}\quad {\rm resp.,}\quad \T_\nu^*(f) := \sup_{\sigma \in \mathcal{P}(X)}\{\int_{X}f d\sigma - \T(\sigma, \nu)\}.
\end{equation}
\noindent Since $T^-$ and $T^+$ arise from a Legendre transform, they satisfy various additional properties, which are essentially those satisfied by Kantorovich operators. 

From now on, we shall restrict our analysis to backward transfers since if $\T$ is a forward linear transfer with $T^+$ as a forward Kantorovich operator, then $\tilde\T(\mu, \nu)=\T(\nu, \mu)$ is a backward linear transfer 
with $T^-f=-T^+(-f)$ as a backward Kantorovich operator. We shall then do without the sign in the exponent when there is no confusion. However, we shall bring back the sign if we are dealing with operators that are both backward and forward.

In some cases, both will be needed (if they co-exist).

 \begin{theorem} \cite{G} Let $X$ and $Y$ be two compact metric spaces.The following are then equivalent:
\begin{enumerate}

\item $T$ is a backward Kantorovich operator from $C(Y)$ to $USC(X)$. 

\item There exists a backward linear transfer $\T :\mcal{P}(X) \times \mcal{P}(Y) \to \R\cup \{+\infty\}$ such that for all $\mu \in \P(X)$ and $g \in C(Y)$, \eq{\lbl{Legendre_trans5}
\T_\mu^*(g) = \int_{X}T g\,  d\mu. 
}
\item There exists a  proper, bounded below, lower semi-continuous cost functional $c: X\times {\mathcal P}(Y)\to  \R \cup\{+\infty\}$ with $\sigma \mapsto c(x, \sigma)$ convex for each $x\in X$, such that 
\begin{equation}\lbl{costly}
Tg(x) := \sup\{ \int_{Y}gd\sigma - c(x,\sigma)\,;\, \sigma \in \mcal{P}(Y)
\}.  
\end{equation}
\end{enumerate}

\end{theorem}
The correspondence between these three notions goes as follows. If $T$ is a backward Kantorovich operator, then it defines a standard backward $\T$  linear transfer via the formula:  
\begin{equation}\label{duality}
\T(\mu,\nu) = \begin{cases}
\sup\limits_{g \in C(Y)}\{\int_{Y}gd\nu - \int_{X}Tgd\mu\}\quad&\hbox{for all $(\mu, \nu)\in \P(X)\times \P(Y)$}\\
+\infty & \text{\rm otherwise.}
\end{cases}
\end{equation}
The cost $c$ is then simply  $c(x, \sigma)=\T(\delta_x, \sigma)$. 

$T$ is positively $1$-homogenous backward Kantorovich operator if and only if its corresponding linear transfer $\T$ is given by a closed convex gambling house $\S\in \P(X)\times \P(Y)$, that is 
 \begin{equation}
\T(\mu, \nu)=\left\{ \begin{array}{llll}
0 \quad &\hbox{if $(\mu, \nu)\in \S$}\\
+\infty \quad &\hbox{\rm otherwise}
\end{array} \right.
\end{equation} 
In which case, 
$Tg(x)=\sup\{\int_Y g d\sigma; (\delta_x, \sigma) \in \S\}=\sup\{\int_Y g d\sigma; \sigma \in \P(Y), \sigma \leq T\}.$

\begin{proposition} Consider the class ${\mathcal K}(Y, X)$ of backward Kantorovich operators from $C(Y)$ to $USC(X)$. Then, 
\begin{enumerate}

\item If $T_1$ and $T_2$ are in ${\mathcal K}(Y, X)$, and $\lambda \in [0, 1]$, then $\lambda T_1+(1-\lambda)T_2 \in {\mathcal K}(Y, X)$. 

\item If $\lambda \in \R^+$ and $T\in {\mathcal K}(Y, X)$, then the map $(\lambda \cdot T)f:=\frac{1}{\lambda}T(\lambda f)$ for any $f\in C(Y)$, belongs to ${\mathcal K}(Y, X)$.

\item If $T_1$ and $T_2$ are in ${\mathcal K}(Y, X)$, then the map $T_1\star T_2$ defined for $f\in C(Y)$ and $x\in X$ by 
\[
(T_1\star T_2)f(x):=\sup_{\sigma \in {\mathcal P}(Y)}\inf_{g, h\in C(Y)}\int_Y(f-g-h)\, d \sigma +T_1g(x) +T_2h(x) 
\]
is in ${\mathcal K}(Y, X)$

\item  If $T_1$ and $T_2$ are in ${\mathcal K}(Y, X)$, then the map $(T_1\vee T_2)f(x):=T_1f(x)\vee T_2f(x)$ for any $f\in C(Y)$, and $x\in X$ belongs to ${\mathcal K}(Y, X)$.  

\item If $(T_i)_{i\in I}$ is a family  in ${\mathcal K}(Y, X)$, such that for every $f\in C(Y)$, the function $\sup_{i\in I}T_if$ is bounded above on $X$, then the map $
f\to \overline{\sup_{i\in I}T_if}$ 
is a Kantorovich operator, where here $\overline g (x)=\inf\{h(x); h\in C(X), h\geq g \, {\rm on}\, X\}$ is the smallest upper semi-continuous function above $g$. 
\end{enumerate}
\end{proposition}

Denote by $USC_f(X)$ (resp., $USC_b(X)$) the cone of functions in $USC(X)$ that are finite (resp., bounded below).
 
\begin{proposition}  Say that a backward Kantorovich operator $T$ (or its associated linear transfer $\T$) is {\bf standard} (resp., {\bf regular}) if it satisfies any one of these equivalent conditions.
\begin{enumerate}
\item $T$ maps $C(Y)$ to $USC_f(X)$ (resp., $USC_b(X)$).

\item The function $T(0)$ is finite (resp., bounded below).

\item The corresponding linear transfer $\T$ satisfies $\inf_{\sigma \in \P(Y)}\T(x, \sigma)<+\infty$ for all $x\in X$, 

\eq{
\hbox{(resp.,}\quad  d:=\sup_{x\in X}\inf_{\sigma \in \P(Y)}\T(x, \sigma)<+\infty. )
}
\end{enumerate}
Moreover, if  $T$ is standard then it is a contraction from $C(Y)$ to $USC_f(X)$, that is 
\eq{\lbl{lip1}
\sup_{x\in X}|Tg(x)-Th(x)|\leq  \|g - h\|_\infty.
}
A regular operator is clearly standard and $T+d$ then maps $C_+(Y)$ to $USC_+(X)$ for each $x\in X$. 
\end{proposition}

  Denote by $F(X)$ (resp., $F_+(X)$) the class of real-valued (resp., non-negative) functions on $X$. We also consider the class $USC_\sigma(X)$, which is the closure of $USC(X)$ with respect to monotone increasing limits.
 
 \defn{{\bf A Choquet functional capacity} is a map $T: F_+(Y)\to F_+(X)$ such that
\begin{enumerate}
\item $T$ is monotone, i.e., $f\leq g \Rightarrow Tf\leq Tg$. 

\item $T$ maps $USC(Y)$ to $USC(X)$ in such a way that if $g_n, g \in USC(Y)$ and $g_n \downarrow g$ on $Y$, then $T g_n \downarrow Tg$ on $X$. 

\item If $g_n, g \in F_+(Y)$  with $g_n \uparrow g$ on $Y$, then $T g_n \uparrow T g$ on $X$.
\end{enumerate}
}

 \begin{theorem}\lbl{capacity}Let $\T$ be a backward linear transfer  and let $T:C(Y) \to USC(X)$ be the associated backward Kantorovich operator. 
 Then
 \begin{enumerate}
 \item $T$ can be extended to a map from $F(Y)$ to $F(X)$ via the formula  
 \begin{equation}\label{express.1}
T g (x)=\sup\{\int^*_Y g d\nu -{\mathcal T}(\delta_x, \nu); \, \nu\in {\mathcal P}(Y), (\delta_x, \nu)\in D(\T) \},
\end{equation}
where $\int^*_Y g d\nu$ is the outer integral of $g$ with respect to $\nu$. The extension -also denoted by $T$ maps bounded above  functions on $Y$ to  bounded above functions on $X$. It is monotone and satisfies $Tg+c=T(g+c)$ for every $g\in F(Y)$ and $c\in \R$. 

   \item For any upper semi-continuous functions $g$ on $Y$,  we have
 \eq{\lbl{rep1}
 T g (x) := \inf\{ T h(x)\,;\, h \in C(Y), \, h \geq g \}, 
 }
 and if $g_n, g \in USC(Y)$ are such that $g_n \downarrow g$ on $Y$, then $T g_n \downarrow Tg$ on $X$. 

   \item If $T$ is standard (resp., regular), then it maps $USC_b (Y)$ 
   to $USC_f (X)$ (resp., $USC_b(Y)$). 
  \item If $T$ is regular, then the map $T+k$ is a functional capacity that maps $F_+(Y)$ to $F_+(X)$, and consequently, if $g$ is a $K$-analytic function in $F_+(Y)$, then
 \eq{\lbl{rep2}
Tg(x) := \sup\{ T h(x)\,;\, h \in USC(Y)\,,  h \leq g\}.
}
 \item For any $(\mu, \nu)\in {\mathcal P}(X)\times {\mathcal P}(Y)$, we have
\begin{align*}
 {\mathcal T}(\mu, \nu) =\sup\big\{\int_{Y}g\, d\nu-\int_{X}{T}g\, d\mu;\,  g \in USC_b(Y)\big\}.
\end{align*}
\item The Legendre transform formula \refn{Legendre_trans} for $\T_\mu$ extends from $C(Y)$ to $USC(Y)$, that is, for $\mu \in \mcal{P}(X)$,  we have for any $g\in USC(Y)$,
\begin{equation}\label{extLT}
{\mathcal T}^*_\mu (g):=\sup\{\int_Yg d\sigma -\T(\mu, \sigma); \sigma \in {\mathcal P}(Y)\}=
 \int_XT ^-g \, d\mu.
\end{equation}
\end{enumerate} 
\end{theorem}
Now that regular Kantorovich operators can be  extended to maps from $USC_b(Y)$ into $USC_b(X)$, we can compose them in the following way.

 \prop{\label{conv} Let $X_1,...., X_n$ be $n$ compact spaces, and suppose for each $i=1,..., n$,  ${\mathcal T}_i$ is a regular  backward linear transfer on ${\mathcal P}( X_{i-1})\times {\mathcal P}(X_i)$ with  Kantorovich operator  $T _i:C(X_{i})\to USC_f(X_{i-1})$.
 For any probability measures $\mu$ on $X_1$ (resp., $\nu$ on $X_n$), define
\begin{equation}
\T_1\star\T_2...\star\T_n(\mu, \nu)=\inf\{{\mathcal T}_{1}(\mu, \sigma_1) + {\mathcal T}_{2}(\sigma_1, \sigma_2) ...+{\mathcal T}_{n}(\sigma_{n-1}, \nu);\, \sigma_i \in {\mathcal P}(X_i), i=1,..., n-1\}.
\end{equation}
Then,  ${\mathcal T}:=\T_1\star\T_2...\star\T_n$ is a linear backward transfer with a Kantorovich operator given by 
 \eq{    
T=T _1\circ T_2\circ...\circ T_{n}.
 }
In other words, the following duality formula holds:
\begin{equation}
{\mathcal T}(\mu, \nu)=\sup\big\{\int_{X_n}g(y)\, d\nu(y)-\int_{X_1}T_1\circ T_2\circ...\circ T_{n}g(x);\,  g\in C(X_n) \big\}.
\end{equation}
We shall denote by $\T_n$ the linear transfer $\T\star\T\star...\star\T$ by iterating $\T$ $n$-times. The corresponding Kantorovich operator is then $T^n=T\circ T\circ ...\circ T$. 
}
Full proofs can be found in \cite{G}. 

\section{The Mather constant and weak KAM solutions for a Kantorovich operator}

\defn{{\bf The Mather constant} of a backward Kantorovich operator $T: C(X)\to USC(X)$ is the -possibly infinite- scalar,
 \eq{\lbl{stc}
c(T) :=  \inf_{\mu \in \mcal{P}(X)}\sup_{h\in C(X)}\{\int_X(h-Th)\, d\mu.
}
}
If $\T:\mcal{P}(X)\times\mcal{P}(X)\to \R\cup\{+\infty\}$ is the backward linear transfer associated to $T$, then one can readily see that
\eq{\lbl{stc1}
c(T)=c(\T) := \inf_{\mu \in \mcal{P}(X)}\T(\mu,\mu).
}
Also note that if $c(T)$ is finite, then since $\T$ is weak$^*$-lower semi-continuous, there exists $\bar \mu \in \P(X)$ such that 
\[
\T({\bar \mu}, {\bar \mu})=\sup_{h\in C(X)}\{\int_X(h-Th)\, d{\bar \mu}=c(\T).
\]
Such measures will be called {\bf minimal measures}. 
\thm{\lbl{constant_equality}
Let $\T: \mcal{P}(X)\times \mcal{P}(X) \to \R\cup\{+\infty\}$ be a backward linear transfer, $T$ its associated backward Kantorovich operator, and $c(T)$ its Mather constant.  
\begin{enumerate}
\item If $c(\T)$ is  finite, then
 \eq{\lbl{min}
c(\T) = \lim_{n \to \infty}\frac{\inf\limits_{(\mu,\nu)\in \mcal{P}(X)\times \mcal{P}(X)}\T_n(\mu,\nu)}{n}.
}
\item If ${\bar \mu}$ is a minimal measure, then 
for each $g \in C(X)$, 
\eq{\lbl{mean}
\lim_{n \to \infty}\frac{1}{n}\int_XT^ng \, d\bar{\mu} = -c(\T).
}
\item  If $\T$ is a regular backward linear transfer, then 
\eq{\lbl{point}
c(\T) = \sup_{g \in C(X)}\inf_{x \in X}\{g(x) - T g(x)\}.
}

\end{enumerate}
}
\prf{1) 
First note that 
$\inf_{(\mu,\nu)}\T_n(\mu,\nu) \leq \T_n(\mu,\mu) \leq n\T(\mu,\mu)$,
hence 
\eqs{
\limsup_{n \to \infty}\frac{\inf_{(\mu,\nu)}\T_n(\mu,\nu)}{n} \leq  \inf_{\mu}\T(\mu,\mu) =c(\T).
}
 On the other hand, let $\mu_1^{n}, \mu^{n}_{n+1}$ be such that
\eq{\lbl{joint_convexity_inequality1}
\inf_{(\mu,\nu)}\T_n(\mu,\nu) = \T_n(\mu_1^{n}, \mu^{n}_{n+1}).
}
By definition of $\T_n$, 
we may write
$\T_n(\mu_1^{n}, \mu^{n}_{n+1}) = \sum_{j = 1}^{n}\T(\mu_j^{n}, \mu_{j+1}^{n})
$ 
  for some $\mu_j^{n} \in \mcal{P}(X)$, $j = 1,\ldots, n+1$ (the infimum is achieved by weak$^*$ lower semi-continuity). Hence by joint convexity, 
\eq{\lbl{joint_convexity_inequality3}
\frac{1}{n}\T_n(\mu_1^{n}, \mu^{n}_{n+1}) = \frac{1}{n}\sum_{j = 1}^{n}\T(\mu_j^{n}, \mu_{j+1}^{n})\geq \T(\frac{1}{n}\sum_{j = 1}^{n}\mu_j^{n}, \frac{1}{n}\sum_{j = 1}^{n}\mu_{j+1}^{n}).
}
Define $\nu_n := \frac{1}{n}\sum_{j = 1}^{n}\mu_j^{n}$. Then 
\eq{\lbl{joint_convexity_inequality4}
\T(\frac{1}{n}\sum_{j = 1}^{n}\mu_j^{n}, \frac{1}{n}\sum_{j = 1}^{n}\mu_{j+1}^{n}) = \T(\nu_n, \nu_{n} + \frac{1}{n}(\mu_{n+1}^n-\mu_1^n)).
}
Now let $n_k$ be a subsequence such that 
\eqs{
\liminf_{n \to \infty}\frac{\inf_{(\mu,\nu)}\T_n(\mu,\nu)}{n} = \lim_{k \to \infty}\frac{\inf_{(\mu,\nu)}\T_{n_k}(\mu,\nu)}{n_k}.
}
Up to extracting a further subsequence, we may assume that $\nu_{n_k} \to \bar{\nu}$ for some $\bar{\nu} \in \mcal{P}(X)$. It then follows from \refn{joint_convexity_inequality1}, \refn{joint_convexity_inequality3}, and \refn{joint_convexity_inequality4}, together with weak$^*$ lower semi-continuity of $\T$, that
\eqs{
\liminf_{n \to \infty}\frac{1}{n}\inf_{(\mu,\nu)}\T_n(\mu,\nu) \geq \T(\bar{\nu}, \bar{\nu}) \geq c(\T),
}
which concludes the proof of (1).

2) We actually show that for each $g\in C(X)$ and $\mu \in {\cal P}(X)$, we have 
\eq{
-\T(\mu,\mu) \leq \liminf_{n\to\infty}\frac{1}{n}\int_XT^ng \, d\mu
\leq \limsup_{n\to\infty}\frac{1}{n}\int_{X}T^ngd\mu \leq -c(\T),
}
which will clearly imply our claim if $\mu$ is minimal. Indeed, for the upper bound, write
\begin{align*}
\int_{X} T^ng d\mu =\sup\{\int_X g\ d\nu -\T_n(\mu, \nu)\,;\, \nu \in \mcal{P}(X)\}
\leq \sup_X g -\inf_{(\mu,\nu)}\T_n(\mu, \nu).
\end{align*}
Dividing by $n$ and using that $c(\T) = \lim_{n \to \infty}\frac{\inf_{(\mu,\nu)}\T_n(\mu,\nu)}{n}$, we deduce the stated upper bound for $\limsup_{n\to\infty}\frac{1}{n}\int_{X}T^ngd\mu$. For the lower bound, it suffices to write
\begin{align*}
\int_XT^ng d\mu =\sup\{\int_X g \ d\sigma -\T_n(\mu, \sigma)\,;\, \sigma \in \mcal{P}(X)\}
\geq \int_X g \, d\mu -\T_n(\mu, \mu)
\geq \int_X g \, d\mu -n\T (\mu, \mu).
\end{align*}

(3) Suppose now that that $\T$ is a regular transfer, that is $\sup\limits_{x \in X}\inf\limits_{\nu \in \mcal{P}(X)}\T(\delta_x, \nu) < +\infty$, then  $T g \in USC_b(X)$ for all $g \in C(X)$. Since $g - Tg$ is then a bounded below  lower semi-continuous function, it achieves a minimum on the compact space $X$, so that 
\as{
c(\T) &= \inf_{\mu \in \mcal{P}(X)}\sup_{g \in C(X)}\int_{X}(g-Tg)d\mu\\
&= \sup_{g \in C(X)}\inf_{\mu \in \mcal{P}(X)}\int_{X}(g-Tg)d\mu\\
&= \sup_{g \in C(X)} \min_{x \in X}\{g(x) - Tg(x)\}.  
}
Here we have applied Sion's min-max principle to the convex-concave function 
$
f:\mcal{P}(X)\times C(X) \to \R$ defined by$ f(\mu,g) := \int_{X}(g - Tg)d\mu$, which is finite since $Tg \in USC_b(X)$. Indeed, we have that $g \mapsto f(\mu,g)$ is upper semi-continuous on $C(X)$ since $T$ is lower semi-continuous on $C(X)$, and $\mu \mapsto f(\mu,g)$ is lower semi-continuous on $\mcal{P}(X)$ since $Tg \in USC_b(X)$. Moreover, $\mu \mapsto f(\mu,g)$ is quasi-convex, i.e. $\{\mu \in \mcal{P}(X)\,;\, f(\mu, g) \leq \lambda\}$ is convex or empty for $\lambda \in \R$, and $g \mapsto f(\mu,g)$ is quasi-concave, i.e. $\{g \in C(X)\,;\, f(\mu,g) \geq \lambda\}$ is convex or empty for $\lambda \in \R$.
}

\lem{Let $g$ be a bounded above function in $USC_\sigma(X)$, then for any minimal measure $\mu$, we have 
\eq{\lbl{mono}
\int_{X}gd\mu \leq \int_{X}T gd\mu + c(\T).
}
}
\prf{Take $\mu \in \mcal{P}(X)$ for which $c(\T) = \T(\mu,\mu)$. Since $\T(\mu,\mu) = \sup_{h \in C(X)}\{\int_{X}hd\mu - \int_{X}Th d\mu\},$ we have for any $h \in C(X)$, 
$
\int_{X}(h- Th)d\mu \leq c(\T).
$
 For any $g \in USC(X)$, take a decreasing sequence $(h_j) \subset C(X)$ with $h_j \searrow g$. Since $T$ is a capacity, we have $Th_j \searrow T g$, so by monotone convergence
\eq{\lbl{mono}
\int_{X}gd\mu \leq \int_{X}T gd\mu + c(\T).
}
And again since $T$ is a capacity, we have (\ref{mono}) for any $g\in USC_\sigma(X)$ that is bounded above. 
}
\defn{ 
(i) We say that 
a function $g$ is a \textbf{backward subsolution (resp., solution) for $T$ at level $k \in \R$} if
\enum{
\item $g$ is proper (i.e., $g \not\equiv -\infty$) and bounded above.  
\item $Tg(x) + k \leq g(x)$ (resp., $Tg(x) + k = g(x)$) for all $x \in X$.
} 

\noindent (ii)  We say that $g$ is \textbf{very proper} if $g\in USC_\sigma(X)$ and $\int_{X}gd\mu > -\infty$ for some minimal measure $\mu \in \mcal{P}(X)$. 
 }
\defn{
The \textbf{Ma\~n\'e constant} $c_0(T)$ is the supremum over all $k \in \R$ such that there exists a very proper subsolution $g$ for $T$ at level $k$. 
}
\prop{\lbl{Mane_constant_equality}
Let $T$ be a backward Kantorovich operator such that $c(T)<+\infty$. Then, 
\enum{
\item If $g$ is a backward subsolution at level $c(T)$ in $USC_\sigma(X)$, then for any minimal measure $\mu$,
\eq{\lbl{support}
-\infty \leq \int_{X}gd\mu = \int_{X}T gd\mu + c(\T). 
}
\item If $g$ is a very proper backward subsolution at level $k$, then $k\leq c(\T)$, hence $c_0(T) \leq c(T)$.

\item If $g$ is a very proper backward solution at level $k$ then necessarily, $k =c_0(T)= c(\T)$.

\item In general, we have 
\eq{\lbl{mane-transfer}
\sup_{g \in C(X)}\inf_{x \in X}\{g(x) - Tg(x)\} \leq c_0(T) \leq c(T), 
}
with equality if $T$ is regular. 
}
}
\prf{1) If $g$ is in $USC_\sigma(X)$ and is bounded above, then inequality (\ref{mono}) holds for any minimal measure $\mu$. If $g$ is also a subsolution $g$ at level $k$, then we have
\eqs{
\int_{X}gd\mu \leq \int_{X}T gd\mu + c(\T) \leq \int_{X}gd\mu + c(\T) - k, 
}
which proves (\ref{support}) when $k=c(T)$.

2) If now $g$ is very proper, then $\int_{X}gd\mu > -\infty$ for some minimal measure $\mu$, hence we may subtract it and deduce that $k \leq c(\T)$, therefore $c_0(T) \leq c(T)$. 

3) If $g$ is a very proper backward solution at level $k$, then it is a subsolution, so $k \leq c_0(T) \leq c(\T)$. On the other hand, $g$ also satisfies $T^n g + nk = g$, where $T^n$ is the $n$-fold composition of $T$ with itself. This means that
\eqs{
\int_{X}gd\mu - nk = \int_{X}T^n gd\mu = \sup_{\nu \in \mcal{P}(X)}\{\int_{X}gd\nu - \T_n(\mu,\nu)\} \leq \sup_Xg - \inf_{(\mu,\nu)}\T_n(\mu,\nu),
}
where for the second equality in the above, we have used the fact that $T^n$ is the Kantorovich operator associated to $\T_n$ 
and the extension of $T^n$ to $USC_\sigma^b(X)$. 
Dividing by $n$ and letting $n \to \infty$, we have
\eq{\lbl{weak_KAM_constant}
-k \leq -\lim_{n \to \infty}\frac{\inf_{(\mu,\nu)}\T_n(\mu,\nu)}{n} = -c(\T),
}
where the latter equality holds by Proposition \ref{constant_equality}. The above inequality \refn{weak_KAM_constant} then implies $c(\T) \leq k$, which concludes the proof of 3).

4) It is clear that if $k=\sup\limits_{g \in C(X)}\inf\limits_{x \in X}\{g(x) - Tg(x)\}>c_0(T)$, then there is $g\in C(X)$ so that $\inf_{x \in X}\{g(x) - Tg(x)\}=k' >c_0(T)$, and $g$ is  therefore a very proper subsolution at a level above $c_0$, which is a contradiction.

Assuming now that $\T$ is regular, then by Theorem \ref{constant_equality}, we have $\sup_{g \in C(X)}\inf_{x \in X}\{g(x) - Tg(x)\} = c_0(T) = c(T)$.
}

 A backward subsolution (resp., solution) for $T$ at level $c(T)$ will be called a {\bf backward weak KAM solution (resp., subsolution) for $T$}.

  \lem{\lbl{oscillation_limits} Let $\T:\mcal{P}(X)\times\mcal{P}(X) \to \R\cup\{+\infty\}$ be a standard backward linear transfer with $T$ as a Kantorovich operator. Assume that $c(\T)  < +\infty$, then  
for $g \in USC(X)$ that is bounded below, we have for all $n \in \N$,  
\begin{equation}\lbl{bounded_below}
 \sup_{x \in X} \{T^ng(x) +nc(\T)\}  \geq \inf_{y\in X}g(y).
\end{equation}
}
\noindent {\bf Proof:} Using the fact that $\inf_{ \mu,\sigma} \T_n(\mu, \sigma) \leq nc(\T)$, we have
\as{
\sup_{x \in X} (T^ng(x) +nc(\T))&=  \sup_{\mu \in {\cal P}(X)} \int_{X} (T^ng(x) +nc(\T))\, d\mu  \\
 &=  \sup_{\mu\in \mcal{P}(X)}\sup_{\sigma \in \mcal{P}(X)} \{\int_{X} gd\sigma -\T_n(\mu, \sigma) +nc(\T)\}\\ 
 &\geq  \sup_{\mu\in \mcal{P}(X)}\sup_{\sigma \in \mcal{P}(X)}\{ \inf_X g -\T_n(\mu, \sigma) +nc(\T)\}\\ 
      & = \inf_{X} g  - \inf_{\P(X)\times \P(X)} \T_n  +nc(\T)\\ 
                          &\geq \inf_X g. 
}
   
   \lem{\lbl{weak_KAM_for_monotone_decreasing} Suppose $(S_n)_n$ is a decreasing sequence of backward Kantorovich operators from $USC(Y)$ to $USC(X)$ and let $g\in USC(Y)$ be such that
   \eq{\lbl{bd}
   \inf\limits_n\sup\limits_{x\in X}S_ng(x) >-\infty.
   }
   Then, $h(x) := \lim\limits_{n \to \infty}\downarrow S_ng(x)$ belongs to $USC(X)$ (in particular, it is proper).
  } 
\noindent {\bf proof:} $(S_ng)_n$ is a decreasing sequence in $USC(X)$, hence it converges to its infimum $h$, which is therefore upper semi-continuous. To see that $h$ is proper, suppose not, which means that for every $x$, $S_ng(x)$ decreases to $-\infty$. There is then for each $x_0$, an $n_0 \in \N$ such that $S_{n_0}g(x_0) < K - 1$, where  $K:=\inf\limits_n\sup\limits_{x\in X}S_ng(x)$. Since $S_{n_0}g$ is upper semi-continuous, the inequality $S_{n_0}g(x) < K - 1$ must hold in a neighbourhood of $x_0$. Since $X$ is compact and $S_ng$ is decreasing, it follows that there is $M \in \N$ such that $S_Mg(x) \leq K - 1$ for all $x \in X$, which contradicts (\ref{bd}). 

 \lem{\lbl{decrease} Suppose $T$ is a backward Kantorovich operator from $USC(X)$ to $USC(X)$ such that $c(\T)  < +\infty$. If $g \in USC(X)$ is bounded below and $\{T^ng + nc(\T)\}_{n \in \N}$ is a decreasing sequence, then $h(x) := \lim\limits_{n \to \infty}\downarrow T^n g(x) + nc(\T)$ belongs to $USC(X)$ and $T h + c(\T) = h$.
 }
 \noindent {\bf proof:} It  follows from the previous lemma  applied to $S_ng:=T^ng + nc(\T)$. Note that since $g$ is bounded below, then by Lemma \ref{oscillation_limits}, we have $\sup\limits_{x\in X}S_ng(x) \geq \inf_Xg>-\infty$. We conclude that $h$ is proper and therefore belongs to $USC(X)$.  Since $T$ is a capacity, we conclude that 
 \eqs{
 Th + c(\T) = \lim_{n \to \infty}T(T_n g + nc(\T)) + c(\T) = \lim_{n \to \infty}T_{n+1} g + (n+1)c(\T) = h.
  }
  Here are situations where we can easily get weak KAM solutions in $USC(X)$. Stronger results will be obtained in the next section. 

\begin{corollary}Let $\T$ be a regular backward linear transfer such that $c(\T)=\inf\limits_{\P(X)\times \P(X)}\T$, then there exists a weak KAM solution in $USC(X)$ for the corresponding Kantorovich operator.
\end{corollary}
\prf{Note that in this case $\T-c$ is non-negative, hence 
\[
T1(x)+c= \sup_{\sigma \in \P(X)}\{\int_X1\, d\sigma -\T(x, \sigma)+c\}\leq 1.
\] 
It follows that 
 $(T^n1+nc)_n$ is a decreasing sequence  in $USC(X)$. By Lemma \ref{weak_KAM_for_monotone_decreasing}, the function  $h(x)=:\inf_n(T^n1(x)+nc)$ is a weak KAM solution in $USC(X)$.
}
 \prop{\label{hor1} Suppose $\T$ is a regular backward linear transfer on $\mcal{P}(X) \times \mcal{P}(X)$ such that  
there exist $u, v \in USC_b(X)$ satisfying 
\begin{equation}\label{hor}
T^nu +nv=u \quad \hbox{ for all $n\in \N$}.
\end{equation}
  Then, there exists a weak KAM solution $h\in USC(X)$.
}
\prf{ 
By (\ref{hor}) and since $u$ is bounded above and below, Assertion (\ref{mean}) of Theorem \ref{constant_equality} yields, 
$-v(x) = \limsup_{n}\frac{T^nu-u}{n}\leq -c(\T)$. Therefore,
$
T^nu +nc(\T) \leq u$ for all $n\in \N$. 
Applying $T_m$ and using the linearity of $T_m$ with respect to constants, we find $T_{m+n}u + nc(\T) \leq T_m^-u$, and hence
\eqs{
T_{m+n}u + (m+n)c(\T) \leq T_mu + mc(\T).
}
So $n \mapsto T^nu + nc(\T)$ is decreasing, and $u \in USC(X)$ and is bounded below. Consequently by Lemma \ref{weak_KAM_for_monotone_decreasing}, there exists $h \in USC(X)$ such that $T h + c(\T) = h$.
}

\thm{\label{gen} Let $\T$ be a regular  Kantorovich operator such that $c(\T)<+\infty$. Then, either there exists a weak KAM subsolution in $USC_\sigma(X)$, or for every $g\in C(X)$, there exists a sequence $(x_n)_n$ in $X$ such that $\liminf_{n\to \infty}T^ng(x_n)+nc(\T) =+\infty$.
}
\prf{Assume first there is $g\in C(X)$  
so that 
\begin{equation}\lbl{case1}
\hbox{$\forall x\in X$, there exists $n\geq 1$ with $T^ng(x)+nc(\T)< g(x)$.}
\end{equation}
  Since $T^ng$ is in $USC(X)$ and $g\in C(X)$, 
  it follows that $T_{n} g +nc(\T)< g$ on an open neighborhood $B_x \subset X$ of $x$. 
  Since $X$ is compact, there exists a finite number $\{x_1,\ldots, x_k\} \subset X$ such that $\{B_{x_j}\}_{1\leq j \leq k}$ cover $X$. Set $N := \max\{n_{x_1},\ldots, n_{x_k}\}$ and define $\vphi_N(x) :=\inf_{1\leq n \leq N}(T^ng(x) +nc(\T))$. We have $\vphi_N \in USC (X)$, which is bounded below since $\T$ is regular. Moreover, by construction we have for any $x \in X$, $\vphi_N(x) \leq g(x)$. By the monotonicity property of $T$, we have 
 \[
 T\vphi_N +c(\T) \leq Tg +c(\T) \quad \text{and}\quad T\vphi_N +c(\T)\leq \inf_{2\leq n \leq N+1}\{T^ng +nc(\T)\}.
 \]
Therefore, combining the two we deduce that 
\[
 T\vphi_N+c(\T)\leq \inf_{1\leq n \leq N}\{T^ng +nc(\T)\} =\vphi_N.
 \]
It follows that the sequence $\{T^n\vphi_N+nc(\T)\}_n$ is decreasing to a function $\tilde g \in USC(X)$ such that $T\tilde g + c(\T) = \tilde g$. Since $\vphi_N \in USC(X)$ is bounded below, Lemma \ref{decrease} yields that $\tilde g$ is proper. 

Suppose now that (\ref{case1}) does not hold. In this case, if for some $g\in C(X)$, the function $\underline T_\infty g:=\liminf_{n\to \infty}(T^ng+nc(\T))$ is bounded above, then it belongs to $USC_\sigma (X)$ and is proper since  $\underline T_\infty g (x_0) \geq g(x_0)>-\infty$. It is also a weak KAM subsolution since $\inf_{m \geq n}\{T_m^-g + m c\}$ is increasing to $\underline T_\infty g$, and $T$ being a capacity we have
 \as{
 T\underline T_\infty g + c = \lim_{n \to \infty}T(\inf_{m \geq n}\{T_m^-g + m c\}) + c
  \leq \liminf_{n \to \infty}(T_{n+1}^-g + (n+1)c)
   = \underline T_\infty g.
   }
   The remaining case is when $\underline T_\infty g:=\liminf_{n\to \infty}(T^ng+nc(\T))$ is not bounded above for any $g\in C(X)$. In other words, $\sup\limits_{x\in X}\sup\limits_n\inf\limits_{m \geq n}T_mg(x) + m c=+\infty$, hence $\alpha_n=\sup\limits_{x\in X}\inf\limits_{m \geq n}T_m^-g + m c\uparrow +\infty$. Since $\inf\limits_{m \geq n}T_m^-g + m c$ is in USC, $\alpha_n$ is attained at some $x_n$. It is now clear that 
   \[
T^ng(x_n)+nc(\T)\geq \inf\limits_{m \geq n}T_mg(x_n) + m c=\alpha_n\uparrow +\infty.
   \]
}

\subsection*{Power bounded Kantorovich operators}

 \defn{Let $\T$ be a standard linear backward transfer on $\P(X)\times \P(X)$, $T$ its Kantorovich operator and $c(T)$ the corresponding Mather constant.
 \begin{enumerate}
 \item  Say that the linear transfer  $\T$ has {\bf bounded oscillations} if
 \begin{equation}\lbl{bounded.osc}
\limsup_{n\to\infty} \{nc(\T)-\inf_{(x, \sigma) \in X\times \P(X)}\T_n(\delta_x, \sigma)\} <+\infty.
 \end{equation}
\item  Say that the Kantorovich operator $T$ is {\bf power bounded} if for any $g \in USC(X)$, 
 \begin{equation}\lbl{bounded.osc}
\limsup_{n\to\infty}\sup_{x\in X}T^ng(x) +nc(\T) 
<+\infty.  
 \end{equation}\lbl{power.bounded}
 \end{enumerate}
 }
 \thm{\lbl{when.bounded} Let $\T$ be a standard linear backward transfer on $\P(X)\times \P(X)$, $T$ its Kantorovich operator and $c(T)$ the corresponding Mather constant.
  \begin{enumerate}
 \item If $\T$  has bounded oscillations, then $T$ is {\it power bounded}.
 
 \item If $\T$ is bounded above, then $\T$  has bounded oscillations. In this case,   
 \begin{equation}
\frac{\T_n(\mu, \nu)}{n} \to c(\T) \quad \hbox{uniformly on ${\cal P}(X)\times {\cal P}(X)$,}
\end{equation}
and for every $g\in C(X)$, 
\begin{equation}\lbl{uniform}
\frac{T^ng(x)}{n}\to -c(T)  \quad \hbox{uniformly on $X$}.
\end{equation}
\end{enumerate}
 }
 \prf{1) For any $g \in USC(X)$ and any $x\in X$, 
 \begin{align*} 
 T^mg(x) +mc(\T) &\leq 
    \sup_{\sigma \in {\cal P}(X)}\{\int_X gd \sigma -\T_m(\delta_x, \sigma)  +mc(\T) \} \no\\
      &\leq  \sup_Yg+ mc(\T) -  
      \inf_{\sigma \in {\cal P}(X)}\T_m(\delta_x, \sigma)\\
       &\leq  \sup_Yg+ mc(\T) -  
      \inf_{(y, \sigma) \in X\times {\cal P}(X)}\T_m(\delta_y, \sigma),   
 \end{align*}
   which readily implies (1).
   
   2) If now $\T$ is bounded above, then $\{\frac{\sup_{\mu,\nu}\T_n(\mu,\nu)}{n}\}_{n \in \N}$ is a sub-additive sequence that satisfies \\
$\frac{\sup_{\mu,\nu}\T_n(\mu,\nu)}{n} \geq \inf_{\mu,\nu}\T(\mu,\nu) > -\infty,$
hence it converges to its infimum.  
Since 
\eqs{
c(\T) = \lim_{n \to \infty}\frac{\inf_{\mu,\nu}\T_n(\mu,\nu)}{n}  \quad \text{and} \quad  \frac{\inf_{\mu,\nu}\T_n(\mu,\nu)}{n} \leq \frac{\sup_{\mu,\nu}\T_n(\mu,\nu)}{n},
}
we conclude that
$c(\T) \leq \inf\limits_{n \in \N}\frac{\sup_{\mu,\nu}\T_n(\mu,\nu)}{n}.
$
Therefore, for all $n \in \N$,
\eqs{
\inf_{\mu,\nu}\T_n(\mu,\nu) \leq n c(\T) \leq \sup_{\mu,\nu}\T_n(\mu,\nu),
}
hence,
\eq{\lbl{uniform_est}
|\T_n(\mu,\nu) - nc(\T)| \leq \sup_{\mu,\nu}\T_n(\mu,\nu) - \inf_{\mu,\nu}\T_n(\mu,\nu).
}
At the same time, we have for any $\mu, \nu \in {\cal P}(X)$ (writing $\inf\limits_{{\cal P}\times {\cal P}} \T_{n}$ for $\inf\limits_{\mu,\nu \in \mcal{P}(X)}\T_n(\mu,\nu)$ for brevity),  
 \[
 \inf\limits_{\mcal{P}\times\mcal{P}} \T_{n-2} + 2\inf\limits_{\mcal{P}\times\mcal{P}}  \T \leq \T_{n}(\mu, \nu) \leq 2\sup\limits_{\mcal{P}\times\mcal{P}}  \T +\inf\limits_{\mcal{P}\times\mcal{P}}  \T_{n-2}, 
 \]
from which follows that 
\begin{align}
\sup_{\mcal{P}\times \mcal{P}}  \T_{n} -\inf_{\mcal{P}\times\mcal{P}}  \T_{n} &\leq 2\sup_{\mcal{P}\times\mcal{P}}  \T +\inf_{\mcal{P}\times\mcal{P}}  \T_{n-2}- \inf\limits_{{\cal P}\times {\cal P}} \T_{n-2} - 2\inf\limits_{{\cal P}\times {\cal P}}  \T \no\\
&=2\sup\limits_{{\cal P}\times {\cal P}}  \T - 2\inf\limits_{{\cal P}\times {\cal P}}  \T \no\\
&=:K < \infty.\lbl{upper_and_lower}
\end{align}
Combining \refn{uniform_est} and \refn{upper_and_lower}, we conclude that
\eq{\lbl{osc}
|\T_n(\mu,\nu) - nc(\T)| \leq K \quad \text{for all $\mu,\nu \in \mcal{P}(X)$ and all $n \in \N$,}
}
and $\T$ has therefore bounded oscillations. This also implies that 
  $\frac{\T_n(\mu, \nu)}{n} \to c(\T)$ uniformly on ${\cal P}(X)\times {\cal P}(X)$.

Now note that $T^n g(x) + nc(\T) = \sup_{\sigma}\{\int g d\sigma - \lf(\T_n(\delta_x, \sigma)-nc(\T)\rt)\}$, hence 
$\sup_{X}g - K \leq T^n g(x) + nc(\T) \leq \sup_{X}g + K,$
which yields  
(\ref{uniform}).
}

 \thm{\label{bdo} 
 If $T$ is a power bounded Kantorovich operator,  then there exists a backward weak KAM solution $h$ in $USC_\sigma (X)$.
    }
\prf{ Since $T$ is power bounded, then for any $g \in USC(X)$, 
$\sup\limits_{m\geq n}(T^mg +mc)$ is bounded above on $X$ for $n$ large enough.  
Setting $S_ng:= \overline{\sup_{m\geq n} T^m g+mc}$, where $\overline h$ is the upper semi-continuous envelope of $h$, we claim that for $n$ large enough,
\begin{itemize}
\item $S_ng$ is bounded above, and $TS_ng +c\geq S_{n+1}g,$
\item The sequence $(S_ng)_n$ decreases to an upper semi-continuous function $\bar T_\infty g$,
  which is a weak KAM supersolution, that is, 
 \eq{ T\bar T_\infty g +c \geq  \bar T_\infty g.
 } 
 \end{itemize}
 Indeed, use the Kantorovich properties of each $T^m$, of the operator $f\to \overline f$,  and the fact that $T\bar h \geq \overline{Th}$ since $T$ is monotone and maps $USC (X)$ into $USC(X)$, we can write
   \as{
 TS_n g +c&=T\big(\overline {(\sup_{m\geq n} T^m g+mc})) +c\\
  &\geq \overline {T (\sup_{m\geq n} T^m g+(m+1)c)}\\
  &\geq \overline {\sup_{m\geq n} (T^{m+1} g+(m+1)c)}\\
    &= \overline {\sup_{m\geq {n+1}} T^{m} g+mc}\\
    &=S_{n+1} g. 
}
Consider now the function 
 \begin{equation}
\bar T_\infty g :=\inf_nS_ng=\inf_n \overline{\sup_{m\geq n} T^m g+mc}\}\geq \overline{\limsup_{n\to\infty} T_ng +nc(\T)},
 \end{equation}
and note that $\bar T_\infty g$ is upper semi-continuous. Moreover, by Lemma \ref{oscillation_limits}, we have for each $n$, 
$$\sup_{x\in X}S_ng(x)\geq  \sup_{x\in X}T_ng(x) +nc(\T) \geq \inf_{x\in X} g (x),$$
and therefore, $\bar T_\infty g$ is proper by Lemma \ref{weak_KAM_for_monotone_decreasing}.\\
Since $T$ is a capacity, we have 
\as{
 T \,\bar T_\infty g +c=T(\inf_nS_ng) +c=\inf_n TS_ng+c\geq \inf_nS_{n+1}g
    = \bar T_\infty g. 
 }
Now the sequence $(T^n\, \bar T_\infty g +nc)_n$ is increasing and since $T$ is power bounded, we have 
that $\sup_nT^n\,  \bar T_\infty g+nc$ is bounded above, hence by setting $h=\sup_nT^n\,  \bar T_\infty g+nc$, we have 
 that $h\in USC_\sigma(X)$, that it is bounded above and since $T$ is a capacity, 
$$Th+c=T(\lim_n\uparrow T^n\,  \bar T_\infty g +nc) +c=\lim_n\uparrow (T^{n+1}\, \bar T_\infty g +(n+1)c)=h.$$
} 
\begin{corollary}Let $\T$ be a backward linear transfer that is bounded above, then there exists a backward weak KAM solution $h$ in $USC_\sigma (X)$.
\end{corollary}
\prf{It suffices to apply Theorem \ref{bdo} since by Lemma \ref{when.bounded}, $\T$ has then bounded oscillations. 
}

\section{Idempotent Kantorovich operators}

\defn{1) A map $T: USC (X) \to USC(X)$ is said to be \textbf{idempotent} if $T \circ T g = T g$ for all $g \in USC(X)$.

2) A functional $\T:\mcal{P}(X)\times \mcal{P}(X) \to \R\cup\{+\infty\}$ is said to be {\bf idempotent} if $\T\star \T=\T$, that is if 
 for all $\mu,\nu \in \mcal{P}(X)$, 
\eqs{
\T(\mu,\nu) := \inf\{\T(\mu,\sigma) + \T(\sigma, \nu)\,;\, \sigma \in \mcal{P}(X)\}.
}
}
An immediate consequence of Corollary \ref{conv}  is the following 
\prop{A backward linear transfer $\T$ is idempotent if and only if its corresponding backward Kantorovich operator $T$ is idempotent.
}
\prop{\lbl{A_factorisable} Let $T$ be an idempotent backward Kantorovich operator, then $c(T)=0$, and if $\T$ is the associated linear transfer, then 
\begin{enumerate}
\item The set $\mcal{N} := \{\mu \in \mcal{P}(X)\,;\, \T(\mu,\mu) = 0\}$ is non-empty.
 
 \item For all $\mu,\nu \in \mcal{P}(X)$, we have 
 \eq{
\T(\mu,\nu) = \inf\{\T(\mu,\sigma) + \T(\sigma, \nu)\,;\, \sigma \in \mcal{N}\}.
}
\end{enumerate} 
}
We shall then say that {\it $\T$ is Null-factorizable}.
\prf{First note that since $\T_n=\T$ for all $n$, we have
\[
c(\T)=\lim_{n\to +\infty} \frac{\inf_{\mu, \nu}\T_n(\mu, \nu)}{n}= \lim_{n\to \infty}\frac{\inf_{\mu, \nu}\T (\mu, \nu)}{n}=0.
\]
For the remaining claim, we shall use repeatedly the following observation:  if 
$\sigma_1,\sigma_2$ are such that 
\eqs{
\T(\mu,\nu) = \T(\mu,\sigma_1) +  \T(\sigma_1, \nu) \quad \text{and}\quad \T(\sigma_1, \nu) = \T(\sigma_1,\sigma_2) + \T(\sigma_2, \nu), 
}
then $\T(\mu,\sigma_1) + \T(\sigma_1,\sigma_2) = \T(\mu,\sigma_2)$.
Indeed 
\begin{align*}
\T(\mu,\nu) &= \T(\mu,\sigma_1) + \T(\sigma_1, \sigma_2)+ \T(\sigma_2, \nu)\\
&\geq \T\star\T(\mu, \sigma_2) + \T(\sigma_2, \nu)\\
&= \T(\mu, \sigma_2) + \T(\sigma_2, \nu)\\
&\geq \T\star\T(\mu,\nu)\\
&= \T(\mu,\nu)
\end{align*}
so all the inequalities are in fact equalities. This means in particular comparing the first and third line that $\T(\mu,\sigma_1) + \T(\sigma_1, \sigma_2) = \T(\mu, \sigma_2)$.

Fix now $\mu, \nu \in \mathcal{P}(X)$.  
There exists then $\sigma_1 \in \mathcal{P}(X)$ such that 
\begin{equation}\lbl{first_inf}
\T(\mu,\nu) = \T(\mu,\sigma_1) + \T(\sigma_1, \nu).
\end{equation}
Similarly, there exists a $\sigma_2$ such that 
\begin{equation}\lbl{second_inf}
\T(\sigma_1, \nu) = \T(\sigma_1, \sigma_2)+ \T(\sigma_2, \nu).
\end{equation}
Adding \refn{first_inf} and \refn{second_inf}, we have
\begin{equation}\lbl{inf_equality_double}
\T(\mu,\nu) = \T(\mu,\sigma_1) + \T(\sigma_1, \sigma_2)+ \T(\sigma_2, \nu).
\end{equation}
From the above observation, we have
\begin{equation}\lbl{split1}
\T(\mu,\sigma_1) + \T(\sigma_1, \sigma_2) = \T(\mu,\sigma_2).
\end{equation}

We therefore have $(\sigma_1, \sigma_2)$ such that \refn{first_inf}, \refn{second_inf},   and \refn{split1}  (and, consequently, \refn{inf_equality_double}) hold.
Continue this process with $\T(\sigma_2, \nu)$, to find a $\sigma_3$ satisfying $\T(\sigma_2, \nu) = \T(\sigma_2, \sigma_3) + \T(\sigma_3, \nu)$, and a $\sigma_4$ such that $\T(\sigma_3, \nu)=\T(\sigma_3, \sigma_4)+\T(\sigma_4, \nu)$. \\
 Inductively, we get a sequence $(\sigma_k)_{k \in \N}$ with the property that for $(\sigma_1, \sigma_2, \ldots, \sigma_k)$, we have 
\eq{\lbl{prop_idem_1}
\T(\mu,\nu) = \T(\mu,\sigma_1) + \sum_{i = 1}^{k-1}\T(\sigma_{i},\sigma_{i+1}) + \T(\sigma_k, \nu),
}
\al{
\T(\mu,\sigma_1) + \T(\sigma_1,\sigma_2) &= \T(\mu, \sigma_2)\lbl{prop_idem_2}\\
\T(\sigma_{k-1}, \sigma_k) + \T(\sigma_k, \nu) &= \T(\sigma_{k-1}, \nu), \lbl{prop_idem_3}
}
as well as
\begin{equation}\lbl{prop_idem_4}
\sum_{i = \ell}^{m} \T(\sigma_{i}, \sigma_{i+1}) =  \T(\sigma_{\ell}, \sigma_{m+1})
\end{equation}
whenever $1 \leq \ell < m \leq k-1$. 

In particular, the same properties \refn{prop_idem_1}, \refn{prop_idem_2}, \refn{prop_idem_3}, \refn{prop_idem_4} above hold if we take instead $m+1$ terms $(\sigma_{k_1}, \ldots, \sigma_{k_{m+1}})$ of any increasing subsequence $k_j$ of $k$. In particular, this means
\begin{equation}\label{sebseq1}
\T(\mu,\sigma_{k_1}) + \sum_{j = 1}^{m} \T(\sigma_{k_j}, \sigma_{k_{j+1}}) + \T(\sigma_{k_{m+1}}, \nu) =  \T(\mu, \nu).
\end{equation}
We take now a subsequence $\sigma_{k_j}$ of $\sigma_k$ converging to some $\sigma \in \mathcal{P}(X)$. By the weak$^*$ lower semi-continuity of $\T$, we have
\as{
\liminf_{j \to \infty}\T(\sigma_{k_j}, \sigma_{k_{j+1}}) \geq \T(\sigma,\sigma), \quad 
\liminf_{j \to \infty}\T(\mu, \sigma_{k_{j}}) \geq \T(\mu,\sigma),\,\,{\rm and}\,\,
\liminf_{j \to \infty}\T(\sigma_{k_j}, \nu) \geq \T(\sigma, \nu).
}
In particular, given $\epsilon > 0$, for all but finitely many $j$, it must hold that
\al{
\T(\sigma_{k_j}, \sigma_{k_{j+1}}) &\geq \T(\sigma,\sigma) - \epsilon\label{biggerthan1}\\
\T(\mu, \sigma_{k_{j}}) &\geq \T(\mu,\sigma) - \epsilon\label{biggerthan2}\\
\T(\sigma_{k_j}, \nu) &\geq \T(\sigma, \nu) - \epsilon \label{biggerthan3}
}
so up to removing the first $N$ terms for a finite $N = N_\epsilon$, we may assume we have a subsequence $\sigma_{k_j}$ satisfying \refn{biggerthan1}, \refn{biggerthan2}, and \refn{biggerthan3} for all $j$, as well as \refn{sebseq1} .

Applying the inequalities of (\ref{biggerthan1}), (\ref{biggerthan2}), and (\ref{biggerthan3}), to (\ref{sebseq1}), we obtain
\begin{equation*}
\T(\mu,\nu) \geq \T(\mu,\sigma) + m\T(\sigma,\sigma) + \T(\sigma,\nu)- (m+2)\epsilon
\end{equation*}
for $m \geq 1$. From the fact that $\T = \T \star \T$, we have $\T(\mu,\sigma) + \T(\sigma,\nu) \geq \T\star\T(\mu,\nu) = \T(\mu,\nu)$, so the above inequality implies 
$\T(\sigma,\sigma) \leq \frac{m+2}{m}\epsilon \leq 2\epsilon.$ 
As $\epsilon > 0$ is arbitrary, we obtain $\T(\sigma,\sigma) \leq 0$, and consequently $\T(\sigma,\sigma) = 0$ (the reverse inequality following from $\T = \T \star \T$). 

Finally, we note that $\T (\mu,\nu) = \T(\mu,\sigma_{k_j}) + \T(\sigma_{k_j},\nu)$ for all $j$, so at the $\liminf$, we find that
$\T(\mu,\nu) \geq \T(\mu,\sigma) + \T(\sigma,\nu).$
The reverse inequality is immediate since $\T = \T \star \T$.
}
Here are a few examples of idempotent  transfers.
\ex{[Convex energy]
 If $I$ is any bounded below, convex, and weak$^*$ lower semi-continuous functional on $\P(Y)$, and $k:=\inf\{I(\sigma); \sigma \in {\mathcal P}(Y)\}$, then $\T(\mu, \nu) := I(\nu)-k$ is an idempotent backward linear transfer with an idempotent Kantorovich map $Tg := I^*(g) +k$, where $I^*$ is the Legendre transform of $I$ on $C(X)$. 
}
\ex{ [Markov operator]Any idempotent Markov operator (i.e.,  bounded positive linear operator $T$ with $T^2=T$ and $T1=1$), and in particular, any point transformation $\sigma$ such that $\sigma^2=\sigma$ is obviously an idempotent Kantorovich operator. 
}

\ex{[Balayage transfer] Let $X$ is a convex compact subset of a locally convex topological vector space, and consider
 \begin{equation}
{\mathcal B}(\mu, \nu)=\left\{ \begin{array}{llll}
0 \quad &\hbox{if $\mu \prec \nu$}\\
+\infty \quad &\hbox{\rm otherwise,}
\end{array} \right.
\end{equation} 
 where 
 $
 \mu \prec \nu$ if and only if  $\int_X\phi \, d\mu \leq \int_X\phi \, d\nu$ for all convex continuous functions $\phi$. It follows from standard Choquet theory \cite{G}, that ${\mathcal B}$ is a backward linear transfer whose Kantorovich operator is $Tf={\hat f}$, where ${\hat f}$ is the upper semi-continuous concave upper envelope of $f$, which is clearly an idempotent operator.  
 
 Note that ${\mathcal B}$ is also a forward  linear transfer with $T^+f$ being the lower semi-continuous convex lower envelope of $f$. 
}

\ex{[Optimal Mass Transport] If $\T_c$ is an optimal mass transport associated to a bounded below lower semi-continuous cost function $c: X\times X \to \R\cup\{+\infty\}$, then $\T_c$ is idempotent if  
\begin{equation}
c(x,x)=0, \quad{\rm and}\quad  c(x,z) \leq c(x, y) +c (y, z) \quad \hbox{for all $x, y, z$ in $X$.}
\end{equation}
In particular, this implies
$\T_c(\mu, \nu)=\sup\{\int_XT_c^-g \, d (\nu-\mu); g\in C(X)\}.$ 
For example, if $c(x,y) = d_X(x,y)^p$ for $0 < p \leq 1$,  then the corresponding optimal mass transport is idempotent since $c$ satisfies the reverse triangle inequality, and therefore $c\star c(x,y) = \inf_{z \in X}\{c(x,z) + c(z,y)\} = c(x,y)$. 
}

\ex{[An optimal Skorohod embedding]
 The following transfer  was considered in \cite{GKP3}
 	\begin{align} \label{eqn:Skorokhod_cost}
		{\mathcal T}(\mu,\nu) :=  \inf
		\Big\{\mathbb{E}\Big[ \int_0^\tau L(t,B_t)dt\Big];\ \tau \in S(\mu, \nu)  
		\Big\},  
	\end{align}
where $S(\mu, \nu)$ denotes the set of (possibly randomized) stopping times with finite expectation such that $\nu$ is realized by the distribution of $B_\tau$ (i.e., $B_\tau \sim\nu$ in our notation), where $B_t$ is Brownian motion starting with $\mu$ as a source distribution, i.e., $B_0\sim \mu$. Note that  ${\mathcal T}(\mu,\nu)=+\infty$ if $S(\mu, \nu)=\emptyset$, which is the case if and only if $\mu$ and $\nu$ are not in subharmonic order, i.e., there is $\phi$ subharmonic such that $\int \phi d\mu > \int \phi d\nu$. 

We claim that $\T$ is idempotent provided $t\to L(t, x)$ is decreasing. Indeed,  the backward linear transfer is given by $Tg=J_g(0, \cdot)$,  
 where $J_g:\R^+\times \R^d\rightarrow \R$ is defined via the 
	dynamic programming principle
	\begin{align}\label{eqn:dynamic_programming}
		J_g(t,x) := \sup_{\tau \in \mathcal{R}^{t,x}}\Big\{\mathbb{E}^{t,x}\Big[g(B_\tau)-\int_t^\tau L(s,B_s)ds\Big]\Big\},
	\end{align}
	where the expectation superscripted with $t,x$ is with respect to the Brownian motions satisfying $B_t=x$, and the minimization is over all finite-expectation stopping times $\mathcal{R}^{t,x}$ on this restricted probability space such that $\tau \ge t$.
	$J_g(t, x)$ is actually a ``variational solution" for 
the quasi-variational Hamilton-Jacobi-Bellman equation:
	\begin{align} 
		 \min\left\{\begin{array}{r} J(t,x) - g(x)\\ -\frac{\partial}{\partial t}J(t,x)-\frac{1}{2}\Delta J(t,x)+L(t,x)\end{array}\right\}=0.
	\end{align}
	Note that $J_g (t, x) \geq g (x)$, that is $Tg \geq g$ for every $g$, hence $T^2g \geq T g$.
	
Assume now $t\to L(t, x)$ is decreasing, which yields that $t\to J(t, x)$ is increasing. Given $\epsilon > 0$, fix a stopping time $\tau_\epsilon$ such that
$T^2 g(x) = J_{Tg} (0, x)\leq \mathbb{E}^{t,x}\Big[Tg(B_{\tau_\epsilon})-\int_t^{\tau_\epsilon} L(s,B_s)ds\Big]+\epsilon. $ Since 
 $Tg(B_{\tau_\epsilon}) = J_g(0, B_{\tau_\epsilon}) \leq J_g(t, B_{\tau_\epsilon})$, we have
\begin{align*}
T^2g(x) &\leq \mathbb{E}^{t,x}\Big[J_g(0, B_{\tau_\epsilon})-\int_t^\tau L(s,B_s)ds\Big]+\epsilon\\
&\leq \mathbb{E}^{t,x}\Big[J_g(t, B_{\tau_\epsilon})-\int_t^{\tau_\epsilon} L(s,B_s)ds\Big]+\epsilon\\
&\leq J_g (0,x)+\epsilon = Tg(x) + \epsilon,  
\end{align*}
where the last inequality uses the supermartingale property of the process $t\to J_g(t, B_{\tau_\epsilon})-\int_t^{\tau_\epsilon} L(s,B_s)ds$. It follows that  $T^2g \leq T g$, and $T$ is therefore idempotent.
}

\defn{Let $\T$ be both a backward and forward transfer with Kantorovich operators $T^-$ and $T^+$. A couple $(f, g)$
 is said to be a \textbf{conjugate pair} if
\begin{equation}
\hbox{$T^-g=f$\quad  and \quad $T^+f=g$.}
\end{equation}
}
\noindent The following proposition shows in particular that for any function $g\in C(Y)$, the couple $(T^-g, T^+\circ T^-g)$ form a conjugate pair. 

\prop{\label{few} Suppose  ${\mathcal T}: {\mathcal P}(X)\times {\mathcal P}(Y) \to \R\cup \{+\infty\}$ is both a forward and backward linear transfer, with $T^+:C(X) \to LSC(Y)$ (resp., $T^-:C(Y) \to USC(X)$) being the corresponding forward (resp., backward) Kantorovich operator. 
Then for $g\in LSC(Y)$
(resp.,\, $f\in USC(X)$), we have  
\begin{equation}\label{compare}
\hbox{$T ^+\circ T ^-g \geq g$  \quad  \quad $T ^-\circ T ^+f \leq f$,}
\end{equation}
and for $g\in C(Y)$
(resp.,\, $f\in C(X)$), we have 
 \begin{equation} 
 \hbox{$T ^-\circ T ^+\circ T ^-g=  T ^-g$ \quad and \quad $T ^+\circ T ^-\circ T ^+f = T ^+f.$}
 \end{equation} 
 In particular, the couple $(T^-g, T^+\circ T^-g)$ form a conjugate pair.
}
\prf{Since $f\in USC(X)$, $T^+f\in LSC_\delta(Y)$, and 
$$T^- \circ T^+ f(x) = \sup_{\sigma \in \mcal{P}(Y)}\inf_{\mu \in \mcal{P}(X)}\{\int_{X}fd\mu + \T(\mu,\sigma) - \T(\delta_x,\sigma)\}
\leq f(x),
$$
and by a similar argument, $T^+ \circ T^- g(y) \geq g(y)$ for all $g \in LSC(Y)$.

If $g \in C(Y)$, then by applying (\ref{compare}) with $f=T^-g$, we get that 
$T^-\circ T^+ \circ T^-g \leq T^-g$. At the same time applying $T^-$ to the inequality $T^+\circ T^- g \geq g$, implies $T^-\circ T^+\circ T^-g \geq T^-g$. Therefore we have the equality $T^-\circ T^+\circ T^-g = T^-g$. Similarly, we have $T ^+\circ T ^-\circ T ^+f = T ^+f$.}

\thm{\lbl{lipschitz}
Suppose $\T: \mcal{P}(X)\times \mcal{P}(X) \to \R \cup \{+\infty\}$ is a bounded below, Null-factorisable weak$^*$-lower semi-continuous and convex functional, where $\mcal{N} := \{\mu \in \mcal{P}(X)\,;\, \T(\mu,\mu) = 0\}$, and let $\Phi: {\cal N}\to \R$ be a bounded functional that is $\T$-Lipschitz on $\mcal N$, i.e., 
\eqs{
\Phi(\nu) - \Phi(\mu) \leq \T(\mu,\nu),\quad \text{for all }\mu,\nu \in \mcal{N}.
}
Then, the following hold:
\begin{enumerate}
\item 
There exists ${\bar f}\in C(X)$ such that 
$\Phi(\mu)= \int_{X} {\bar f}d\mu \quad \text{for all $\mu \in \mcal{N}$.}$
\item If $\T$ is also a backward linear transfer such that 
$\sup\limits_{\nu \in \mcal{P}(X)}\inf\limits_{\mu \in \mcal{N}}\T(\mu,\nu) < +\infty,
$
then
\begin{equation}\lbl{equality_on_N}
\Phi(\mu) =\int_X T^-{\bar f}d\mu = \int_X {\bar f}d\mu \hbox{\,  for every  $\mu\in \cal N$. }
\end{equation} 

\item If in addition, $\T$ is both a forward and backward linear transfer, then 
\begin{equation*}
\Phi(\mu)=\int_X {\bar f}d\mu=\int_X T^-{\bar f}d\mu=\int_X T^+\circ T^-{\bar f}d\mu \hbox{\,  for every  $\mu\in \cal N$. }
\end{equation*} 
and the functions  $\psi_0:=T^-{\bar f}$ and $\psi_1:=T^+\circ T^-{\bar f}$ are conjugate in the sense that $\psi_0=T^-\psi_1$ and $\psi_1=T^+\psi_0$.\\
 Moreover, if $g$ is any other function in $C(X)$ such that $\int_X g d\mu=\Phi(\mu)$  for all $\mu \in \cal N$, then 
\begin{equation}\label{unique}
\psi_0 \leq T^-g \quad \hbox{and}\quad \psi_1 \geq T^+g.
\end{equation}
 \end{enumerate}
 }
 \prf{
 1. Let $\Phi$ be such that $\mu \to \Phi (\mu)$ is $\T$-Lipschitz on $\cal N$ and define 
$$\Phi_0(\mu) :=\sup\limits_{\sigma \in {\cal N}} \{\Phi (\sigma) -\T (\mu, \sigma)\} \quad \hbox{and \quad $\Phi_1(\nu):=\inf\limits_{\sigma \in {\cal N}} \{\Phi (\sigma) +\T (\sigma, \nu)$\}}.
$$ 
Note that $\Phi_0$ and $\Phi_1$ are both finite on $\mcal{N}$, but in general $\Phi_0$ may be $-\infty$ (resp., $\Phi_1$ may be $+\infty$) on certain subsets of $\mcal{P}(X)$, depending on the effective domain of $\T$.

We now show that $\Phi_0 \leq \Phi_1$ on $\mcal{P}(X)$. This is trivally true if one of $\Phi_0$, $\Phi_1$ is not finite, so assume $\mu \in \mcal{P}(X)$ is such that $\Phi_0(\mu), \Phi_1(\mu) \in \R$. Then, by definition of $\Phi_0$, $\Phi_1$, and the fact that $\Phi$ is $\T$-Lipschitz on $\mcal{N}$, and $\T$ is $\mcal{N}$-factorisable, we may write
\begin{align*}
\Phi_0(\mu) - \Phi_1(\mu)=\sup\limits _{\sigma, \tau \in {\cal N}}\{\Phi (\sigma) -\T (\mu, \sigma) -\Phi (\tau) -\T (\tau, \mu)\}
\leq \sup\limits _{\sigma, \tau \in {\cal N}}\{\Phi (\sigma) - \Phi (\tau) -\T(\tau, \sigma)\}
 \leq 0.
\end{align*}

Note now that $\Phi_0$ is concave and weak$^*$ upper semi-continuous, while $\Phi_1$ is convex and weak$^*$ lower semi-continuous, on the convex subset $\mcal{P}(X)$ of the real vector space $\mcal{M}(X)$. Thus by the Hahn-Banach, there exists $\bar{f} \in C(X)$ such that $\mu \mapsto \int_{X}{\bar f}d\mu$ on $\mcal{P}(X)$ lies between $\Phi_0$ and $\Phi_1$:
\begin{equation}\lbl{sandwich}
\Phi_0 (\mu) \leq \int_{X}{\bar f}d\mu \leq \Phi_1(\mu) \quad \hbox{for all  $\mu\in {\cal P}(X)$.}
\end{equation} 
On the other hand, it also holds by the definition of $\Phi_0$ and $\Phi_1$, that
\begin{equation}\label{on.N}
\Phi_1(\mu) \leq \Phi (\mu) \leq \Phi_0 (\mu) \hbox{\quad for all $\mu\in {\cal N}$} 
\end{equation}
so that  \refn{sandwich} and \refn{on.N} together shows that
\eqs{
\Phi (\mu) = \Phi_1(\mu) = \Phi_0(\mu) = \int_{X} {\bar f}d\mu \hbox{\quad for all $\mu\in {\cal N}$}.  
}
2. Suppose now $\T$ is a backward linear transfer with $T^-$ as its Kantorovich operator. If $\mu\in {\cal N}$, 
\as{
\int_XT^-{\bar f}d\mu =\sup\limits_{\nu \in {\cal P}(X)} \{\int_X {\bar f}d\nu-\T(\mu, \nu)\} \geq \int_X {\bar f}d\mu -\T(\mu, \mu)= \int_X {\bar f}d\mu.
}
Assume now that $\sup_{\nu \in \mcal{P}(X)}\inf_{\mu \in \mcal{N}}\T(\mu,\nu) < +\infty$. It can be checked that this implies $\Phi_0$, $\Phi_1$ are bounded. Then it holds that
\as{
\sup_{\nu \in {\cal P}(X)} \{\Phi_1(\nu) -\T(\mu, \nu)\} \leq \sup_{\nu \in \mcal{P}(X)}\Phi_1(\nu) - C < +\infty,
}
where $C$ is a lower bound for $\T$. This means that for any $\mu \in \mcal{N}$, by \refn{sandwich}, we can write
\eq{\lbl{Phi_0}
-\infty < \int_XT^-\bar{f} d\mu =\sup\limits_{\nu \in {\cal P}(X)} \{\int_X {\bar f}d\nu -\T(\mu, \nu)\} \leq \sup\limits_{\nu \in {\cal P}(X)} \{\Phi_1(\nu) -\T(\mu, \nu)\}.
}
In view of item 1, we will be done if we show that we have the conjugate formula, 
\eq{\lbl{BB222}
\sup_{\sigma \in {\cal P}(X)} \{\Phi_1(\sigma)-\T(\mu, \sigma) \} = \Phi_0(\mu).
}
To this end, for every $\mu \in {\cal P}(X)$, we have 
\begin{align*}
\Phi_0(\mu)=\sup\limits_{\sigma \in {\cal N}} \{ \Phi (\sigma) -\T (\mu, \sigma)\}
=\sup\limits_{\sigma \in {\cal N}} \{ \Phi_1(\sigma) -\T (\mu, \sigma)\} 
\leq \sup\limits_{\sigma \in {\cal P}(X)} \{  \Phi_1(\sigma) -\T (\mu, \sigma)\}.   
\end{align*}
On the other hand, for any $\nu, \mu \in {\cal P}(X)$, we have 
\begin{align*}
\Phi_1(\nu)-\Phi_0(\mu)&=\inf\limits _{\sigma, \tau \in {\cal N}}\{\Phi (\sigma) +\T(\sigma, \nu)- \Phi (\tau) +\T(\mu, \tau)\} \\
&\leq \inf\limits _{\sigma, \tau \in {\cal N}}\{\T(\sigma, \nu) +\T (\tau, \sigma) +\T(\mu, \tau)\} \\
&\leq \inf\limits _{\sigma \in {\cal N}}\{\T(\sigma, \nu) +\T (\mu, \sigma)\} \\
&=\T (\mu, \nu).
\end{align*}
This shows \refn{BB222}.

3. Suppose in addition that $\T$ is a forward linear transfer with $T^+$ as a Kantorovich operator. Recalling the property $T^+\circ T^- f \geq f$, we have
$\int_XT^+\circ T^-{\bar f} d\mu \geq \int_X {\bar f} d\mu \geq \Phi_0(\mu)$.
On the other hand, by \refn{Phi_0} and \refn{BB222},
\[
\int_XT^+\circ T^-{\bar f}d\mu=\inf\limits _{\sigma \in {\cal P}(X)} \{\int_XT^-{\bar f}d\sigma +\T(\sigma, \mu)\}\leq \inf\limits _{\sigma \in {\cal P}(X)} \{\Phi_0(\sigma )+\T(\sigma, \mu)\}.
\]
In other words,
$T^-{\bar f}$ and $T^+\circ T^-{\bar f}$ are two conjugate functions verifying
\[
\int_XT^+\circ T^-{\bar f} d\mu=\int_XT^-{\bar f} d\mu=\Phi (\mu) \quad \hbox{for all $\mu\in {\cal N}$}.  
\]
Finally, to prove (\ref{unique}), let $g$ be any function in $C(X)$ such that $\int_X g d\mu=\Phi(\mu)$  for all $\mu \in \cal N$, then \begin{align*}
\int_XT^-{\bar f} d\mu \leq \Phi_0(\mu)=\sup \{\Phi (\sigma) -\T(\mu, \sigma); \sigma \in {\cal N}\}
\leq \sup \{\int_X gd\sigma -\T(\mu, \sigma); \sigma \in {\cal P}(X)\}
=\int_X T^-g d \mu.
\end{align*}
On the other hand, 
\begin{align*}
\int_XT^+\circ T^-{\bar f} d\mu &=\inf\{\int_XT^-{\bar f} d\sigma +\T(\sigma, \mu); \sigma \in {\cal P}(X)\}\\
&=\inf\{\int_XT^-{\bar f} d\sigma +\T(\sigma, \lambda)+\T(\lambda, \mu); \lambda \in {\cal N}, \sigma \in {\cal P}(X)\}\\
&=\inf\{\int_XT^+\circ T^-{\bar f} d\lambda  +\T(\lambda, \mu); \lambda \in {\cal N}\}\\
&=\inf\{\int_Xg d\lambda  +\T(\lambda, \mu); \lambda \in {\cal N}\}\\
&\geq \inf\{\int_Xg d\lambda  +\T(\lambda, \mu); \lambda \in  {\cal P}(X)\}\\
&=\int_XT^+g d\mu, 
\end{align*}
which completes the proof.
 }
\begin{corollary}
Let $\T$ be an idempotent  backward and forward linear transfer $\T$ such that 
\eqs{
\sup\limits_{\nu \in \mcal{P}(X)}\inf\limits_{\mu \in \mcal{N}}\T(\mu,\nu) < +\infty. 
}
Then, for any $f\in C(X)$, there is a function ${\bar f}\in C(X)$ such that 
\eq{
T^-{\bar f} \leq T^-f,\quad \quad T^+\circ T^-\bar f \geq T^+\circ T^-f, 
}
and 
\eq{
\int_XT^-\bar f d\mu=\int_XT^+\circ T^-\bar f d\mu \hbox{\quad for all measures $\mu$ with $\T(\mu, \mu)=0$}.
}
In other words, we have two conjugate functions $\psi_0:=T^-{\bar f}$ and $\psi_1=T^+\circ T^-\bar f$ such that 
\eqs{
\int_X\psi_0 d\mu=\int_X\psi_1 d\mu \hbox{\quad for all minimal measures $\mu$}.
}
\end{corollary}
\prf{It suffices to apply the above theorem to the functional $\Phi (\mu)=\int_XT^-f\, d\mu$ for any fixed $f\in C(X)$. Note that 
\begin{eqnarray*}
\Phi(\nu)-\Phi (\mu)&=&\int_XT^-f\, d\nu-\int_XT^-f\, d\mu\\
&=&\int_XT^-f\, d\nu-\int_Xf\, d\nu+\int_Xf\, d\nu -\int_XT^-f\, d\mu\\
&=&\int_XT^-f\, d\nu-\int_X(T^-)^2f\, d\nu+\int_Xf\, d\nu -\int_XT^-f\, d\mu\\
&\leq&
\T(\nu, \nu)+\T(\mu, \nu),
\end{eqnarray*}
in such a way that if $\nu\in {\cal N}$, then $\Phi(\nu)-\Phi (\mu) \leq \T(\mu, \nu).$ The conclusion follows from Theorem \ref{lipschitz} since being idempotent, $\T$ is Null-factorizable and $c(\T)=0$. 
}

\section{Weak KAM operators associated to Kantorovich operators}

\defn{Let $T: C(X) \to USC(X)$ be a backward Kantorovich operator with a finite Mather  constant $c(T)$. We shall say that a Kantorovich operator $T_\infty: C(X) \to USC(X)$ is {\bf a backward weak KAM operator associated to $T$} if
\begin{enumerate}
\item $T_\infty$ is idempotent.

\item $TT_\infty =T_\infty T$.

\item $T_\infty$ maps $C(X)$ to the class of backward weak KAM solutions for $T$, i.e., for any $g\in C(X)$, 
\begin{equation}
TT_\infty g +c(T)=T_\infty g.
\end{equation}
\end{enumerate}
The linear transfer associated to $T_\infty$ is then
$$
\T_\infty(\mu, \nu)=\sup\limits_{g\in C(X)}\{\int_Xg\, d\nu-\int_X T_\infty g\, d\mu\},
$$
and will be called the {\bf Peirls barrier associated to $T$}.
}
\begin{theorem}\lbl{powerbd} 
Let 
$T$ be a power-bounded Kantorovich operator such that for any $f\in C(X)$, 
 \eq{\lbl{unif}
\omega_{_{Tf}} \leq \omega_f,
 }
 where 
 $\omega_h$ is the modulus of uniform continuity of a continuous function $h$. Then, there exists a backward weak KAM operator $T_\infty: C(X) \to C(X)$ associated to $T$. 
   \end{theorem}
\prf{As suggested by the proof of Theorem \ref{gen}, the weak KAM operator could be defined for any $g\in C(X)$ as 
\eqs{
T_\infty g := \lim_n \uparrow T^n\circ \bar T_\infty g= \lim_n \uparrow T^n\circ (\lim_m\downarrow S_mg)=\lim_n \uparrow T^n\circ (\lim_m\downarrow \overline{\sup_{k\geq m}(T^kg +k c(T)))},
} 
provided we can show that the operator considered in the proof of Theorem \ref{gen},  
$\bar T_\infty g:=\lim_m\downarrow \overline{\sup_{k\geq m}(T^kg +k c(T))}$ is in $C(X)$ as opposed to $USC (X)$. Actually, in this case both $\bar T_\infty g$ and $T_\infty g$ are in $C(X)$.  

Indeed, first we note that for any $g\in C(X)$, $\overline{T} g(x) := \limsup_{n \to \infty}(T^ng(x) +n c(T))$ is in $C(X)$ since  $T$ is power bounded and the sequence $(T^ng+n c(T))_n$ is equicontinuous by assumption (\ref{unif}). This implies that $\bar T_\infty g=\overline{T}g$.
Moreover, $\overline{T}$ is itself a backward Kantorovich operator since it inherits the monotonicity, convexity, and affinity with respect to constants, properties of $T$. The remaining property to check is the lower semi-continuity. For this, suppose $g_k \to g$ in $C(X)$ for the sup norm. We have by Theorem \ref{capacity} that 
$T^ng(x) \leq T^ng_k(x) + \|g-g_k\|_\infty$,
 so that 
$\overline{T}g(x) \leq \overline{T}g_k(x) + \|g-g_k\|_\infty$
and $\overline{T}g(x) \leq \liminf_{k \to \infty}\overline{T}g_k(x)$.

Since $T{\overline T}g +c(\T) \geq {\overline T}g$, we have $\{T_{n}\circ \overline{T}g + nc(\T)\}_{n \geq 1}$ is a monotone increasing sequence of continuous functions, with 
$|T_{n}\circ \overline{T}g(x) - T_{n}\circ \overline{T}g(x')| \leq \delta(d_X(x,x')).$
 Since $T$ is power bounded, we have 
$\|T_{n}\circ \overline{T}g(x) + nc(\T)\|_\infty \leq K <+\infty$,
hence the operator $T_\infty: C(X) \to C(X)$ defined now via
\eqs{
T_\infty g(x) := \lim_{n \to \infty}T^n\circ\overline{T}g(x) + nc(\T),
}
is the claimed backward weak KAM operator associated to $T$. Indeed, as noted in Theorem \ref{gen}, we have for all $g\in C(X)$, 
\eq{\lbl{notreverse}
T\circ T_\infty g+c(T)=T_\infty g.
}
 We can also consider the composition 
\eq{\lbl{reverse_weakkam}
T_\infty \circ Tg(x) = \lim_{n \to \infty}( T^n \circ \overline{T}\circ T g(x) + nc(\T)),
}
and notice that
\as{\lbl{reverse}
\overline{T} \circ T g(x) = \limsup_{n \to \infty} (T^n \circ T g(x) + nc(\T))
= \limsup_{n \to \infty} (T_{n+1}g(x) + (n+1)c(\T)) - c(\T)
= \overline{T}g(x) - c(\T),
}
which means from \refn{reverse_weakkam} that
\eq{\lbl{reverse}
T_\infty\circ T g + c(\T) = T_\infty g.
}
We claim that $T_\infty$ is a backward Kantorovich operator. Indeed, since $T^n$ and $\overline{T}$ are backward Kantorovich operators, the monotonicity, convexity, and affine on constants properties hold in the pointwise limit $n \to \infty$, and therefore hold for $T_\infty$. Regarding lower semi-continuity,  we can repeat a similar argument as was given for $\overline{T}$ to deduce that 
\eq{\lbl{lower_semi_continuity_operator}
T^n\circ \overline{T} g(x) \leq T^n\circ \overline{T} g_k + \|g-g_k\|_\infty,
}
which gives
$T_\infty g(x) \leq \liminf_{k \to \infty}T_\infty g_k(x).
$

To show that $T_\infty$ is idempotent, use that $T^n\circ T_\infty g + nc(\T) = T_\infty g$ to obtain $\overline{T}\circ T_\infty g = T_\infty g$ and therefore  $T_\infty\circ T_\infty g = T_\infty g$.
Finally note that (\ref{reverse}) and (\ref{notreverse}) give that $T_\infty T=TT_\infty$, which completes the proof that $T_\infty$ is a backward weak KAM operator. 
}
\begin{remark} Another process for obtaining a weak KAM operator is to consider the Kantorovich operator $Ug=g\vee (Tg+c)$. Note that $(U^ng)_n$ is increasing to $U_\infty g$, that $TU_\infty g +c(T) \geq U_\infty g$ and therefore,
\eqs{
T_\infty g := \lim_n \uparrow T^n\circ U_\infty g= \lim_n \uparrow T^n\circ  (\lim_m \uparrow U^mg +nc(T)).
}
\end{remark}
\begin{remark}
Property (\ref{unif}) appears for example in the case where $X$ is the closure of an open bounded domain $O$ in $\C^n$ (resp., $\R^n$), and the Kantorovich operators are of the form 
 \[
Tf(x):= \sup_{v\in \R^n} \bigg\{ \int ^{2\pi}_0 f(x +
 e^{i\theta}v) {d\theta \over 2\pi};\,  x + \bar \Delta v\subset O \bigg\}, 
 \]
 (resp.,  
 \[
Tf(x) =\sup_{r \geq 0} \bigg\{ \int_{B} f (x + r y) \,dm(y);\,  x + r \overline{B}\} \subset O \bigg\}),
\]
where $\Delta = \{ z \in {\mathbb C}, \vert z\vert < 1\}$ is the open unit disc, 
$B$ is the open unit ball in $\R^n$ centered at $0$, and $m$ is normalized Lebesgue measure on $\R^n$. Condition (\ref{unif}) is verified in either case.

It is also satisfied by the ``Sinkhorn" operators, 
\eq{
T_\nu g(x)=\epsilon \log \int_{Y}e^{\frac{g(y)-c(x, y)}{\epsilon}}d\nu (y),
} 
where $\nu \in \P(X)$, which arises from an entropic regularization of the above optimal mass transport with cost $c$, as well as the composition $T_\nu\circ T_\mu$, where $\mu$ is another probability in $\P(X)$. We then have for $\kappa (\nu) <1$, 
\eq{
\|T_\nu g_1-T_\nu g_2\|_{o, \infty}\leq \kappa(\nu) \|g_1-g_2\|_{0, \infty}
}
and
\eq{
\|T_\mu\circ T_\nu g_1-T_\nu \circ T_\mu g_2\|_{o, \infty}\leq \kappa(\mu)\kappa(\nu) \|g_1-g_2\|_{0, \infty}
}
where $g_{o, \infty}:=\frac{1}{2}(\sup g-\inf g).$

We will now see that (\ref{unif})  also holds when the linear transfer associated to $T$ is weak$^*$-continuous on $\P(X)\times \P(X)$. 
\end{remark}


\thm{ \label{weakKAMthm_discrete_time} Let $\T$ be a weak$^*$-continuous backward linear transfer on $\mcal{M}(X)\times \mcal{M}(X)$ 
with backward Kantorovich operator $T$, and a Mather constant $c(\T)$. 
Then, there exists a backward weak KAM operator $T^-_\infty: C(X) \to C(X)$
with a corresponding backward linear transfer $\T_\infty$ satisfying the following: 
\begin{enumerate}
\item 
$(\T_n - nc(\T)) \star \T_\infty = \T_\infty  =   \T_\infty\star(\T_n - nc(\T)) \quad \text{for every $n \in \N$}.
$
\item For every $\mu,\nu \in \mcal{P}(X)$, we have 
\begin{equation*}
\sup\lf\{\int_X T^-_\infty g d (\nu-\mu)\,;\, g \in C(X)\rt\} \leq\T_\infty(\mu,\nu) \leq \liminf_{n\to\infty}(\T_n(\mu,\nu)- nc(\T)).
\end{equation*}
\item  
The set ${\cal A}:=\{\sigma \in {\cal P}(X); {\cal T}_\infty(\sigma, \sigma)=0\}$ contains all minimal measures of $\T$. Moreover, $\mu$ is minimal for $\T$ if and only if $(\mu, \mu)$ belongs to the set 
\eqs{\mathcal{D}:= \{ (\mu,\nu) \in \mathcal{P}(X)\times\mathcal{P}(X);\, \T(\mu,\nu) + \T_\infty(\nu,\mu) = c(\T)\}.
}
\item  If $\T$ is also a forward linear transfer, then there exists conjugate functions $\psi_0, \psi_1$ for $\T_\infty$ in the sense that 
$\psi_0=T_\infty^-\psi_1$ and $\psi_1=T_\infty^+\psi_0$,
 such that 
\eq{
T^-\psi_0+c=\psi_0, \quad T^+\psi_1-c=\psi_1,
}
and 
\eq{
\int_X \psi_0d\mu=\int_X \psi_1d\mu \hbox{\,  for every  $\mu\in \cal A$}.
}
 \end{enumerate}
}
\prf{We shall show that in this case $\T$ has bounded oscillation and that $T$ satisfies condition (\ref{unif}). 
For that we recall that since $X$ is compact with a metric $d_X$, the \textit{quadratic Wasserstein distance}, denoted by $W_2(\mu,\nu) := \sqrt{\T_2(\mu,\nu)}$ where $\T_2$ is optimal transport with quadratic cost $c(x,x') := d_X(x,x')^2$, metrizes the topology of weak$^*$ convergence on $\mcal{P}(X)$. Therefore, if $\T$ is weak$^*$ continuous on $\mcal{P}(X)\times \mcal{P}(X)$, then it is weak$^*$ uniformly continuous 
and there exists a \textit{modulus of continuity} $\delta: [0,\infty) \to [0,\infty)$, $\delta(0) = 0$, such that 
\eqs{
|\T(\mu,\nu) - \T(\mu',\nu')| \leq \delta(W_2(\mu,\mu') + W_2(\nu,\nu')) \hbox{ for all $\mu,\mu',\nu,\nu' \in \mcal{P}(X)$}.
}
We need the following lemmas. The first is an adaptation of an argument by Bernard-Buffoni \cite{BB2}.
}
\lma{\label{estimation}
Let $\T:\mcal{P}(X)\times \mcal{P}(X) \to \R$ be a weak$^*$ continuous functional with a modulus of continuity $\delta$. Then, the sequence $\{\T_n\}_{n \in \N}$ is equicontinuous with the same modulus of continuity $\delta$, and $\T$ has bounded oscillation, i.e., there exists a positive constant $C > 0$ such that

\eqs{
\lf|\T_{n}(\mu,\nu) - nc(\T)\rt| \leq C,\quad \text{for every $n \geq 1$ and all $\mu,\nu \in \mcal{P}(X)$.}
}
}
\prf{
Define $M_n := \max_{\mu,\nu}\T_n(\mu,\nu)$ and $M := \inf_{n\geq 1}\{\frac{M_n}{n}\} > -\infty$. The sequence $\{M_n\}_{n \geq 1}$ is subadditive, that is $M_{n+m} \leq M_n +M_m$, hence $\{\frac{M_n}{n}\}_{n \geq 1}$ decreases to its infimum $M$ as $n \to \infty$. 
On the other hand, the sequence $m_n := \min_{\mu,\nu}\T_n(\mu,\nu)$ is superadditive, hence  $\lim_{n \to \infty}\frac{m_n}{n} = m$, where $m := \sup_{n}\frac{m_n}{n}$.
  We now show that $m = M$. 
  
  Note that all $\T_n$ have the same modulus of continuity; this follows because for each choice of $(\sigma_1,\ldots,\sigma_n) \in \mcal{P}(X)\times\ldots\times\mcal{P}(X)$, the functional 
  $$(\mu,\nu) \mapsto \T(\mu, \sigma_1) + \T(\sigma_1,\sigma_2) + \ldots + \T(\sigma_n, \nu)$$ has the same modulus of continuity $\delta$, and so the infimum (i.e. the function $\T_n$) also has modulus of continuity $\delta$. 
  This then implies the existence of a constant $C > 0$, such that $M_n - m_n \leq C$ for every $n \geq 1$. Then, we obtain the string of inequalities
\eqs{
nM - C \leq M_n - C \leq m_n \leq \T_n(\mu,\nu) \leq M_n \leq m_n + C \leq nm + C.
}
The left-most and right-most inequalities imply $M \leq m$ upon sending $n \to \infty$, hence $m  = M$.
}
\lma{\label{Kant} Let $\T:\mcal{P}(X)\times \mcal{P}(X) \to \R$ be a weak$^*$ continuous backward linear transfer with a modulus of continuity $\delta$, and let $T$ be the corresponding Kantorovich operator. Then, 
\begin{enumerate}
\item $T$ is power bounded.
\item The semi-group of operators $\{T^n\}_{n \geq 1}$ has the same modulus of continuity as $\T$.
\item The constant $c(\T)$ is critical in the sense that for any $g \in C(X)$, $T^ng + kn \to \pm \infty$ as $n \to \infty$ if $k \neq c(\T)$, depending on whether $k < c(\T)$ or $k > c(\T)$, .
\end{enumerate}
}

\prf{ 1) follows from the fact that $\T$ has bounded oscillation and Lemma \ref{when.bounded}.

For 2) note that 
\begin{align*}
T^n g(x) &= \sup_{\sigma}\{\int g d\sigma - \T_n(\delta_x, \sigma)\} \\
&\leq \sup_{\sigma}\{\int g d\sigma - \T_n(\delta_y, \sigma)\} + \sup_{\sigma}\{\T_n(\delta_y, \sigma) - \T_n(\delta_x, \sigma)\}\\
&= T^n g(y) + \delta(d(x,y)).
\end{align*}
We then interchange $x$ and $y$ to obtain the reverse inequality.

3) follows from 1) since 
$C + (k-c(T))n \leq T^n g(x) + kn \leq C + (k-c(T))n.
$
}
\noindent{\bf Proof of Theorem \ref{weakKAMthm_discrete_time}:} 
 First note that Lemma \ref{Kant} shows that the hypothesis of Theorem \ref{powerbd} are satisfied, hence there exists a Kantorovich operator $T_\infty$ associated with $T$. We consider the corresponding Peirls barrier, 
\eqs{
\T_\infty(\mu,\nu) := \sup_{g \in C(X)}\lf\{\int_{X} gd\nu - \int_{X} T_\infty gd\mu \rt\}
}
which is a backward linear transfer. Since $T^n\circ T_\infty g + nc(\T) = T_\infty g$ and $T_\infty\circ T_n g + nc(\T) = T_\infty g$, we obtain 
$\T_\infty = (\T_n - nc(\T))\star\T_\infty = \T_\infty\star(\T_n - nc(\T))$ for all $n \geq  1$.

$2)$  Recall that $T_\infty g(x) \geq \limsup_{n \to \infty}(T^ng(x) + nc(\T))$, so that
\begin{align*}
\int_{X} T_\infty g d\mu \geq \int_{X}\limsup_{n \to \infty}(T^ng(x) +n c(\T))d\mu
\geq \limsup_{n \to \infty}\int_{X}(T^ng(x) + nc(\T))d\mu.
\end{align*}
Hence
\begin{align*}
\T_\infty(\mu,\nu) 
&\leq \sup\limits_{g \in C(X)}\liminf_{n \to \infty}\left\{\int_{X}gd\nu - \int_{X}T^n gd\mu - nc(\T)\right\}\\
&\leq \liminf_{n \to \infty}\sup\limits_{g \in C(X)}\left\{\int_{X}gd\nu - \int_{X}T^n gd\mu - nc(\T)\right\}\\
&= \liminf_{n \to \infty}(\T_n(\mu,\nu)- nc(\T)).
\end{align*}
On the other hand, since $T_\infty$ is idempotent, 
\begin{align*}
\T_\infty(\mu,\nu) = \sup\lf\{\int_{X}gd\nu - \int_{X}T_\infty gd\mu\,;\, g \in C(X)\rt\}
\geq \sup\lf\{\int_{X}T_\infty gd(\nu-\mu)\,;\, g \in C(X)\rt\}.
\end{align*}

$3)$   Suppose $\mu$ is a minimal measure. i.e., $c(\T) = \T(\mu,\mu)$, then   \as{
0 \leq \T_\infty(\mu,\mu) \leq \liminf_{n \to \infty}(\T_n(\mu,\mu)-nc(\T))
\leq \liminf_{n \to \infty}(n\T(\mu,\mu)-nc(\T)) = 0,
}
so $\mu \in \mcal{A}$.  If $(\mu,\mu) \in \mcal{D}$, then clearly $\mu \in \mcal{A}$ and $\T(\mu,\mu) = c(\T)$. The converse is also clear.

$4)$ is an immediate application of Theorem \ref{lipschitz}. 

\begin{corollary}
Let $\T$ be an optimal mass transport with a continuous cost $A(x, y)$ on $X\times X$, that is 
$
\T(\mu, \nu)=\inf\{\int_{X\times X} A(x, y) \, d\pi; \pi \in \mcal K(\mu, \nu)\}. 
$
Then, 
\begin{enumerate}
\item The associated Mather constant is given by,
\begin{equation}\label{mather1001}
c(\T) = \min \{ \int_{X \times X}A(x,y)d \pi ; \pi \in {\mathcal P} (X\times X), \pi_1=\pi_2\}.
\end{equation}
\item For each $n\in \N$, $\T_n$ is the optimal mass transport associated with the cost $$A_n(x, y)=\inf\{A(x, x_1)+A(x_1, x_2)+\, ..., A(x_{n-1}, y); x_1,..., x_{n-1} \in X\}.$$ 
\item The corresponding Peirls barrier is given by 
$$\T_\infty (\mu, \nu)=\T_{A_\infty} (\mu, \nu):=\inf\{\int_{X\times X}A_\infty(x,y)d\pi(x,y)\,;\, \pi \in \mcal{K}(\mu,\nu)\},$$
and the associated weak KAM operators are 
$$T^-_\infty f(x) = \sup\{f(y) - A_\infty (x,y)\,;\, y \in X\}\quad \text{and }\quad T^+_\infty f(y) = \inf\{f(x) + A_\infty (x,y)\,;\, x \in X\},$$
where $A_\infty(x,y) := \liminf\limits_{n \to \infty}(A_n(x,y) - c(\T)n)$ is a continuous cost on $X\times X.$
\item The set ${\mathcal A}:=\{\sigma \in {\mathcal P}(X); {\T}_\infty (\sigma, \sigma)=0\}$ consists of those $\sigma \in {\mathcal P}(X)$ supported on the set $D_\infty=\{x\in X; A_\infty (x, x)=0\}$. 

\item  The minimizing measures in (\ref{mather1001}) are all supported on the set 
$$D:= \{ (x,y) \in X \times X\,;\, A(x,y) + A_\infty(y,x) = c(\T)\}.$$

\item  There exists conjugate functions $\psi_0, \psi_1$ for $\T_\infty$ in the sense that $
T^-_\infty \psi_1=\psi_0$ and $ T^+_\infty \psi_0=\psi_1$,
which are also weak KAM solutions for $\T$, i.e.,
\eq{
T^-\psi_0+c=\psi_0, \quad  T^+\psi_1-c=\psi_1,
}
and such that 
\eq{
\psi_0(x)=\psi_1(x) \hbox{\,  {\rm whenever}  $A_\infty(x, x)=0$}.
}

\end{enumerate}
\end{corollary}
\prf{ $1)$ If $\T$ is an optimal mass transport with cost $A$, then (\ref{mather1001}) follows immediately from the fact that $c(\T)=\inf_{\mu \in \P(X)}\T(\mu, \mu)$. 

$2)$ and $3)$: It is easy to see that its $n$-th convolution $\T_n$ and the optimal mass transport with cost $A_n$ have the same backward Kantorovich operator given by  $T_n^- g(x) = \sup\{ g(y) - A_n(x,y)\,;\, y \in X\}$. Hence, they are identical. We now show that 
\eq{\lbl{first}
\overline{T}^- g(x) := \limsup_{n \to \infty}(T_n^- g(x)+c(\T)n)=T_{A_\infty}^-g(x) := \sup_{y \in X}\{g(y) - A_\infty(x,y)\},
}
the latter being 
 the backward Kantorovich operator for the optimal transport with cost $A_\infty$.
 Indeed, first note that
\begin{equation}\label{masstransport1}
\limsup_{n}(T_n^- g(x)+c(\T)n) \geq \sup_{y \in X}\{g(y) - A_\infty(x,y)\}= T_{A_\infty}^-g(x).
\end{equation}
On the other hand, let $y_n$ achieve the supremum for $T^n g(x) = \sup\{g(y) - A_n(x,y)\,;\, y \in X\}$, and let $(n_j)_j$ be a subsequence   such that $\lim_{j \to \infty}(T_{n_j}^-g(x) + c(\T)n_j) = \limsup_{n}(T_n f(x) + c(\T)n)$. By refining to a further subsequence, we may assume by compactness of $X$, that $y_{n_j} \to \bar{y}$ as $j \to \infty$. From the equi-continuity of the $A_n$'s, we deduce that
\begin{equation}
\limsup_{n}(T^n g(x)+c(\T)n)  = \lim_{j \to \infty}(T_{n_j}^-g(x)+c(\T)n_j) = g(\bar{y}) - \liminf_{j}(A_{n_j}(x, \bar{y})-c(\T)n_j).
\end{equation}
As $\liminf_{j}(A_{n_j}(x, \bar{y})-c(\T) n_j) \geq \liminf_{n}(A_n (x,\bar{y})-c(\T) n) = A_\infty(x, \bar{y})$, we obtain
\begin{equation}\label{masstransport2}
\limsup_{n}(T^n g(x) + nc(\T)) \leq g(\bar{y}) - A_\infty(x, \bar{y}) \leq \sup_{y}\{g(y) - A_\infty(x,y)\}=T_{A_\infty}^-g(x).
\end{equation}
Inequality (\ref{masstransport2}) holds for every subsequence $(n_k)_k$ going to $\infty$, hence 
$\limsup_{n}(T_n^- g(x) + c(\T)n) \leq T_{A_\infty}^-g(x)$. By combining this with (\ref{masstransport1}) we get the equality (\ref{first}).  

Now note that $T_m^-(\limsup_{n}(T_n^- g + c(\T)n))(x) + c(\T)m = T_m^- T_{A_\infty}^- g(x) + c(\T)m = T_{A_\infty}^-g(x)$ thanks to the fact that $A_m\star A_\infty = A_\infty$. Recalling that the weak KAM operator is $T^-_\infty g=\lim_mT^m\circ {\overline T}^-g$, we conclude 
that $T^-_\infty g(x) = T_{A_\infty}^-g(x)$.

$4)$  Note that since $T_\infty$ is idempotent, then 
\[
c(\T_\infty)=\inf_{\sigma \in \P(X)}\T_\infty(\sigma, \sigma)=\inf\{\int_{X\times X} A_\infty (x, y)\, d\pi; \pi\in \P(X\times X), \pi_1=\pi_2\}=0. 
\] 
If $\mu$ is not supported on the set $\{x\in X; A_\infty(x, x)=0\}$, then if $A_x=\{y\in X; A_\infty (x, y)>0\}$, we have $\mu (\{x\in X; \mu(A_x)>0\})>0$, which readily implies that 
$\T_\infty(\mu, \mu)= \int_{X\times X} A_\infty (x, y)\, d\mu \otimes \mu>0$.

$5)$ Since $T_\infty$ is idempotent, $\int_{X\times X} A_\infty (x, y)\, d\pi \geq 0$ for all  
$\pi\in \P(X\times X)$, with $\pi_1=\pi_2$. Hence, if such an $\eta$ is supported on $D$, then 
\[
\int_{X\times X} A (x, y)\, d\eta=-\int_{X\times X} A_\infty (x, y)\, d\eta +c(\T)\leq c(\T),
\]
which means that $\int_{X\times X} A (x, y)\, d\eta=c(\T)$ and $\eta$ minimizes  (\ref{mather1001}).

Conversely, a result of Bernard-Buffoni \cite{BB2} yields that if 
$\eta$ is such that $\eta_1=\eta_2$ minimizes (\ref{mather1001}), then it is supported on the set $D=\{(x, y); A(x, y)+A_\infty (y, x)=0\}$.

$6)$ follows from Theorem \ref{lipschitz}. 
}

\ex{[Iterates of power costs]  Let $X \subset \R^n$ be compact and convex, and $c(x,y) := |x-y|^p$ for $p > 0$ and $x,y \in X$. Let $\T_c$ denote the optimal transport with cost $c$, with corresponding backward Kantorovich operator $T_c^- g(x) = \sup_{y \in X}\{g(y) - c(x,y)\}$.

If $0 < p \leq 1$, then $c$ satisfies the reverse triangle inequality $c(x,z) + c(z,y) \geq c(x,y)$, hence $c \star c (x,y) = \inf\{ |x-z|^p + |z-y|^p\,;\, z \in X\} = c(x,y)$. Therefore $T_\infty g(x) = T_c g(x)$ (i.e. $T_c^-$ is itself idempotent).

If $p > 1$, then $c \star c (x,y) = \inf\{ |x-z|^p + |z-y|^p\,;\, z \in X\}$ is minimised at some point $z = (1-\lambda)x + \lambda y$ on the line between $x$ and $y$, so that $c_p \star c_p(x,y) = \left(\lambda^p + (1-\lambda)^p\right)|x-y|^p$. The optimal $\lambda$ is $\frac{1}{2}$. Hence, 
$$(T_c^-)^n g(x) =\sup_{y \in X}\{ g(y) - \frac{1}{n^{p-1}}|x-y|^p\}.$$ 
Therefore, when $n \to \infty$, $(T_c^-)^n g(x) \to \sup_{x \in X}g(x)$, and it follows that $T_\infty g (x) := \sup_{x \in X}g(x)$, with the corresponding backward linear transfer $\T_\infty$ being the null transfer. 
}
\prop{\label{distance} If $\T$ is a distance-like regular backward linear transfer (i.e., $\T \leq \T\star\T$), then $c(\T)\geq 0$. If $c(\T)=0$, then the corresponding Kantorovich operator $T$ has a weak KAM operator given for every $g\in C(X)$ by 
\eq{T_\infty g:=\lim_n \downarrow T^ng.
}
}
\prf{Since $\T\leq \T\star \T$, we have  $\T (\mu, \nu)\leq \T_n(\mu, \nu)$ for all $(\mu, \nu) \in \P(X)\times \P(X)$, hence
\[
0=\lim_{n\to \infty}\frac{\inf_{\mu, \nu}\T (\mu, \nu)}{n} \leq \lim_{n\to +\infty} \frac{\inf_{\mu, \nu}\T_n(\mu, \nu)}{n}=c(\T).
\]
We also have $T^2 \leq T$, so by assuming $c(\T)=0$, we get that for any $g\in C(X)$, $(T^ng)_n$ is a decreasing sequence and Lemma \ref{oscillation_limits} gives that for every $n$, 
 $\sup_{x \in X} T^n g(x) \geq \inf_{y\in X}g(y)$.
Since $T$ is a capacity, it follows that $T_\infty g=\lim_n\downarrow T^ng$ is a backward weak KAM operator for $T$. 
}
\thm{\label{inf=c} Let $\T$ be a backward linear transfer such that 
$c(\T)=\inf\limits_{\P(X)\times \P(X)}\T$. Then the corresponding Kantorovich operator $T$ has a weak KAM operator given for every $g\in C(X)$ by 
\eq{T_\infty g=\lim_n\downarrow T^n\hat g+nc(\T),}
 where  
 \eq{\lbl{A}
\hat g (x):=\inf \{ h(x); h\in -\A, h\geq g\}\,\,\hbox{ and \,\, $\A:=\{h\in LSC(X);  T_r(-h)\leq -h\},$ }
}
$T_r$ being the positively $1$-homogenous recession operator associated with $T$, 
$$T_rg(x)=\lim\limits_{\lambda \to +\infty}\frac{T (\lambda f)(x)}{\lambda}=\sup\limits_{\sigma \in \P(X)}\{\int_Xg\, d\sigma; (\delta_x, \sigma)\in D(\T)\}.$$
}

\prf{Following \cite{G}, we note that if $k:=\inf \T$, then $\T-k \geq 0$ and therefore 
\as{
Tg(x)+k&= \sup\limits_{\sigma \in \P(X)}\{\int_Xg\, d\sigma -\T(\delta_x, \sigma)+k\}\\
&\leq \sup\limits_{\sigma \in \P(X)}\{\int_Xg\, d\sigma; \T(\delta_x, \sigma)<+\infty\}\\
&=\sup\limits_{\sigma \in \P(X)}\{\int_Xg\, d\sigma; (\delta_x, \sigma)\in D(\T)\}\\
&=T_rf(x).
}

Consider now the closed convex cone $\A$ defined in (\ref{A}), and 
note that $\A$ contains the constants and is stable under finite maximum. Moreover,  
\[
Tg(x)+k\leq T_rg(x) 
\leq \inf \{ h(x); h\in -\A, h\geq g\}:=\hat g(x).
\]
Indeed, if $\phi \in -\A$ and $\phi\geq g$, then 
\[
Tg(x)+k\leq T\phi (x)+k\leq  
T_r\phi (x) \leq \phi (x),
\]
hence $Tg(x)+k\leq \phi (x)$ from which follows that $Tg+k\leq \hat g$. Since $g\to \hat g$ is clearly idempotent, we get that $T\hat g +k\leq \hat g$.

Assuming now $k=c(T)$, then  $T_\infty g:=\lim_n\downarrow T^n\hat g +nc \leq \hat g$ is  the claimed weak KAM operator associated to $T$. Note that by Lemma \ref{weak_KAM_for_monotone_decreasing}, $T_\infty g$ is proper and hence is in $USC(X)$. 
}

\defn{\lbl{transfer.sets_defn} 1) A subset ${\mathcal S}$ of ${\mathcal P}(X)\times {\mathcal P}(Y)$ is said to be a {\em backward transfer set} if its characteristic function 
\begin{equation}
\T(\mu,\nu) := \begin{cases} 0 & \text{if } (\mu,\nu) \in \S,\\
+\infty & \text{otherwise},
\end{cases}
\end{equation}
is a backward linear transfer. The corresponding Kantorovich operator is then the  positively $1$-homogenous map $Tf(x)=\sup\{\int_Xf\, d\sigma; (\delta_x, \sigma)\in \S\}$. 

2) ${\mcal S}$ is said to be {\em transitive} if 
 \begin{equation}
 (\mu, \sigma)\in \S \,\, {\rm and}\,\, (\sigma, \nu)\in \S \Rightarrow (\mu, \nu)\in \S.
 \end{equation}
 }

\begin{corollary} Let $T$ be a positively $1$-homogenous Kantorovich operator and let $\S$ be the corresponding transfer set in ${\mathcal P}(X)\times {\mathcal P}(Y)$.\begin{enumerate}
\item If there exists $\mu \in \P(X)$ such that $(\mu, \mu)\in \S$, then its weak KAM operator is given by $T_\infty g:=\lim_n\downarrow T^n\hat g$,
 where $\hat g$  
 and  $\A$ are defined in (\ref{A}).
 
 \item If $(\mu, \mu)\in \S$ for all $\mu\in \P(X)$, then  $g\leq S_\infty g:= \lim_n \uparrow T^ng\leq T_\infty g:=\lim_n\downarrow T^n\hat g\leq \hat g (x)$.

 \item If in addition, $\S$ is transitive, then $Tg=S_\infty g=T_\infty g=\hat g$ is itself idempotent.
\end{enumerate}
\end{corollary}
\prf{1) follows immediately from Proposition \ref{inf=c} since the hypothesis implies that $c(\T)=\inf \T =0$. The hypothesis for (2) implies that $Tf(x):=\sup\{\int_Xfd\sigma; (x, \sigma) \in \S\}\geq f(x)$, hence $T^ng$ is increasing. However, note that $S_\infty$ produces weak KAM solutions in 
$USC_\sigma(X)$ and not in $USC(X)$, while $T_\infty$ is a true weak KAM operator according to our definition. 
 
The transitivity condition in 3) yields that $\T\leq \T^2$, hence $T\geq T^2$, which once combined with 2) yields that $T$ itself is idempotent and $Tg=\hat g$. Note that in this case, it was shown in \cite{G} that 
\[
\S=\{(\mu, \nu) \in \P(X)\times \P(X);\, \mu \prec_{\mcal A} \nu\},
\]
where $\A$ is the cone defined in (\ref{A}).
}
Besides the case of transfer sets, 
the assumption $\inf\{\T(\mu,\nu); \mu,\nu \in \mcal{P}(X)\} = \inf\{\T(\mu,\mu)\,;\, \mu \in \mcal{P}(X)\}$ is satisfied in a number of settings. 

\ex{\lbl{the_convex_energy_case}
 Let $\T$ be the backward linear transfer associated to the convex energy $I(\nu)$, that is $\T (\mu, \nu):=I(\nu)$ and $Tg(x) = I^*(g)$, where $I^*$ is the Legendre transform of $I$.   It is immediate that $\inf_{\mu,\nu}\T(\mu,\nu) = \inf_{\mu}\T(\mu,\mu)=c(\T)$, 
and the operator $T_\infty g (x) := I^*(g)+c(\T)$ is the corresponding backward weak KAM operator. }

\ex{\lbl{pt} Consider the example of a continuous point transformation 
\eqs{
\T(\mu,\nu) = \begin{cases} 0 & \text{if } \nu = F_{\#}\mu,\\
+\infty & \text{otherwise,}
\end{cases}
}
with $Tg(x)=g( F(x))$ for a continuous map $F: X \to X$. The Krylov-Bogolyubov theorem yields a measure $\mu \in \mcal{P}(X)$ such that $F_{\#}\mu = \mu$, hence $c(\T)<+\infty$, and $\inf_{\mu,\nu}\T(\mu,\nu) =c(\T)=0$. In this case, Theorem \ref{inf=c} applies and a weak KAM operator is given by $T_\infty f (x)=\lim_n\downarrow {\hat f}(F^n(x))$, where $\hat f$ is the upper envelope with respect to the cone  
$
-\A=\{g\in USC (X); g(x) \geq g(F(x))\}.
$

Assuming $X \subset \R$, there is another formulation for the corresponding weak KAM operator. Indeed, consider $F^\infty(x) := \limsup_{m \to \infty}F^m(x)$ and  
\eqs{
S_\infty g(x) = \overline{g\circ F^\infty}(x),
}
where this time $\overline f$ is the upper envelope of $f$ with respect to the cone $USC(X)$. To show that $S_\infty$ is a weak KAM operator for $\T$, write  
\as{
T \circ S_\infty g(x) &= \overline{g \circ F^\infty}(F(x))\\
&= \inf\{h(F(x))\,;\, h \in USC(X),\, h \geq g\circ F^\infty\}\\
&= \inf\{h(x)\,;\, h \in USC(X),\, h \geq g\circ F^\infty\}\\
&= \overline{g \circ F^\infty}(x) = S^\infty g(x),
}
where we have used the fact that if $h \in USC(X)$ with $h \geq g\circ F^\infty$, then $h \circ F \in USC(X)$ with $h \circ F \geq g \circ F^\infty$ (here we use the fact that $F^\infty \circ F = F^\infty$). It is not hard to show that $T_\infty=S_\infty$. 

However, if we don't require that the weak KAM operator maps into upper semi-continuous functions, then we can have several candidates. For example, consider the operator $S^\infty_1g(x)=g\circ F^\infty(x)$. Although idempotent and satisfies $T \circ (g\circ F^\infty) = g\circ F^\infty$, it is not always valued in $USC(X)$.
For example taking $X = [0,1]$, $F(x) = x^2$, then
\eqs{
g \mapsto g \circ F^\infty(x) = \begin{cases} g(0) & \text{if }x \in [0,1)\\
g(1) & \text{if }x = 1,
\end{cases}
}
is not in $USC(X)$ unless $g(0) \leq g(1)$.
Another idempotent operator mapping into backward weak KAM solutions but not in $USC(X)$ can be considered when $F^m(x)$ has more than one subsequential limit. It is $S^\infty_2g(x)=g\circ F_\infty(x)$, where $F_\infty(x)=\liminf_{m \to \infty}F^m(x)$. Note for example that if 
$X = [0,1]$ and $F(x) = 1-x$, then $F^\infty (x) = \max\{x, 1-x\}$, while $F_\infty(x) = \min\{x, 1-x\}$.
}
\ex{\lbl{ergodic_maximization}
For any lower semi-continuous $A: X \to \R$, consider $T g(x) = g(x) - A(x)$ and the corresponding
\eqs{
\T(\mu, \nu) = \begin{cases}
\int_{X}Ad\mu & \text{if $\nu = \mu$ and $\int_{X}Ad\mu < +\infty$,}\\
+\infty & \text{otherwise.}
\end{cases}
}
We again have $c(\T) = \inf\limits_{\mu,\nu \in \mcal{P}(X)}\T(\mu,\nu) =  \inf\limits_{x \in X}A(x),$ 
and the corresponding weak KAM operator is 
\as{
T_\infty g(x) = \begin{cases}
g(x) & \text{if } A(x) = c(\T)\\
-\infty & \text{otherwise.}
\end{cases}
}
The corresponding backward linear transfer $\T_\infty$ is
\eqs{
\T_\infty(\mu,\nu) = \begin{cases}
0 & \text{if $\mu= \nu$ is supported on the set $\{x \in X\,;\, A(x) = c(\T)$\}}\\
+\infty & \text{otherwise}
\end{cases}
}
and $\mcal{A} = \{\mu \in \mcal{P}(X)\,;\, \text{spt}(\mu) \subset \{x\,;\, A(x) = c(\T)\}\}$.

Note that if $T g(x) = g \circ F(x) - A(x)$ for a continuous $A$, then the resulting linear transfer
\eqs{
\T(\mu, \nu) = \begin{cases}
\int_{X}Ad\mu & \text{if $\nu = F_{\#}\mu$,}\\
+\infty & \text{otherwise,}
\end{cases}
}
 fails in general the assumption  $\inf_{(\mu,\nu)}\T(\mu,\nu) = \inf_{\mu}\T(\mu,\mu)$. The latter minimisation, $\inf_{\mu}\T(\mu,\mu)$, is of interest in ergodic optimisation for expanding dynamical systems. We discuss more on this problem in the last section. 
}

\section{Semi-groups of Kantorovich operators and weak KAM theories}

Theorem \ref{weakKAMthm_discrete_time} holds when we have an appropriate semi-group of backward linear transfers $(\T_t)_{t > 0}$. In this case, we do not build a discrete-time semi-group $(\T_n)_{n \in \N}$ from one linear transfer $\T$, but instead already have in hand an appropriate semi-group $(\T_t)$. In particular, if $\{\T_{t}\}_{t \geq 0}$ is a family of backward linear transfers on $\mcal{P}(X)\times \mcal{P}(X)$ with associated Kantorovich operators $\{T_t\}_{t \geq 0}$, where $\T_0$ is the identity transfer, 
\as{
\T_0(\mu,\nu) = 
\begin{cases}
0 & \text{if } \mu = \nu \in \mcal{P}(X)\\
+\infty & \text{otherwise.}
\end{cases}
}
then under the following assumptions
\enum{
\item[(H0)] The family $\{\T_{t}\}_{t \geq 0}$  is a semi-group under inf-convolution: $\T_{t+s} = \T_t\star \T_s$ for all $s,t\geq 0$.
\item[(H1)] For every $t > 0$, the transfer $\T_t$ is weak$^*$-continuous, and the Dirac measures are contained in $D_1(\T_t)$.
\item[(H2)] For any $\epsilon > 0$, $\{\T_t\}_{t \geq \epsilon}$ has common modulus of continuity $\delta$ (possibly depending on $\epsilon$).
}
The results of Theorem \ref{weakKAMthm_discrete_time} and their proofs hold (with appropriate changes from $n$ to $t$). We reproduce them here without a proof in the case of a semi-group of standard mass transports that correspond to a semi-group of cost functionals. 

\prop{\lbl{equicont_cost}
 Let $A_t(x, y)$ be a semi-group of equicontinuous cost functions on $X\times X$, that is 
\begin{equation*}
A_{t+s}(x, y)=A_t\star A_s (x, y):=\inf\{A_t(x, z)+A_s(z, y); z\in X\}, 
\end{equation*}
and 
consider the associated optimal mass transports 
\begin{equation*}
\T_t(\mu,\nu) = \inf\{\int_{X\times X} A_t(x,y)d\pi(x,y)\,;\, \pi \in \mcal{K}(\mu,\nu)\}.  
\end{equation*}
\begin{enumerate} 
\item The family $(\T_t)_t$ then forms a semi-group of linear transfers for the convolution operation ( i.e., $\T_{t+s}=\T_t\star \T_s$ for $s, t \geq 0$), that is equicontinuous on ${\mathcal P}(X)\times {\mathcal P}(X)$. One can therefore associate a weak KAM operator $T_\infty$ and a Peirl barrier $\T_\infty$ such that the following holds: 
\item The Mather consatnt $c=c((\T_t)_t): = \lim\limits_{t \to \infty}\frac{\inf_{\mu,\nu \in \mcal{P}(X)}\T_t(\mu,\mu)}{t}$ satisfies:
\begin{equation}\label{mather1000}
c= \min \{ \int_{X \times X}A_1(x,y)d \pi ; \pi \in {\mathcal P} (X\times X), \pi_1=\pi_2\}.
\end{equation}
\item The Peirl barrier $\T_\infty$ satisfies 
 \begin{equation*}
\T_\infty (\mu, \nu)=\T_{A_\infty} (\mu, \nu):=\inf\{\int_{X\times X}A_\infty(x,y)d\pi(x,y)\,;\, \pi \in \mcal{K}(\mu,\nu)\},  
\end{equation*}
where $A_\infty(x,y) := \liminf\limits_{t \to \infty}(A_t(x,y) - ct)$ is a continuous on $X\times X$.
\item The corresponding weak KAM operators are defined as
 \begin{equation*}
T^-_\infty f(x) = \sup\{f(y) - A_\infty (x,y)\,;\, y \in X\}\, \text{and }\, T^+_\infty f(y) = \inf\{f(x) + A_\infty (x,y)\,;\, x \in X\}.
\end{equation*}

\item The set ${\mathcal A}:=\{\sigma \in {\mathcal P}(X); {\T}_\infty (\sigma, \sigma)=0\}$ consists of those $\sigma \in {\mathcal P}(X)$ supported on the set $A=\{x\in X; A_\infty (x, x)=0\}$. 
\item The minimizing measures in (\ref{mather1000}) are all supported on the set 
$$D:= \{ (x,y) \in X \times X\,;\, A_1(x,y) + A_\infty(y,x) = 0\}.$$

\item  There exist continuous conjugate functions $\psi_0, \psi_1$ for $\T_\infty$ in the sense that 
$
T^-_\infty \psi_1=\psi_0$ and $ T^+_\infty \psi_0=\psi_1$,
which are weak KAM solutions for $T_1$, i.e.,
\eq{
T_1^-\psi_0+c=\psi_0, \quad  T_1^+\psi_1-c=\psi_1,
}
and such that 
\eq{
\psi_0(x)=\psi_1(x) \hbox{\,  {\rm whenever}  $A_\infty(x, x)=0$}.
}

\end{enumerate}

} 
 \subsection*{Weak KAM solutions in Lagrangian dynamics} Let $L$ be a time-independent \textit{Tonelli Lagrangian} on a compact Riemanian manifold $M$,  
and consider  $\T_t$ to be the cost minimizing transport 
$$\T_t(\mu,\nu) = \inf\{\int_{M\times M} A_t(x,y)d\pi(x,y)\,;\, \pi \in \mcal{K}(\mu,\nu)\},$$ 
where 
$$A_t(x,y) := \inf\{\int_{0}^{t}L(\gamma(s), \dot{\gamma}(s))d s\,;\, \gamma \in C^1([0,t];M); \gamma (0)=x, \gamma (t)=y\}.$$
The backward Lax-Oleinik semi-group is defined for $t>0$, via
\eqs{  
S_t^- u(x) = \sup\{ u(\gamma(t)) - \int_{0}^{t}L(\gamma(s), \dot{\gamma}(s))d s\,;\, \gamma \in C^1([0,t]; M), \gamma(0) = x\},
}
while the forward semi-group is 
\eqs{
S_t^+ u(x) := \inf\{ u(\gamma(0)) + \int_{0}^{t}L(\gamma(s), \dot{\gamma}(s))d s\,;\, \gamma \in C^1([0,t]; M), \gamma(t) = x\}.
}
\thm{\lbl{mather_fathi}
There exists a unique constant  $c \in \R$ such that the following hold:
\begin{enumerate}
\item {\rm (Fathi \cite{F})} There exist continuous backward and forward weak KAM solutions, i.e., functions $u_-, u_+: M \to \R$ such that $S_t^-u_- + c t = u_-$ and $S_t^+u_- - c t = u_-$ for each $t \geq 0$.
\item {\rm (Bernard-Buffoni \cite{BB1})} The \textit{Peierls barrier function} $A_\infty(x,y) := \liminf_{t \to \infty}A_t(x,y) -tc$ satisfies $A_\infty(x,y) = \min_{z \in \mcal{A}}\{A_\infty(x,z) + A_\infty(z,y)\}$ and the following duality holds:
\eqs{
\inf\{\int_{M\times M}A_\infty(x,y)d\pi(x,y)\,;\, \pi \in \mcal{K}(\mu,\nu)\} = \sup_{u_+, u_-}\{\int_{M}u_+d\nu - \int_{M}u_-d\mu\},
}
where the supremum ranges over all $u_+, u_- \in C(M)$ with $u_+$ (resp. $u_-$) is a forward (resp. backward) weak KAM solution, and such that 
$$\hbox{$u_+ = u_-$ on the set $\mcal{A} := \{x \in M\,;\, A_\infty(x,x) = 0\}$.}
$$ 
\item {\rm (Bernard-Buffoni \cite{BB2})} The constant $c$ satisfies
\eqs{
c = \min_{\pi} \int_{M \times M}A_1(x,y)d \pi(x,y), 
}
where the minimum is taken over all $\pi \in \mcal{P}(M\times M)$ with equal first and second marginals. The minimizing measures are all supported on $\mcal{D} := \{ (x,y) \in M \times M\,;\, A_1(x,y) + A_\infty(y,x) = c\}$.
\item {\rm  (Mather \cite{Mat})} The constant $c = \inf_{m}\int_{TM}L(x,v)d m(x,v)$ where the infimum is taken over all measures $m \in \mcal{P}(TM)$ which are invariant under the Euler-Lagrange flow (generated by $L$).
\end{enumerate}
}

As mentioned previously, the backward (resp. forward) weak KAM solutions are the images of any $f\in C(M)$ by the Kantorovich operators $T^+_\infty$ (resp. $T^-_\infty$), and are given by
\eqs{
T^-_\infty f(x) = \sup\{f(y) - A_\infty (x,y)\,;\, y \in M\}\quad \text{and}\quad T^+_\infty f(y) = \inf\{f(x) + A_\infty (x,y)\,;\, x \in M\}.
}
The cost-minimizing transport $\T_{A_\infty}$ with cost $A_\infty$, is then the idempotent backward (and forward) linear transfer associated to $T^-_\infty$ and $T_\infty^+$, which by duality we can write as
\eqs{
\inf\{\int_{M\times M}A_\infty(x,y)d\pi(x,y)\,;\, \pi \in \mcal{K}(\mu,\nu)\}=\sup\{\int_{M}T^+_\infty gd\nu - \int_{M}T^-_\infty\circ T^+_\infty g d\mu\,;\, g\in C(M)\}.
}
It can be checked this is exactly statement 2 in Theorem \ref{mather_fathi} above.  In particular, Theorem \ref{weakKAMthm_discrete_time} provides us with statements 1,2, and 3, in the above theorem. To avoid repetition, we shall prove 4) in the stochastic setting below.

We note that Fathi \cite{F} has also shown that a continuous function $u: M \to \R$ is a viscosity solution of $H(x, \nabla u(x)) = c$ if and only if it is Lipschitz and $u$ is a backward weak KAM solution (i.e. $S_t^-u + ct = u$).  

\subsection*{The Schr\"odinger semigroup}
Let $M$ be a compact Riemannian manifold and fix some reference non-negative measure $R$ on path space $\Omega=C([0,\infty], M)$. Let $(X_t)_t$ be a random process on $M$ whose law is $R$, and denote by $R_{0t}$ the joint law of the initial position $X_0$ and the position $X_t$ at time $t$, that is $R_{0t}=(X_0, X_t)_\#R$. Assume $R$ is the reversible Kolmogorov continuous Markov process associated with the generator $\frac{1}{2}(\Delta -\nabla V\cdot \nabla)$ and the initial probability measure $m=e^{-V(x)}dx$ for some function $V$. For probability measures $\mu$ and $\nu$ on $M$, define
\begin{equation}\label{schrotrans}
\T_{t}(\mu,\nu) := \inf\{ \int_{M} \mathcal{H}(r_t^x, \pi_x)d \mu(x)\,;\, \pi \in \mathcal{K}(\mu,\nu),\, d \pi(x,y) = d \mu( x) d\pi_x(y)\}
\end{equation}
where $d R_{0t}(x, y) = d m(x)d r_t^x(y)$ is the disintegration of $R_{0t}$ with respect to its initial measure $m$.

\prop{
The collection $\{\T_t\}_{t \geq 0}$ is a semigroup of backward linear transfers with Kantorovich operators $T_t f(x) := \log S_t e^f (x)$ where $(S_t)_t$ is the semi-group associated to $R$; in particular, 
\begin{equation}\label{dualitySchro}
\T_t(\mu,\nu) = \sup\left\{\int_{M} f d\nu - \int_{M}\log S_t e^f d\mu\,;\, f \in C(M)\right\}.
\end{equation} 
The corresponding idempotent backward linear transfer is $\T_\infty(\mu,\nu) = \mathcal{H}(m, \nu)$, and its effective Kantorovich map is $T_\infty f(x) := \log S_\infty e^f$, where $S_\infty g  := \int g d m$.
}
\noindent{\bf Proof:} It is easy to see that for each $t$, $T_t$ is monotone, 1-Lipschitz and convex, and also satisfies $T_t(f + c) = T_t f + c$ for any constant $c$. It follows that $\T_{t, \mu}^*(f)=\int_MT_tf\, d\mu$ for each $t$. 
The semigroup property then follows from the semigroup $(S_t)_t$ and the property that $\T_t \star \T_s$ is a backward linear transfer with Kantorovich operator $T_t \circ T_s f(x) = \log S_tS_s e^f (x) = \log S_{s+t}e^f (x) = T_{t+s} f(x)$. 

Now we remark that it is a standard property of the semigroup $(S_t)_t$ on a compact Riemannian manifold, that under suitable conditions on $V$, $S_t e^f \to S_\infty e^f$, uniformly on $M$, as $t \to \infty$, for any $f \in C(M)$. This immediately implies by definition of $T_t$, that $T_t f \to T_\infty f$ uniformly as $t \to \infty$ for any $f \in C(M)$. We then deduce from the $1$-Lipschitz property, that $T_t \circ T_\infty f(x) = T_\infty f(x)$. We conclude that $T_\infty$ is a Kantorovich operator from Theorem \ref{weakKAMthm_discrete_time}. Finally we see that the Peirls barrier is
\begin{align*}
\T_\infty(\mu,\nu):=
= \sup\{\int_M f d\nu - \log \int_M e^f d m\,;\, f \in C(M)\}
= \mathcal{H}(m, \nu).
\end{align*}
\subsection*{Stochastic mass transport}   
We restrict our setting to $M = \mathbb{T}^d := \R^d/\Z^d$, the $d$-dimensional flat torus, and the following assumptions on a Lagrangian $L$:
 \enum{
\item[(A0)] $L$ is continuous, non-negative, $L(x,0) = 0$, and $D^2_v L(x,v)$ is positive definite for all $(x,v) \in TM$ (in particular $v \mapsto L(x,v)$ is convex).
 
\item[(A1)] There exists a function $\gamma = \gamma(|v|): \R^n \to [0,\infty)$ such that $\lim_{|v| \to \infty} \frac{L(x,v)}{\gamma(v)} = +\infty$ and $\lim_{|v| \to \infty}\frac{|v|}{\gamma(v)} = 0$.
}
Let $(\Omega, \mcal{F}, \P)$ be a complete probability space with normal filtration $\{\mcal{F}_t\}_{t \geq 0}$, and define $\mcal{A}_{[0,t]}$ to be the set of continuous semi-martingales $X: \Omega \times [0,t] \to M$ such that there exists a Borel measurable drift $\beta_X: [0,t] \times C([0,t]) \to \R^d$ for which
\enum{
\item $\omega \mapsto \beta_X(s,\omega)$ is $\mcal{B}(C([0,s]))_{+}$-measurable for all $s \in [0,t]$, where $\mcal{B}(C([0,s]))$ is the Borel $\sigma$-algbera of $C[0,s]$.
\item $W_X(s) := X(s) - X(0) - \int_{0}^{s}\beta_{X}(s', X)\d s'$ is a $\sigma(X(s)\,;\, 0 \leq s \leq t)$  
$M$-valued Brownian motion. 
}
 The stochastic mass transport of a probability measure $\mu$ to another one $\nu$ on $\mathcal{P}(M)$ in time $t > 0$   is defined via the formula,
\begin{equation}\label{stoctrans}
\T_{t}(\mu,\nu) := \inf\lf\{\E \int_{0}^{t} L(X(s), \beta_X(s,X))\d s\,;\, X(0) \sim \mu, X(t) \sim \nu, X \in \mcal{A}_{[0,t]}\rt\},
\end{equation}

Stochastic transport has a dual formulation (first proven in Mikami-Thieullin \cite{MT} for the space $\R^d$) that permits it to be realised as a backward linear transfer. 
Consider the operator $T_t: C(M) \to USC(M)$ via the formula
\begin{equation}\label{stochasticop}
T_{t} f(x) := \sup_{X \in \mcal{A}_{[0,t]}} \lf\{\E \lf[f(X(t))
-\int_{0}^{t} L(X(s),\beta_X(s,X))\d s\big|\, X(0) = x\rt]\rt\}. 
\end{equation}

 An adaptation of their proofs to the case of a compact torus yields the following.  

\begin{proposition}\lbl{backwardlintrans} Under the above hypothesis on $L$, the following assertions hold:
\begin{enumerate}
\item For each $t > 0$, $\T_t$ is a backward linear transfer with Kantorovich operator $T_t$, and the family $\{\T_t\}_{t > 0}$  is a semi-group of transfers under convolutions.

\item For any $\mu,\nu \in \mcal{P}(M)$ for which $\T_t(\mu,\nu) < \infty$, there exists a minimiser $\bar{X} \in \mcal{A}_{[0,t]}$ for $\T_t(\mu,\nu)$. For every $f \in C(M)$ and $x \in M$, there exists a maximiser for  $T_tf(x)$.
\item Fix $t_1 > 0$, and $u\in C(M)$, the function $U(t,x) := T_{t_1-t} u(x)$ defined for $0 \leq t \leq t_1$  
 is the unique viscosity solution of 
\begin{equation}\label{timedep}
\frac{\partial U}{\partial t}(t,x) + \frac{1}{2}\Delta_x U(t,x) + H(x, \nabla_x U(t,x)) = 0,\quad (t,x) \in [0,t_1)\times M, 
\end{equation}
with $U(t_1,x) = u(x)$.
\item If $f \in C^\infty(M)$ and $t > 0$,  $U(t',x) := T_{t-t'} f \in C^{1,2}([0,t]\times M)$ and $U$ is a classical solution to the Hamilton-Jacobi-Bellman equation \refn{timedep}. 
The maximiser $\bar{X}$ satisfies
\eqs{
\beta_{\bar{X}}(s,\bar{X}) = D_p H(\bar{X}(s), D_x U(s,\bar{X}(s))).
}
\end{enumerate} 
\end{proposition}
 In order to define the Mather  constant $c(\T)$ and develop a corresponding Mather theory, we need to establish that there exists a probability measure $\mu \in {\cal P}(M)$ such that $\T_1(\mu, \mu)<+\infty$. Such a measure can be obtained as the first marginal of a probability measure $m$ on phase space $TM$ that is flow invariant, that is one that satisfies
\begin{equation}
\int_{TM}A^v \vphi(x)\d m(x,v) = 0 \hbox{ for all $\vphi \in C^2(M)$ where $A^v\vphi := \frac{1}{2}\Delta \vphi + v\cdot \nabla \vphi$. }
\end{equation} 
To this end, let $\mcal{P}_\gamma(M)$ be the set of probability measures on $TM$ such that $\int_{TM}\gamma(v)\d m(x,v) < +\infty,$ and denote by ${\cal N}_0$ the class of such probability measures $m$, that is,
\begin{equation*}
\mcal{N}_0 := \{ m \in \mcal{P}_\gamma(TM)\,;\, \int_{TM}A^v \vphi(x)\d m(x,v) = 0 \text{ for all }\vphi \in C^2(M)\}.
\end{equation*}
 
\begin{proposition} The set ${\cal N}_0$ of `flow-invariant' probability measures $m$ on $TM$ is non-empty and 
\begin{equation}
c := \inf\{\T_1(\mu,\mu)\,;\, \mu \in \mcal{P}(M)\}=\inf \{\int_{TM} L(x,v)\d m(x,v);\, m \in {\cal N}_0\}. 
\end{equation}
Moreover, the infimum over  ${\cal N}_0$  
is attained by a measure $\bar m$, that we call a {\it stochastic Mather measure.} Its projection  $\mu_{\bar m}$  on $\mcal{P}(M)$ is a minimiser for $\T_1$. 

Conversely, every minimizing measure $\bar \mu$ of $\T_1(\mu,\mu)$ induces a stochastic Mather measure $m_{\bar \mu}$.
\end{proposition} 
\noindent {\bf Proof:} 
Given $\mu \in \mcal{P}(M)$, consider $X \in \mcal{A}_{[0,1]}$ that realises the infimum for $\T_1(\mu,\mu)$, that is
\eqs{
\T_1(\mu,\mu) = \E \int_{0}^{1}L(X(s), \beta_X(s,X))\d s.
}
Define a probability measure $m = m_\mu \in \mcal{P}_\gamma(TM)$ via its action on the subset of continuous functions $\psi :TM \to \R$ with $\sup_{(x,v) \in TM}\lf|\frac{\psi(x,v)}{\gamma(v)}\rt| < +\infty$ and $\lim_{|(x,v)| \to \infty}\frac{\psi(x,v)}{\gamma(v)} \to 0$ via the formula
\begin{equation}\label{definemathermeasure}
\int_{TM}\psi(x,v)\d m(x,v) := \E \int_{0}^{1}\psi(X(s), \beta_X(s,X))\d s.
\end{equation}
We claim that $\int_{TM}A^v \vphi(x)\d m(x,v) = 0$ for every $\vphi \in C^2(M)$. 
Indeed, by the definition of $m$,
\as{
\int_{TM}A^v\vphi(x)\d m(x,v) &= \E \int_{0}^{1}A^{\beta_X(s,X)}\vphi(X(s))\d s\\
 &= \E \int_{0}^{1}\frac{\d }{\d s}[\vphi(X(s))]\d s\quad \text{(It\^o's lemma)}\\
 &= \E \vphi(X(1)) - \E\vphi(X(0))\\
 &= 0,\quad\quad \text{($X(0) \sim \mu \sim X(1)$).}
}
This implies that $m \in \mcal{N}_0$, so that
\begin{align}
\T_1(\mu,\mu) = \E \int_{0}^{1}L(X(s), \beta_X(s,X))\d s
 = \int_{TM}L(x,v)\d m(x,v)\lbl{equalitymathermane}
 \geq \inf_{m \in \mcal{N}_0}\int_{TM}L(x,v)\d m(x,v), 
\end{align}
hence
$\inf_{\mu \in \mcal{P}(M)}\T_1(\mu,\mu) \geq \inf_{m \in \mcal{N}_0}\int_{TM}L(x,v)\d m(x,v).$

Conversely, suppose $m \in \mcal{N}_0$, and let $\vphi(x,t)$ be a smooth solution to the Hamilton-Jacobi-Bellman equation. Since  $\int_{TM}A^v \vphi(x,t)\d m(x,v)= 0$ for every $t$, it follows that $\mu_m := \pi_M\# m$ satisfies
\as{
\int_{M}[\vphi(x,1) - \vphi(x,0)]\d \mu_m(x) &= \int_{[0,1]}\frac{\d}{\d t}\lf[\int_{M}\vphi(x,t)\d \mu_m\rt] \d t\\
&= \int_{0}^{1}\int_{TM}\partial_t \vphi(x,t)\d m(x,v)\d t\\
&= \int_{0}^{1}\int_{TM}[v\cdot \nabla \vphi(x,t) - H(x, \nabla_x \vphi(x,t))]\d m(x,v)\d t.
}
Since $H(x,p) := \sup_{v}\lf\{\langle p, v\rangle - L(x,v)\rt\}$, we have 
$v\cdot \nabla \vphi(x,t) - H(x, \nabla_x \vphi(x,t)) \leq L(x,v),$
hence combining the above two displays implies
\eqs{
\int_{M}[\vphi(x,1) - \vphi(x,0)]\d \mu_m(x) \leq \int_{TM}L(x,v)\d m(x,v)
}
for every Hamilton-Jacobi-Bellman solution $\vphi$ on $[0,1)\times M$ with $\vphi(\cdot,1) \in C^\infty(M)$. Taking the supremum over all such solutions $\vphi$ yields
\eqs{
\sup\lf\{\int_{M}[\vphi(x,1) - \vphi(x,0)]\d \mu_m(x)\,;\, \vphi(\cdot, 1) \in C^\infty(M)\rt\} \leq \int_{TM}L(x,v)\d m(x,v).
}
By duality, $\T_1(\mu_m,\mu_m) = \sup\lf\{\int_{M}[\vphi(x,1) - \vphi(x,0)]\d \mu_m(x)\,;\, \vphi(\cdot, 1) \in C^\infty(M)\rt\}$, so that 
\begin{equation}\label{optimalprojection}
\T_1(\mu_m,\mu_m) \leq \int_{TM}L(x,v)\d m(x,v)
\end{equation}
and therefore
$\inf_{\mu \in \mcal{P}(M)}\T_1(\mu,\mu) \leq \int_{TM}L(x,v)\d m(x,v),$ and we are done. \qed

\begin{proposition} Let $\{\T_t\}_{t \geq 0}$ be the family of stochastic transfers defined via \refn{stoctrans} with associated backward Kantorovich operators $\{T_t\}_{t \geq 0}$ given by \refn{stochasticop}. Let $c$ be the critical value obtained in the last proposition. Then, 
 
\begin{enumerate}
\item The equation
$T_t u + kt = u$ for $t \geq 0$, $u \in C(M)$,
has solutions (the \textit{backward weak KAM solutions}) if and only if $k = c$.
 
\item The backward weak KAM solutions are exactly the viscosity solutions of the stationary Hamilton-Jacobi-Bellman equation
\begin{equation}\label{stat-HJB}
\frac{1}{2}\Delta u + H(x, D_x u) = c.
\end{equation}
 \end{enumerate}
\end{proposition}
\begin{proof} The fact that there are solutions for (\ref{stat-HJB}) was established by Gomes \cite{Gom}. 
Let $\alpha >0$ and consider 
$$u_\alpha(x):= \inf\lf\{\E \int_{0}^{+\infty} e^{-s}L(X(s), \beta_X(s,X))\d s\,; \, X \in \mcal{A}_{[0,t]}, \,  X(0)=x\rt\}.$$
The dynamic programming principle yields that 
\[
u_\alpha(x)=\inf\lf\{\E \int_{0}^{t} e^{-s}L(X(s), \beta_X(s,X))\d s + e^{-\alpha t}u_\alpha (X(t))\,; \, X \in \mcal{A}_{[0,t]}, \,  X(0)=x\rt\}.
\]
One can then check that
$
T_t u_\alpha +t \alpha u_\alpha = u_\alpha 
$
 and use Proposition \ref{hor1} to get the result with $t=n$. Note that to construct a viscosity solution for  \refn{stat-HJB}, it suffices to find a weak KAM solution for $T_1$.

As to the relationship between 1) and 2) observe that if $u$ is a backward weak KAM solution,  
then $U(t,x) := T_{t_1-t}u(x) + c(t_1 - t)= u(x)$ is a viscosity solution to 
\begin{equation}\label{viscoinitial}
\begin{cases}
\frac{\partial U}{\partial t}(t,x) + \frac{1}{2}\Delta U(t,x) + H_c(x, \nabla U(t,x)) = 0,\quad t \in [0,1), x \in M\\
U(1,t) = f(x).
\end{cases}
\end{equation}
where $H_c(x,p) := H(x,p) + c$, with the additional property that $U(0,x) = U(1,x)$. 
where the final time is $t_1$.
Hence $u$ is a viscosity solution of \refn{stat-HJB}.

Conversely, suppose $u$ is a viscosity solution to \refn{stat-HJB}, then, $(x,t) \mapsto u(x)$ is a viscosity solution to \refn{viscoinitial}.
 On the other hand,  
$T_{t_1-t} u +c(t_1 - t)$ is also a viscosity solution of \refn{viscoinitial}. By the uniqueness of such solutions,  
it follows that $T_{t_1-t}u(x) + c(t_1-t) = u(x) $. As $t_1 > 0$ is arbitrary, this shows that $u$ is a backward weak KAM solution. \end{proof}

We finish this section with the following characterization of the Man\'e value, motivated by the work of Fathi \cite{F} in the deterministic case. Let $u \in C(M)$ and $k \in \R$, and  say that {\it $u$ is dominated by $L-k$} and write $u \prec L - k$ if for every $t > 0$, it holds for every $X \in \mcal{A}_{[0,t]}$ and every $x\in M$, 
\begin{equation}
\E [u(X(t))|X(0) = x] - u(x) \leq \E \lf[\int_{0}^{t}L(X(s), \beta_X(s,X)\d s|X(0) = x\rt] - kt.
\end{equation}
\begin{proposition}
The Ma\~n\'e critical value satisfies
\eqs{
c = \sup\lf\{k \in \R\,:\, \exists u \quad \hbox{such that}\quad  u \prec L - k\rt\}.
}
\end{proposition}

\begin{proof}
By the above,  
there exists a $u$ such that $T_t u + c t = u$, so that by definition of $T_t$,
\eqs{
u(x) - ct \geq \E [u(X(t))|X(0) = x] - \E \lf[\int_{0}^{t} L(X(s),\beta_X(s,X)\d s|X(0) = x\rt]
} 
for every $X \in \mcal{A}_{[0,t]}$. This shows that $u \prec L - c$, so $c$ is itself admissible in the supremum.

On the other hand, if $k \in \R$ is such that $u \prec L - k$, then it is easy to see that $T_tu(x) \leq u(x) - kt$ for all $t$. In particular, applying $T_s$ and using the linearity of $T_s$ with respect to constants, we find $T_{s+t}u + kt \leq T_su$, and hence
\eqs{
T_{s+t}u + k(t+s) \leq T_su + ks
}
So $t \mapsto T_tu + kt$ is decreasing and the result follows from Lemma \ref{weak_KAM_for_monotone_decreasing}.
 \end{proof}

\section{Linear Transfers and Ergodic Optimization}

This section was developed jointly with Dorian Martino \cite{Mar}. 
We shall consider here linear transfers where the associated Kantorovich maps are affine operators that is of the form $Tf(x)=Sf(x)-A(x)$, where $S$ is a Markov operator and $A$ is a given function (observable).  
For simplicity, we shall focus here on the case where the linear Markov operator $S$  is given by a point transformation $\sigma$, i.e., $Sf=f\circ \sigma$.
We consider the following setting: 

Let $X$ be a compact metric space, $\sigma:X\to X$ a continuous and onto map. Assume there is another compact metric space $X^*$ such that: 
\begin{enumerate}
\item For each $x^*\in X^*$, there exists a compact subset $X_{x^*}$ of $X$ and a continuous map $\tau_{x^*}: X_{x^*} \to X$, such that 
$\sigma\circ \tau_{x^*} (x) = x$ for all $x \in X_{x^*}$.

\item 
If $(x_k, x^*_k)$ is a sequence in $X\times X^*$ with $x_k \in X_{x^*_k}$ is such that $(x_k, x^*_k) \to (x, x^*)$, then $x\in X_{x^*}$
and $\tau_{x^*_k}(x_k) \to \tau_{x^*}(x)$. 

\end{enumerate}

For $x\in X$, we let $X^*_x=\{x^*\in X^*; x\in X_{x^*}\}$ in such a way that $x\in X_{x^*}$ if and only if $x^* \in X^*_x$. 
\prop{ \label{symbolicdynamics}
In the above setting, let $A \in C(X^* \times X)$ be a continuous function, $\bar{A}(x) := c(x,\sigma(x))$, where the cost function $c: X \times X \to \R \cup\{+\infty\}$ is defined by 
\eqs{
c(z,x) :=\begin{cases}
\inf\{A(x^*,x)\,;\, x^* \in X^*_x, \tau_{x^*}(x) = z\}  & \text{if }\sigma(z) = x\\
+\infty & \text{\rm otherwise.}
\end{cases}
}
Then, $c$ is lower semi-continuous and the optimal mass transport $\T$ associated to the cost $c$ is  \eqs{
\T(\mu,\nu) = \begin{cases}
\int_{X}\bar{A}d\mu & \text{if }\nu = \sigma_{\#}\mu\\
+\infty & \text{\rm otherwise.}
\end{cases}
}
The corresponding backward 
Kantorovich operators is $T g(x) = g(\sigma(x)) - \bar{A}(x)$. 
}

\prf{
To see that $c$ is lower semi-continuous, suppose $x_k \to x$ and $z_k \to z$. If for all but finitely many $(z_k,x_k)$ we have $c(z_k,x_k) = +\infty$, then there is nothing to prove. Therefore assume $\sigma(z_k) = x_k$, and hence
$c(z_k,x_k) = \inf\{A({x^*},x_k)\,;\, x^* \in X^*_{x_k}, 
\tau_{x^*}(x_k) = z_k\}. $

Since $A$ is continuous and $\{x^*\in X^*; x^* \in X^*_{x_k},\, \tau_{x^*}(x_k) = z_k\}$ is closed, 
the infimum is achieved by some $x^*_k$, i.e. $
c(z_k,x_k) = A(x^*_k, x_k)$. 
The sequence $\{x^*_k\} \subset X^*$ then 
has a convergent subsequence (that we relabel back to $x^*_k$). This ${x^*}$ has the property that $x \in X_{x^*}$ and $\tau_{x^*}(x) = z$ by assumption. Consequently,
\eqs{
\liminf_{k \to \infty}c(z_k, x_k) = A({x^*},x) \geq \inf\{A({x^*},x)\,;\, x^* \in X^*_{x}, \tau_{x^*}(x) = z\} = c(z,x).
}
We now evaluate the optimal transport
$
\T(\mu,\nu) = \inf_{\pi \in \mcal{K}(\mu,\nu)}\int_{X\times X}c(z,x) d\pi(z,x).
$

If $\nu \neq \sigma_{\#}\mu$, then $\pi \in \mcal{K}(\mu,\nu)$ will give non-zero mass to a region $\{(z,x); \,\sigma(z) \neq x\}$. Since $c(z,x) = +\infty$ there, we deduce that $\T(\mu,\nu) = +\infty$ in this case. Otherwise assume $\nu = \sigma_{\#}\mu$. Then $\pi \in \mcal{K}(\mu,\nu)$ is supported on $\{(z, \sigma(z))\,;\, z \in X\}$, so that
$$\T(\mu,\nu) = \inf_{\pi \in \mcal{K}(\mu,\nu)}\{\int_{X\times X}c(z,\sigma(z)) d\pi(z,x)\}
= \int_{X}c(z, \sigma(z))d\mu 
= \int_{X}\bar{A}d\mu.
$$
The backward  Kantorovich operator is therefore given by
\as{
T^- g(z) = \sup_{x \in X}\{g(x) - c(z,x)\} = g(\sigma(z)) - c(z,\sigma(z)) = g\circ \sigma(z) - \bar{A}(z).
}
}

\thm{\lbl{optimization_equalities}
In the set-up described above, 
let $\mcal{P}_\sigma(X)$ denote the subset of probability measures in $\mcal{P}(X)$ which are invariant under $\sigma$. Let $\hat{X}=\{(x^*,x)\in X^*\times X; x \in X_{x^*}\}$, and let $\mcal{M}_0(\hat{X})$ be the set of ``holonomic'' probability measures, i.e.,
 \[
 \mcal{M}_0(\hat{X})=\{\hat{\mu}\in \mcal{P}(\hat{X});\,\, \int_{\hat{X}}[f(\tau_{x^*}(x))-f(x)]d\hat{\mu} = 0\}.
 \]
If $c(\T)=\inf\limits_{\mu \in {\cal P}(X)}\T(\mu, \mu)$ is the Mather constant associated to the transfer $\T$ considered in Proposition \ref{symbolicdynamics}, then the following identities hold:
 \begin{align}
 c(\T)&
 = \inf\{\int_X \bar{A}(x)\, d\mu(x); \mu \in {\cal P}_\sigma(X)\}\\
 &=\inf\{\int_{\hat{X}}A(x^*,x)d\hat{\mu}(x^*,x)\,;\, \hat{\mu} \in \mcal{M}_0(\hat{X})\}\\
 &=\sup_{f\in C(X)}\inf_{x\in X}\{f(x)-f(\sigma (x))+\bar{A}(x)\}\\
&= \sup_{f \in C(X)}\inf_{x \in X}\inf_{x^*\in X_x^*}\{f(\tau_{x^*}(x)) - f(x) + A(x^*,x)\}.
 \end{align}
 }
 \rem{
 The equality of the second and fourth lines has already been established in \cite{GL} for the setting of symbolic dynamics (see the next subsection).
 }
 \prf{First note that the equality $c(\T) =\inf\{\int_X \bar{A}(x)\, d\mu(x); \mu \in {\cal P}_\sigma(X)\}$ follows readily from the definition of $\T$ in Proposition \ref{symbolicdynamics}. By the identification of $T^-$, we have 
\eq{
\sup_{f\in C(X)}\inf_{x\in X}\{f(x) - T^-f(x)\}= \sup_{f\in C(X)}\inf_{x\in X}\{f(x)-f(\sigma (x))+\bar{A}(x)\}.
} 
Since  
$c(\T)= \inf\limits_{ \mu \in \mcal{P}(X)}\inf\limits_{ \pi \in \mcal{K}(\mu,\mu)}\int_{X\times X}c(x,y)d\pi(x,y)
= \inf\limits_{\mu \in \mcal{P}(X)}\sup\limits_{f \in C(X)}\{\int_{X}(f-T^-f)d\mu\}, $ 
 it suffices to show that 
\eq{\lbl{min_max_symbolic}
\sup_{f\in C(X)}\inf_{x\in X}\{f(x) - T^-f(x)\} = \inf_{\mu \in \mcal{P}(X)}\sup_{f \in C(X)}\{\int_{X}(f-T^-f)d\mu\}
}
to conclude that $c(\T)=\sup\limits_{f\in C(X)}\inf\limits_{x\in X}\{f(x)-f(\sigma (x))+\bar{A}(x)\}.$

But this follows from Sion's minimax principle applied to the functional $F: \mcal{P}(X)\times C(X)$ defined by $F(\mu,f) := \int_{X}(f-T^-f)d\mu$.  
Indeed, $\mu \mapsto F(\mu,f)$ is weak$^*$ lower semi-continuous since $f - T^-f = f - f\circ \sigma + \bar{A}$ is a lower semi-continuous function for each $f \in C(X)$. Moreover $\mu \mapsto F(\mu,f)$ is quasi-convex on $\mcal{P}(X)$, i.e. $\{\mu \in \mcal{P}(X)\,;\, F(\mu,f) \leq \lambda\}$ is convex or empty for all $\lambda \in \R$. On the other hand, $f \mapsto F(\mu,f)$ is upper semi-continuous, since if $f_k \to f$ in $C(X)$, then 
$ \limsup_{k \to \infty}\int_{X}(f_k-T^-f_k)d\mu \leq \int_{X}(f - \liminf_{k \to \infty}T^-f_k)d\mu \leq \int_{X}(f-T^-f)d\mu.$
 Moreover, $f \mapsto F(\mu,f)$ is quasi-concave on $C(X)$, i.e. $\{f \in C(X)\,;\, F(\mu,f) \geq \lambda\}$ is convex or empty for all $\lambda \in \R$. Hence the minimax principle applies, and (\ref{min_max_symbolic}) therefore holds.
 
 Next, we show that
 \eq{\lbl{ergodic_3}
 \inf\{\int_{\hat{X}}A(x^*,x)d\hat{\mu}(x^*,x)\,;\, \hat{\mu} \in \mcal{M}_0(\hat{X})\} =
 \sup_{f \in C(X)}\inf_{x \in X}\inf_{x^*\in X_x^*}\{f(\tau_{x^*}(x)) - f(x) + A(x^*,x)\}.
 }
For that we apply the Fenchel-Rockafellar duality (see e.g. \cite{V} Theorem 1.9)], with   
 $h_1, h_2: C(\hat{X})\to \R \cup\{+\infty\}$ defined by 
$$
h_1(\phi) :=\sup_{(x^*,x) \in \hat{X}}\{-\phi(x^*,x)+A(x^*,x)\},
$$
and 
\eqs{
h_2(\phi) =\begin{cases}
0 & \text{if $\phi$ is in the closure of $\{g \in C(\hat{X})\,;\, g(x^*,x) = f(x) - f\circ \tau_{x^*}(x) \text{ for some }f\in C(X)\}$}\\
+\infty & \text{otherwise.}
\end{cases}
}
Note that $h_1(\phi) - h_1(\tilde{\phi}) \leq \|\phi - \tilde{\phi}\|_\infty$, so $h_1$ is continuous on $C(\hat{X})$. To compute their respective Legendre transforms, we have $
h_1^*(\hat{\mu}) = \sup_{\phi \in C(\hat{X})}\{\int_{\hat{X}}\phi d\hat{\mu} - h_1(\phi)\}$.

If $\hat{\mu}(\hat{X}) \neq -1$, then 
$h_1^*(\hat{\mu}) \geq \sup_{\lambda \in \R}\{\lambda(\hat{\mu}(\hat{X})+1) + \inf_{(x^*,x) \in \hat{X}}A(x^*,x)\}$, 
and the supremum is $+\infty$. Suppose now that $\hat{\mu}(\hat{X}) = -1$, but $-\hat{\mu} \notin \mcal{P}(\hat{X})$. Then there exists a sequence of functions $\phi_n \in C(\hat{X})$ such that $\phi_n \leq 0$ and $\lim_{n \to \infty}\int_{\hat{X}}\phi_nd\hat{\mu} = +\infty$. Then we have 
$$
h_1^*(\hat{\mu}) \geq \int_{\hat{X}}\phi_n d\hat{\mu}-h_1(\phi_n)
\geq \int_{\hat{X}}\phi_n d\hat{\mu} - \sup_{(x^*,x) \in \hat{X}}A(x^*,x),$$ 
hence $h_1^*(\hat{\mu}) = +\infty$. Finally suppose $-\hat{\mu} \in \mcal{P}(\hat{X})$. Then
$
h_1^*(\hat{\mu})\geq \int_{\hat{X}}A d\hat{\mu} - h_1(A)= \int_{\hat{X}}A d\hat{\mu},
$
while also since $\phi + h_1(\phi) \geq A$, then $\int_{\hat{X}}(\phi + h_1(\phi))d\hat{\mu} \leq \int_{\hat{X}}A d\hat{\mu}$ (recall $\hat{\mu}(\hat{X}) = -1$), so we have
\as{
h_1^*(\hat{\mu}) = \sup_{\phi \in C(\hat{X})}\{\int_{\hat{X}}(\phi(x^*,x) + h_1(\phi)) d\hat{\mu}\} \leq \int_{\hat{X}}A d\hat{\mu}.
}
Therefore,
\eqs{ 
h_1^*(\hat{\mu}) =\begin{cases}
\int_{\hat{X}}Ad \hat{\mu}\, & \text{if }-\hat{\mu} \in \mcal{P}(\hat{X})\\
+\infty & \text{otherwise.}
\end{cases}
}
We also have $
h_2^*(\hat{\mu}) = \sup_{\phi \in C(\hat{X})}\{\int_{\hat{X}}\phi d\hat{\mu} - h_2(\phi)\}= \sup_{f \in C(X)}\{\int_{\hat{X}}(f - f\circ\tau_{x^*})d\hat{\mu}\}.
$

If $\int_{\hat{X}}(f - f\circ\tau_{x^*})d\hat{\mu} \neq 0$ for some $f$, then substituting $f$ with $\lambda f$, $\lambda \in \R$, implies this supremum will be $+\infty$. Let $\mcal{S}_0 := \{\hat{\mu}\in \mcal{M}(\hat{X})\,;\, \int_{\hat{X}}(f\circ \tau_{x^*}(x) - f(x))d\hat{\mu}(x^*,x) = 0\text{ for all $f \in C(X)$}\}$. Then,
\begin{equation*}
h_2^*(\hat{\mu}) =\begin{cases}
0 & \text{if }\hat{\mu} \in \mcal{S}_0\\
+\infty & \text{otherwise.}
\end{cases}
\end{equation*}
Apply now the identity
$
\inf_{\phi \in C(\hat{X})}\{h_1(\phi) + h_2(\phi)\} = \sup_{\hat{\mu} \in \mcal{M}(\hat{X})}\{-h_1^*(-\hat{\mu}) - h_2(\hat{\mu})\}$ to obtain (\ref{ergodic_3}). 

Finally, we observe that by the definition of $\bar{A}(x) = c(x,\sigma(x)$, we have for any $f \in C(X)$,
\begin{align*}
\inf_{x\in X}\{f(x)-f(\sigma (x))+\bar{A}(x)\} &= \inf_{x \in X}\inf_{\sigma (x) \in X_{x^*} , \tau_{x^*}(\sigma(x)) = x}\{ f(x) - f(\sigma(x)) + A(x^*, \sigma(x))\}\\
&= \inf_{z \in X}\inf_{x^*\in X_z^*}\{f(\tau_{x^*}(z)) - f(z) + A(x^*,z)\},
\end{align*}
where the last equality holds by making the change of variable $z := \sigma(x)$, along with the fact that $\sigma$ is assumed to be surjective. It therefore follows that
\as{
\sup_{f\in C(X)}\inf_{x\in X}\{f(x)-f(\sigma (x))+\bar{A}(x)\}
&= \sup_{f \in C(X)}\inf_{x \in X}\inf_{x^*\in X_x^*}\{f(\tau_{x^*}(x)) - f(x) + A(x^*,x)\}
}
which concludes the proof of the theorem. 
 }
  \thm{\lbl{weak_KAM_soln_holonomic}
There exists a bounded above weak KAM solution $h$ for $T^-$ in $USC_\sigma(X)$, 
that is 
$T^- h + c(\T) = h$. Equivalently,     
 \begin{equation}\label{calibratedsoln}
 h (\sigma (x)) -\bar{A}(x) +c(\T)=h(x) \quad \text{for all $x\in X$},
 \end{equation}
and 
 \begin{equation}\label{calibratedsoln.+}
\inf_{x^*\in X_x^*}\{h(\tau_{x^*}(x)) + A(x^*,x)\} - c(\T) = h(x)\quad \text{for all $x\in X$}. 
\end{equation}
}
\prf{
It suffices to show that the assumptions of Theorem \ref{gen} are satisfied. 
For that, first note that by the Bogolyubov-Krylov theorem, 
$\sigma$ has an invariant measure, hence $c(\T) = \inf_{\mu} \T(\mu, \mu)<+\infty$. In addition, $\T$ is regular since
$$ \sup_{x\in X}\inf_{\nu\in\mathcal{P}(X)} \mathcal{T}(\delta_x,\nu) \leq \sup_{x\in X}\bar{A}(x) <+\infty. $$

In order to show that $\T$ has bounded oscillation, 
we let 
for each $n\in\mathbb{N}$, $\mu_n\in\mathcal{P}(X)$ be such that $\mathcal{T}_n(\mu_n,(\sigma^n)_\sharp\mu_n) = \inf\limits_{\mu,\nu}\mathcal{T}_n(\mu,\nu)$. Select a subsequence (which we relabel back to $n$) so that $\nu_n := (\sigma^n)_\sharp\mu_n$ converges to some  $\bar{\mu}\in\mathcal{P}(X)$. Then, the C\'esaro averages $\frac{1}{n}\sum_{k=1}^{n}\nu_k$ converge to $\bar{\mu}$. Moreover, $\bar{\mu}$ is $\sigma$-invariant. Indeed, 
$$
\sigma_\sharp\bar{\mu} = \lim_{n\to\infty} \frac{1}{n}\sum_{k=1}^n (\sigma^k)_\sharp\mu_n = \lim_{n\to\infty} \left(\frac{1}{n}\sum_{k=0}^{n-1} (\sigma^k)_\sharp\mu_n +\frac{1}{n}((\sigma^n)_\sharp\mu_n-\mu_n)\right) = \bar{\mu}.
$$
Recall that $\bar{A}(x) = c(x,\sigma(x))$ is lower semi-continuous, so 
\eqs{
\limsup_{n \to \infty} \int_X \bar{A} d\bar{\mu} -\frac{1}{n}\sum_{k=0}^{n-1}\int_X \bar{A} d((\sigma^k)_\sharp\mu_n) \leq 0.
}
Therefore, by selecting a further subsequence (which again we may relabel back to $n$), we may assume that this subsequence has the property that there exists $K>0$ such that 
\eqs{
  \int_X \bar{A} d\bar{\mu}-\frac{1}{n}\sum_{k=0}^{n-1}\int_X \bar{A} d((\sigma^k)_\sharp\mu_n) \leq \frac{K}{n}, \quad n \in \N.
}
We then obtain the desired estimate,
\as{
\limsup_{n\to\infty} \left( nc(\T) - \inf_{\mu,\nu\in\mathcal{P}(X)}\mathcal{T}_n(\mu,\nu) \right)& \leq \limsup_{n\to\infty}\lf( n\mathcal{T}(\bar{\mu},\bar{\mu})-\mathcal{T}_n(\mu_n, (\sigma^n)_\sharp\mu_n)\rt) \\
&= \limsup_{n\to\infty}\lf( n\mathcal{T}(\bar{\mu},\bar{\mu})- \sum_{k=0}^{n-1}\mathcal{T}((\sigma^{k})_\sharp\mu_n,(\sigma^{k+1})_\sharp\mu_n)\rt) \\
&= \limsup_{n\to\infty}\left(n\int_X \bar{A} d\bar{\mu}-\sum_{k=0}^{n-1}\int_X \bar{A} d((\sigma^k)_\sharp\mu_n)\right) \\
& \leq K.
}
Theorem \ref{gen} then yields the existence of $h \in USC_\sigma(X)$ such that 
$T^-h + c(\T) = h.$ In other words, 
$ h (\sigma (x)) -\bar{A}(x) +c(\T)=h(x)$ \text{for all $x\in X$}. 
Since  $\bar{A}(x) = c(x,\sigma(x)$, we have 
\begin{align*}
c(\T)=h(x)-h(\sigma (x))+\bar{A}(x)&=\inf_{x^*\in X^*_{\sigma (x)}, \tau_{x^*}(\sigma(x)) = x}\{ h(x) - h(\sigma(x)) + A(x^*, \sigma(x))\}\\
&= \inf_{x^*\in X_z^*}\{h(\tau_{x^*}(z)) - h(z) + A(x^*,z)\},
\end{align*}
where the last equality holds by making the change of variable $z := \sigma(x)$, along with the fact that $\sigma$ is assumed to be surjective. In other words, for all $x\in X$, 
\[
\inf_{x^*\in X_x^*}\{h(\tau_{x^*}(x)) + A(x^*,x)\} - c(\T) = h(x)
\]
}
\begin{remark} Note that 
\as{
T^+ f(x) = \inf_{z \in X}\{f(z) + c(z,x)\} =
\inf_{ x^* \in X^*_x}\{ f(\tau_{x^*}(x)) + c(\tau_{x^*}(x), x)\}, 
}
and
\as{
c(\tau_{x^*}(x), x) = \inf\{A(z^*,x)\,;\, z^*\in X_x^*, \, \tau_{z^*}(x) = \tau_{x^*}(x)\}. 
}
Consider now the following assumption:
\eq{
\hbox{If $z^*, x^*\in X_x^*$ and $\tau_{z^*}(x) = \tau_{x^*}(x)$, then $A(z^*,x)=A(x^*,x)$.}
}
In this case, if $x^*\in X^*_x$, then 
\as{
c(\tau_{x^*}(x), x) = \inf\{A(z^*,x)\,;\, z^*\in X_x^*, \, \tau_{z^*}(x) = \tau_{x^*}(x)\}=A(x^*,x),
}
and 
\[
T^+ f(x) =\inf\{f(\tau_{x^*}(x)) + A(x^*, x);\, x^*\in X^*_x \}.
\]
It follows that if $h$ is a weak KAM solution for $T^-$, then by (\ref{calibratedsoln.+}), it is a weak solution for $T^+$, that is $T^+ h-c(\T)=h$.

\end{remark}

\subsection*{Ergodic optimization in the deterministic holonomic setting}

We now apply the result of the previous section to the setting of symbolic dynamics. Fix $r\in\mathbb{N}$, and let $M$ be an $r\times r$ transition matrix, whose entries are either $0$ or $1$, specifying the allowable transitions. 
Denote by 
$$\Sigma=\{x\in \{1,..., r\}^\mathbb{N}\ ;\ \forall i\geq 0,\ M(x_i,x_{i+1})=1\} $$
 the set of admissible words,  its dual 
$$\Sigma^*=\{y\in \{1,..., r\}^\mathbb{N}\ ;\ \forall i\geq 0,\ M(y_{i+1},y_{i})=1\},$$ 
and consider the space
$$\hat{\Sigma}=\{(y,x)\in\Sigma^*\times\Sigma \ ;\ M(y_0,x_0)=1\}.$$
For each $x\in\Sigma$, we let $\Sigma_x^*=\{y\in\Sigma^*\ ;\ (y,x)\in\hat{\Sigma}\}$ and assume that 
$\forall x,\ \Sigma_x^*\neq\emptyset.$
We will denote the words of $\Sigma$ with their starting letters, i.e., $(x_0,x_1,...)$ while the words in $\Sigma^*$ will be identified with their ending letters, i.e.,  $(...,y_1,y_0)$. We consider $\Sigma$ and $\Sigma^*$ as metric spaces with the distance $d(x,\bar{x}) = 2^{-\min\{j\in\mathbb{N} \ ;\ x_j\neq\bar{x}_j\}}$. In particular, all these sets are compact.\\
Consider now the two continuous maps $\sigma:\Sigma\to\Sigma$ and $\tau:\hat{\Sigma}\to\Sigma$ defined as 
\eq{\lbl{sigma_and_tau}
\hbox{$\sigma(x_0,x_1,...)=(x_1,x_2,...)$ \quad and \quad  $\tau(y,x)=(y_0,x_0,x_1,...)$.}
}
The map $\sigma$ can be considered as the time-evolution operator.
We will denote $\tau(y,x)$ by $\tau_y(x)$ and 
consider the set of holonomic probability measures
$$ \mathcal{M}_0(\hat{\Sigma}) := \left\{ \mu\in\mathcal{P}(\hat{\Sigma})\ ;\ \int_{\hat{\Sigma}} f(\tau_y(x))-f(x)\ d\mu(y,x) = 0 \right\}. $$
 E. Garibaldi and A. O. Lopes studied an Aubry-Mather theory for symbolic dynamics \cite{GL}; in particular, they prove the following results.
 \thm{[Garibali and Lopes \cite{GL}, Theorem 1, Theorem 4]\lbl{Gari_Lop} Under the set-up described above, given $A \in C(\hat{\Sigma})$, define $\beta_{A} :=  \max_{\hat{\mu}\in \mcal{M}_0(\hat{\Sigma})}\int_{\hat{\Sigma}}A(y,x)d \hat{\mu}(y,x)$. Then
 \eqs{
 \beta_A = \inf_{f \in C(\Sigma)}\max_{(y,x) \in \hat{\Sigma}}\{A(y,x) + f(x) - f(\tau_y(x))\}.
 }
If $A \in C^{0,\theta}(\hat{\Sigma})$ is $\theta$-H\"older continuous, then there exists a function $u \in C^{0,\theta}(\Sigma)$ such that
 \eqs{
u(x) = \min_{u \in \Sigma_{x}^*}\{u(\tau_y(x)) - A(y,x) + \beta_A\}.
 }
 } 
By applying our results in the previous subsection in this symbolic dynamics setting, we obtain the following under a much weaker assumption on the potential $A$. 

\prop{Under the set-up described at the beginning of this section, and using the notation $\beta_A$ of Theorem \ref{Gari_Lop} for a given potential $A\in C(\hat\Sigma)$, the following then hold:
\al{\label{holoduality}
\beta_A &= \sup_{\mu \in \mcal{P}_\sigma(\Sigma)}\{\int_{\Sigma}\bar{A} d\mu\}\\
 &=\inf_{f\in C(\Sigma)}\sup_{x\in \Sigma}\{f(\sigma (x)) - f(x) +\bar{A}(x)\}\\
&= \inf_{f \in C(\Sigma)}\sup_{(y,x) \in \hat{\Sigma}}\{ f(x)-f(\tau_y(x)) + A(y,x)\},
}
where $\bar{A}(x) := \sup\{A(y,x)\,;\, y \in \Sigma^*_x, \tau_y(\sigma(x)) = x\}$. Moreover, there exists $h\in USC_\sigma(\Sigma)$   such that  
\begin{equation} \inf_{y\in\Sigma_x^*} \{h(\tau_y(x))-A(y,x) + \beta(A)\} = h(x) \quad  \forall x\in\Sigma. 
 \end{equation}
}
\prf{Apply Proposition \ref{symbolicdynamics}, Theorem \ref{optimization_equalities}, and Theorem \ref{weak_KAM_soln_holonomic}, with $-A$ instead of $A$, and with the following identifications: 
\eqs{
X := \Sigma,\quad Y := \Sigma^*,\quad Y_x := \Sigma^*_{x}, \quad\hat{X}:= \hat{\Sigma}, 
}
and $\sigma$, $\tau_y$, as in \refn{sigma_and_tau}. 
}

\subsection*{Ergodic optimization in the stochastic holonomic setting} 

We now propose the following model: 
Consider inductively, the following sequence: Let $X^0\in\Sigma$ be a random word, $B^0\in\Sigma_{X^0}^*$ a random ``noise'', and let $\bar{X}^{0} :=\tau_{B^0}(X^0)$. We then let  $Y^{0}\in\Sigma_{\bar{X}^{0}}^*$ be a random ``control'', and consider $X^{1}:=\tau_{Y^{0}}(\bar{X}^{0})$. We then choose $B^1 \in \Sigma_{X^1}^*$ and let $Y^1 \in \Sigma_{\bar{X}^1}^*$.

Iterating this process, an entire random past trajectory of $X^0$ is represented via the random family $(X^{n})_n\in\Sigma^{\mathbb{N}}$. The goal is to minimise the long time average cost, 
$$\lim_{n\to\infty}\frac{1}{n}\mathbb{E}[\sum_{i=0}^{n-1}A(Y^i,\bar{X}^i)]$$ 
among all possible such choices.

We assume that $B^0$, $Y^0$, satisfy the following ``martingale-type" property: 
\begin{equation}\label{mtgassumption}
\E[f(Y^0, \tau_{B^0}(X^0))|X^0 = \sigma(x)] = \E[f(Y^0, x)] \quad \hbox{for any $f \in C(\hat{\Sigma})$.}
\end{equation}

For each such trajectory, consider for each $n\in\mathbb{N}$,  the measure $\mu_n\in\mathcal{P}(\hat{\Sigma})$ defined via
\eqs{
\int_{\hat{\Sigma}} \phi d\mu_n := \frac{1}{n}\sum_{i=0}^{n-1}\mathbb{E}[\phi(Y^i,\bar{X}^i)], \quad  \phi\in C(\hat{\Sigma}).
}
From $(\mu_n)_n$, one can extract a subsequence converging to some measure $\mu^{(X^i)_i}$ by compactness of $\mcal{P}(\hat{\Sigma})$. We denote 
$$ 
\mathcal{M}_0 = \overline{\{\mu^{(X^i)_i}\}}\subset \mathcal{M}(\hat{\Sigma}),
$$
as the closure of all such $\mu^{(X^i)_i}$. For $f\in C(\Sigma)$ and $(y,x)\in\hat{\Sigma}$, denote 
$$
\frac{1}{2}D^yf(x) := f(\tau_y(x))-f(x)-\frac{f(\tau_y(x))-2f(x)+f(\sigma(x))}{2} = \frac{f(\tau_y(x))-f(\sigma(x))}{2}.
$$
Note that the assumption made on the random noise $B^0$ yields an It\^o-type formula:  For all $f \in C(\Sigma)$, $x\in \Sigma$, with $B^0 \in \Sigma^*_{\sigma(x)}$ $Y^0 \in \Sigma^*_{\tau_{B^0}(x)}$,
$$
\E[f(\tau_{Y^0}(\tau_{B^0}(\sigma(x)))) - f(\sigma(x))] =  \E [D^{Y^0}f(x)].
$$
Let also
$$\mathcal{N}_0 = \{\mu\in\mathcal{M}(\hat{\Sigma})\ |\ \forall f\in C(\Sigma),\ \int_{\hat{\Sigma}} D^yf(x)d\mu(y,x)=0\},$$ 
which is closed in $\mathcal{M}(\hat{\Sigma})$ as a kernel of a continuous linear map. 

\lma{ We have $\mathcal{M}_0\subset \mathcal{N}_0$. 
}

\begin{proof} For each measure $\mu^{(X^i)_i}\in\mathcal{M}_0$,
\begin{align*}
\int_{\hat{\Sigma}}D^y f(x)d\mu_n(y,x) = \frac{1}{n}\mathbb{E}\left[\sum_{i=0}^{n-1} f(X^{i+1})-f(X^i)\right]
= \frac{1}{n}\mathbb{E}\left[f(X^n)-f(X^0)\right] 
\end{align*}
which tends to zero as $n \to \infty$ since $f$ is bounded.
\end{proof}

\thm{ \lbl{stochastic_symbolic_dynamics}With the above notation, we have the following 
\al{
\inf_{\mu\in\mathcal{N}_0\cap\mathcal{P}(\hat{\Sigma})} \int_{\hat{\Sigma}} Ad\hat{\mu} = \sup_{f\in C(\Sigma)}\inf_{(y,x)\in\hat{\Sigma}} D^yf(x)+A(y,x), \label{stochasticduality}
}
and there exists a $g\in USC_\sigma (\Sigma)$ such that $h:= -g$ satisfies
\begin{equation}\label{HJ-discrete}
\inf_{\mu\in\mathcal{N}_0\cap\mathcal{P}(\hat{\Sigma})}\int_{\hat{\Sigma}} Ad\hat{\mu}= \inf_{y\in\Sigma_x^*} \{D^yh(x)+A(y,x)\}.
\end{equation}
}
\begin{proof}
The last equality in \refn{stochasticduality}, i.e., 
is again an application of the Fenchel-Rockafellar duality, similar to the previous section. Indeed, consider the functions $h_1, h_2: C(\hat{\Sigma}) \to \R\cup\{+\infty\}$ defined by
$
h_1(\phi) := \sup_{(y,x)\in\hat{\Sigma}}\{-\phi(y,x)+A(y,x)\}
$
and
\begin{equation*}
h_2(\phi) := \begin{cases}
0 & \text{if }\phi\in\overline{\{g \in C(\hat{\Sigma})\,;\, g(y,x) = D^yf(x) \text{ for some }f \in C(\Sigma)\}}\\
+\infty & \text{otherwise.}
\end{cases}
\end{equation*}
Note that $h_1$ and $h_2$ are convex, with $|h_1(\phi) - h_1(\tilde{\phi})| \leq \|\phi - \tilde{\phi}\|_\infty$. To compute their Legendre transform, we have \as{
h_1^*(\hat{\mu}) = \sup_{\phi\in C(\hat{\Sigma})} \{\int_{\hat{\Sigma}}\phi d\hat{\mu}-h_1(\phi)\}\geq \sup_{\lambda \in \R}\{(\lambda +1)\hat{\mu}(\hat{\Sigma}) + \inf_{(y,x)\in\hat{\Sigma}}\{A(y,x)\},
}
and therefore if $\hat{\mu}(\hat{\Sigma}) \neq -1$, the supremum is $+\infty$. Suppose now that $\hat{\mu}(\hat{\Sigma}) = -1$, but $-\hat{\mu} \notin \mcal{P}(\hat{\Sigma})$. Then, there exists functions $\phi_n \in C(\hat{\Sigma})$ such that $\phi_n \leq 0$ and $\lim_{n \to \infty}\int_{\hat{\Sigma}}\phi_nd\hat{\mu} = +\infty$. Hence
\as{
h_1^*(\hat{\mu}) \geq \int_{\hat{\Sigma}}\phi_n d\hat{\mu}-h_1(\phi_n)\geq \int_{\hat{\Sigma}}\phi_n d\hat{\mu} - \sup_{(y,x) \in \hat{\Sigma}}A(y,x)
}
and $h_1^*(\hat{\mu}) = +\infty$. Finally suppose $-\hat{\mu} \in \mcal{P}(\hat{\Sigma})$. Then we have
$
h_1^*(\phi) \geq \int_{\hat{\Sigma}}A d\hat{\mu}-h_1(A) = \int_{\hat{\Sigma}}A d\hat{\mu},
$
while on the other hand from $\phi + h_1(\phi) \geq A$, we have  $\int_{\hat{\Sigma}}(\phi -h_1(\phi)) d\hat{\mu} \leq  \int_{\hat{\Sigma}}A d\hat{\mu}$ (recall $\hat{\mu}(\hat{\Sigma}) = -1$), so that
\as{
h_1^*(\phi) = \sup_{\phi\in C(\hat{\Sigma})} \{\int_{\hat{\Sigma}}(\phi -h_1(\phi)) d\hat{\mu}\}\leq \sup_{\phi\in C(\hat{\Sigma})} \{\int_{\hat{\Sigma}}(\phi - \phi - A) d\hat{\mu}\} = \int_{\hat{\Sigma}}A d\hat{\mu}.
}
Consequently,
$$
h_1^*(\hat{\mu}) = \begin{cases}
\int_{\hat{\Sigma}} A d\hat{\mu} & \text{if }\hat{\mu}\in\mathcal{P}(\hat{\Sigma}),\\
+\infty & \text{otherwise.}
\end{cases}
$$
Similarly, we have
$$
h_2^*(\hat{\mu}) = \begin{cases}
0 & \text{if }\hat{\mu}\in\mathcal{N}_0\\
+\infty & \text{otherwise.}
\end{cases}
$$
Indeed, if $\hat{\mu}\notin \mathcal{N}_0$, there exists $\bar{f}\in C(\Sigma)$ such that $\int_{\hat{\Sigma}} D^y\bar{f}(x)d\hat{\mu}(y,x) \neq 0$. Hence, we replacing $\bar{f}$ with $\lambda\bar{f}$, $\lambda \in \R$, we see that
$
h_2^*(\hat{\mu}) = \sup_{f\in C(\Sigma)} \int_{\hat{\Sigma}} D^yf(x) d\hat{\mu}(y,x)  = +\infty.
$

If $\hat{\mu}\in\mathcal{N}_0$, by definition, $h_2^*(\hat{\mu})=0$, then by Fenchel-Rockafellar, 
\eqs{
\inf_{\phi \in C(\hat{\Sigma})}\{ h_1(\phi)+ h_2(\phi)\} = \sup_{\hat{\mu} \in \mcal{M}(\hat{\Sigma})}\{ -h_1^*(-\hat{\mu})-h_2^*(\hat{\mu})\},
}
which completes the proof of \refn{stochasticduality}.

To prove (\ref{HJ-discrete}) consider the functional $\T: \mcal{P}(\Sigma) \times \mcal{P}(\Sigma) \to \R\cup\{+\infty\}$ defined via
$$ \mathcal{T}(\mu,\nu) := \inf\left\{ \mathbb{E}[A(Y^0,\bar{X}^0)]\ \left|\ \begin{array}{l l}
X^0\sim\nu & \\
\bar{X}^0=\tau_{B^0}(X^0) & B^0\in\Sigma_{X_0}^*\\
X^1 = \tau_{Y^0}(\bar{X}^0)\sim\mu & Y^0\in\Sigma_{\bar{X}^0}^* 
\end{array} \right.\right\}
$$
We claim that $\T$ is a forward linear transfer with
$$ 
T^+f(x) := \inf\{\E[f(\tau_{Y^0}(\tau_{B^0}(x))+ A(Y^0,\tau_{B^0}(x))]\,;\,Y^0 \in \Sigma^*_{\tau_{B^0}(x)}\}.
$$
 Indeed, we have
\as{
\int_{\Sigma}T^+f d\nu &= \int_{\Sigma}\lf(\inf\{\E[f(\tau_{Y^0}(\tau_{B^0}(x))+ A(Y^0,\tau_{B^0}(x))]\,;\,Y^0 \in \Sigma^*_{\tau_{B^0}(x)}\}\rt)d\nu\\
&= \inf\{\E[f(\tau_{Y^0}(\tau_{B^0}(X^0))+ A(Y^0,\tau_{B^0}(X^0))]\,;\,Y^0 \in \Sigma^*_{\tau_{B^0}(X^0)},\, X^0 \sim \nu\}.
}
Note then that $\sup_{f \in C(\Sigma)}\{\int_{\Sigma}T^+f d\nu - \int_{\Sigma}f d\mu\}$ will be $+\infty$, unless $\tau_{Y^0}(\tau_{B^0}(X^0)) \sim \mu$, in which case the terms in $f$ cancel, leaving 
\as{
\sup_{f \in C(\Sigma)}\{\int_{\Sigma}T^+f d\nu - \int_{\Sigma}f d\mu\} &= \inf\{\E[A(Y^0,\tau_{B^0}(X^0))]\,;\,Y^0 \in \Sigma^*_{\tau_{B^0}(X^0)},\, X^0 \sim \nu,\, \tau_{Y^0}(\tau_{B^0}(X^0)) \sim \mu\}\\
&= \T(\mu,\nu).
}
We now show that the hypotheses of Theorem \ref{gen} to the (backward) linear transfer $\tilde{\T}(\mu,\nu) := \T(\nu,\mu)$, are satisfied.

First, $\tilde \T$ is regular since
or a fixed $x \in \Sigma$, take any random noise $B^0 \in \Sigma^*_{x}$ and random strategy $Y^0 \in \Sigma^*_{\tau_{B^0}(x)}$, and denote the law of $\tau_{Y^0}(\tau_{B^0}(x))$ by $\bar{\nu}_x$. Then
\eqs{
\sup_{x\in\Sigma}\inf_{\nu\in\mathcal{P}(\Sigma)} \tilde{\T}(\delta_x, \nu) \leq \sup_{x\in\Sigma}\T(\bar{\nu}_x, \delta_x) \leq \sup_{x\in\Sigma}\E [A(Y^0, \tau_{B^0}(x))] \leq  \sup_{\hat{\Sigma}} A <+\infty.
}
That $c(\tilde{\T})< +\infty$ and $\tilde \T$ is of bounded oscillation follow similarly as in the proof of Theorem \ref{weak_KAM_soln_holonomic}. 

Note that
\as{
c(\T)  = \inf_{\mu}\T(\mu,\mu) &= \lim_{n\to\infty}\frac{1}{n}\inf_{\mu,\nu}\mathcal{T}_n(\mu,\nu) \\
&=  \lim_{n \to \infty}\inf_{\mu,\nu} \frac{1}{n}\sum_{i = 0}^{n-1}\inf\{ \E [A(Y^i, \bar{X}^i)]\}\\
 &=  \lim_{n\to\infty}\inf_{\mu,\nu}\inf_{X^0\sim\mu,X^1\sim\nu} \int_{\hat{\Sigma}}  Ad\mu_n^{(Y^i)_i} \\
 &=  \inf_{\hat{\mu}\in\mathcal{M}_0\cap\mathcal{P}(\hat{\Sigma})} \int_{\hat{\Sigma}} Ad\hat{\mu}.
}
Therefore, we have satisfied the hypotheses of Theorem \ref{gen}, and conclude the existence of a $h \in USC_\sigma(\Sigma)$ such that 
$
\tilde{T}^- h(x) + c(\tilde{\T}) = h(x), \forall x \in \Sigma,
$
which since $\tilde{T}^-h = -T^+(-h)$ and $c(\tilde{\T}) = c(\T)$, implies
$ 
T^+(-h) - c(\T) = -h,
$ 
  or with $g := -h$, 
\begin{equation}\label{stochasticsymbolic}
T^+ g(x) - c(\T) = g(x).
\end{equation}
Replacing $x$ with $\sigma(x)$ for $x \in \Sigma$ in equation (\ref{stochasticsymbolic}), we have
$
T^+ g(\sigma(x)) - g(\sigma(x)) = c(\T). 
$

Recalling now the definition of $T^+$ and the martingale assumption (\ref{mtgassumption}), we can write
\begin{align}
c(\T) = T^+ g(\sigma(x)) - g(\sigma(x)) &= \inf_{Y^0}\{\E[g(\tau_{Y^0}(\tau_{B^0}(\sigma(x))) + A(Y^0, \tau_{B^0}(\sigma(x)))]\} - g(\sigma(x))\no\\
&= \inf_{Y^0}\{ \E [g(\tau_{Y^0}(x)) + A(Y^0, x)]\} - g(\sigma(x))\no\\
&= \inf_{y \in \Sigma^*_{x}}\{ g(\tau_y(x)) - g(\sigma(x)) + A(y,x)\}\no\\
&= \inf_{y \in \Sigma^*_{x}}\{ D^y g(x) + A(y,x)\}.\lbl{stochastic_KAM}
\end{align}
Finally, the corresponding Ma\~n\'e constant is given by 
\as{
c(\T)  = \inf_{\mu}\T(\mu,\mu) &= \lim_{n\to\infty}\frac{1}{n}\inf_{\mu,\nu}\mathcal{T}_n(\mu,\nu) \\
&=  \lim_{n \to \infty}\inf_{\mu,\nu} \frac{1}{n}\sum_{i = 0}^{n-1}\inf\{ \E [A(Y^i, \bar{X}^i)]\}\\
 &=  \lim_{n\to\infty}\inf_{\mu,\nu}\inf_{X^0\sim\mu,X^1\sim\nu} \int_{\hat{\Sigma}}  Ad\mu_n^{(Y^i)_i} \\
 &=  \inf_{\hat{\mu}\in\mathcal{M}_0\cap\mathcal{P}(\hat{\Sigma})} \int_{\hat{\Sigma}} Ad\hat{\mu}.
}
Since $\mcal{M}_0 \subset \mcal{N}_0$, it follows that $c(\T) \geq \inf_{\hat{\mu}\in\mathcal{N}_0\cap\mathcal{P}(\Sigma)} \int_{\hat{\Sigma}} Ad\hat{\mu}$.
In view of the duality (\ref{stochasticduality}) together with \refn{stochastic_KAM}, this implies that this inequality is actually an equality, hence (\ref{HJ-discrete}) holds and concludes the proof.
\end{proof}

\end{document}